\documentclass[11pt,twoside, a4paper, english, reqno]{amsart}
\usepackage[dvips]{epsfig}
\usepackage{amscd}
\usepackage{amssymb}
\usepackage{amsthm}
\usepackage{amsmath}
\usepackage{amsfonts}
\usepackage{array}
\usepackage{enumitem}
\usepackage{latexsym}

\usepackage{upref}
\usepackage{hyperref}
\usepackage{color}
\usepackage{mathrsfs}
\usepackage{subcaption}
\usepackage[usenames,dvipsnames]{xcolor}
\theoremstyle{plain}
\newtheorem{thm}{Theorem}[section]
\theoremstyle{plain}

\newtheorem{alg}[thm]{Algorithm}
\newtheorem{prob}[thm]{Problem}

\theoremstyle{definition}

\newtheorem{rem}{Remark}[section]

\newtheorem*{maintheorem*}{Main Theorem}

\newcommand{\R}{\ensuremath{\mathbb{R}}}

\newcommand{\goto}{\ensuremath{\rightarrow}}

\definecolor{dred}{RGB}{180,90,90}
\definecolor{dgreen}{RGB}{70,140,70}
\definecolor{dblue}{RGB}{100,100,180}
\definecolor{boxcolour}{RGB}{220,220,220}

\newcommand{\green}[1]{{#1}}

\newcommand{\fgreen}[1]{}
\newcommand{\fblue}[1]{}
\newcommand{\fred}[1]{}

\makeatletter
\newcommand{\opnorm}{\@ifstar\@opnorms\@opnorm}
\newcommand{\@opnorms}[1]{%
  \left|\mkern-1.5mu\left|\mkern-1.5mu\left|
   #1
  \right|\mkern-1.5mu\right|\mkern-1.5mu\right|
}
\newcommand{\@opnorm}[2][]{%
  \mathopen{#1|\mkern-1.5mu#1|\mkern-1.5mu#1|}
  #2
  \mathclose{#1|\mkern-1.5mu#1|\mkern-1.5mu#1|}
}
\makeatother

\numberwithin{equation}{section} \allowdisplaybreaks

\title[Dynamic Programming for Finite Ensembles of Nanomagnetic Particles]
{Dynamic Programming for Finite Ensembles of Nanomagnetic Particles}

\date{}



 \author[Max Jensen]{Max Jensen}
\address[Max Jensen]{\newline
Department of Mathematics,
University of Sussex,
Pevensey 2 Building, Falmer Campus,
Brighton BN1 9QH, United Kingdom. }
\email[]{m.jensen@sussex.ac.uk}
 
\author[Ananta K. Majee]{Ananta K. Majee}
\address[Ananta K. Majee]{\newline
Mathematisches Institut,
Universit\"{a}t T\"{u}bingen,
Auf der Morgenstelle 10,
D-72076 T\"{u}bingen, Germany. }
\email[]{majee@na.uni-tuebingen.de}

\author[Andreas Prohl]{ Andreas Prohl}
\address[Andreas Prohl]{\newline
Mathematisches Institut
Universit\"{a}t T\"{u}bingen
Auf der Morgenstelle 10
D-72076 T\"{u}bingen, Germany}
\email[]{prohl@na.uni-tuebingen.de}

\author[Christian Schellnegger]{Christian Schellnegger}
\address[Christian Schellnegger]{\newline
Mathematisches Institut
Universit\"{a}t T\"{u}bingen
Auf der Morgenstelle 10
D-72076 T\"{u}bingen, Germany}
\email[]{schellnegger@na.uni-tuebingen.de}

\thanks{}

\begin{document}
\begin{abstract}
We use optimal control via a distributed exterior field to steer the dynamics of an ensemble of $N$ interacting ferromagnetic particles which are immersed into a heat bath by minimizing a quadratic functional. By using dynamic programing principle, we show the existence of a unique strong solution of the optimal control problem. By the Hopf-Cole transformation, the related
Hamilton-Jacobi-Bellman equation from dynamic programming principle may be re-cast into a linear PDE on the manifold ${\mathcal M} = ({\mathbb S}^2)^N$, whose classical solution may be represented
via Feynman-Kac formula. We use this probabilistic representation for Monte-Carlo simulations to illustrate optimal switching dynamics.
\end{abstract}
\maketitle

\section{Introduction}\label{sec:intro}
We control a ferromagnetic $N$-spin system which is exposed to thermal fluctuations via an exterior forcing ${\bf u} := (u_1, \ldots, u_N): [0,T] \times \Omega \rightarrow ({\mathbb R}^3)^N$. A 
relevant application includes data storage devices, for which it is crucial to control the dynamics of the magnetization ${\bf m} := (m_1, \ldots, m_N): [0,T] \times \Omega \rightarrow ({\mathbb S}^2)^N$ in order to 
ensure
a reliable transport of data which are represented by those magnetic structures.
Besides the control $H_{\tt ext,i}(u_i) = C_{\tt ext}\,u_i$, the $i$-th spin of the ensemble with magnetization $m_i$ is exposed to different forces $H_{\tt ani,i}(m_i)$, $H_{\tt d,i}(m_i)$, 
and $H_{\tt exch,i}({\bf m})$: 

\begin{itemize}
\item the anisotropic force $H_{\tt ani,i}(m_i) = {\tt A} m_i$ for some $ {\tt A}\in {\mathbb R}^{3 \times 3}_{\tt diag}$, which favors alignment of $m_i$ with the given material
dependent easy axis $e \in {\mathbb R}^3$,
\item the `stray-field force' $H_{\tt d,i}(m_i) = - { \tt B}_i m_i$ for some ${\tt B}_i \in {\mathbb R}^{3 \times 3}_{\tt diag}$,
\item the exchange force $H_{\tt exch,i} ({\bf m})$, which penalizes non-alignment of neighboring magnetizations via $H_{\tt exch,i}({\bf m})= -({\bf J}{\bf m})_i$, for some positive semi-definite ${\bf J} \in {\mathbb R}^{3N \times 3N}_{\tt sym}$.
\end{itemize}
For every $1 \leq i \leq N$, we denote their superposition 
\begin{align}
H_{\tt eff,i}({\bf m}, {\bf u}) = H_{\tt ani,i}(m_i) + H_{\tt d,i}(m_i) + H_{\tt exch,i}({\bf m}) + H_{\tt ext,i}(u_i) \label{eq:effective-field}
\end{align}
 as the effective field. The dynamics of ${\bf m}$ for times $[0,T]$ is then governed by the following SDE system ($1 \leq i \leq N$):
\begin{equation}\label{sde-1}
\begin{aligned}
{\rm d} m_i(t) &=\Bigl( m_i \times H_{\tt eff,i} ({\bf m}, {\bf u}) - \alpha \, m_i \times
\bigl[ m_i \times H_{\tt eff,i}({\bf m}, {\bf u})\bigr]\Bigr){\rm d}t + \nu \, m_i  \times \circ\, {\rm d}W_i(t)\,, \\ 
m_i(0) &= \bar{m}_{i} \in {\mathbb S}^2\, .
\end{aligned}
\end{equation}
Here, ${\bf W} = (W_1, \ldots,W_N)$ is a $({\mathbb R}^3)^N$-valued Wiener process on a filtered probability space $(\Omega, {\mathcal F}, \{ {\mathcal F}_t \}_t, {\mathbb P})$ satisfying the usual conditions to represent
thermal fluctuations from the surrounding heat bath. The leading term in the drift part in \eqref{sde-1} causes a precessional motion of $m_i$ around $H_{\tt eff,i} ({\bf m}, {\bf u})$, while the 
dissipative second term scaled by $\alpha >0$ favors a time-asymptotic alignment of $m_i$ with $H_{\tt eff,i} ({\bf m}, {\bf u})$. The Stratonovich type of the stochastic integral of the diffusion term ensures
that each state process $m_i$ takes values in ${\mathbb  S}^2$; see e.g. \cite{banas_brzezniak_neklyudov_prohl} for further details.
\vspace{.2cm}

Our aim is to find a control process ${\bf u}^*$ such that the solution process ${\bf m}^*$ from (\ref{sde-1}) approximates a given deterministic
profile $\widetilde{\bf m} \equiv (\widetilde{m}_1, \ldots, \widetilde{m}_N) \in L^2\bigl( 0,T; ({\mathbb S}^2)^N \bigr)$. More precisely, we aim to solve the following problem.

\begin{prob}\label{problem-1}
Let the parameters $\delta, \nu, \alpha \geq 0$, and  $ T, \lambda >0$, $N \in {\mathbb N}$ as well as $h \in C^2\bigl( ({\mathbb S}^2)^N; {\mathbb R}\bigr)$ be given.  Let $\big(\Omega, \mathbb{P}, \mathcal{F},\{\mathcal{F}_t\}_{0 \le t \le T}\big)$ be a given stochastic basis with the usual conditions,
and ${\bf W}$ be a $\{\mathcal{F}_t\}_{0 \le t \le T}$-adapted $(\R^3)^N$-valued Wiener process on it. Find a pair\footnote[1]{$L^2_{\{ {\mathcal F}_t\}}\Big( \Omega; C\bigl( [0,T]; ({\mathbb S}^2)^N\bigr)\Big):=\Big\{ {\bf m}\in L^2_{\{ {\mathcal F}_t\}} \Big( \Omega; C\bigl( [0,T]; (\R^3)^N\bigr)\Big):\,\bf{m}(t)\in (\mathbb{S}^2)^N$, $\mathbb{P}$-a.s.\,for all $t\in [0,T]\Big\}$.} 
 \begin{align*}
 ({\bf m}^*, {\bf u}^*) \in L^2_{\{ {\mathcal F}_t\}} \Bigl( \Omega; C\bigl( [0,T]; ({\mathbb S}^2)^N\bigr) \times L^2(0,T; (\R^3)^N)\Bigr)
 \end{align*}
which minimizes the cost functional
\begin{equation*}
{\mathcal J}_{\tt sto}({\bf m}, {\bf u}) := {\mathbb E}\Bigl[ \int_0^T
\Bigl( \delta\, \Vert {\bf m}(s) - \widetilde{\bf m}(s)\Vert^2_{({\mathbb R}^3)^N}
+ \frac{\lambda}{2}\, \Vert {\bf u}(s)\Vert^2_{({\mathbb R}^3)^N}\Bigr)\, {\rm d}s +  h\bigl( {\bf m}(T)\bigr)\Bigr]
\end{equation*}
subject to \eqref{sde-1}. We call such a minimiser a strong solution of the optimally controlled Landau-Lifschitz-Gilbert equation. 
\end{prob}
In \cite{dunst_prohl}, some of the present authors have studied Problem~\ref{problem-1} in its weak form and constructed a weak optimal solution
$\Pi^* := \bigl(\Omega^*, {\mathbb P}^*, {\mathcal F}^*, \{{\mathcal F}^*_t\}_{0\le t\le T}, {\bf m}^*, {\bf u}^*, {\bf W}^*\bigr)$ of the underlying problem via Young measure~(relaxed control) approach for a
compact set $\mathbb{U} \subset (\R^3)^N$, a control space, such that ${\bf 0}\in \mathbb{U}$, which may be generalized to the case $\mathbb{U}=(\R^3)^N$
thanks to the coercivity of the cost functional with respect to ${\bf u}$;                                                                                                                      see also \cite{dunst_majee_prohl_vallet} for an 
extension to infinite spin ensembles. To approximate it numerically,
    implementable strategies may be developed that rest on Pontryagin's
     maximum principle which characterizes minimizers. In \cite{dunst_prohl}, a stochastic gradient method is proposed to generate a
      sequence of functional-decreasing approximate feedback controls, where
       the update requires to  solve a coupled forward-backward SDE system.
       A relevant part here is to simulate a (time-discrete) backward SDE via
 the least-squares Monte-Carlo method, which requires significant 
 data storage resources \cite{dunst_majee_prohl_vallet,dunst_prohl}, and thus limits  the complexity of practically approachable Problems~\ref{problem-1}.
 \vspace{.1cm}
 
In this work, we use an alternative strategy which rests on the dynamic programming principle. This allows us to prove the existence of a unique {\em strong solution} of Problem~\ref{problem-1}, which sharpens results of \cite{dunst_prohl}. Since the solution of the underlying SDE lies on  ${\mathcal M} = ({\mathbb S}^2)^N$, the Hamilton-Jacobi-Bellman equation is defined on the manifold $[0,T] \times \mathcal{M}$. In Section \ref{sec:hjb} we verify that the value function is the unique solution of that Bellman equation and
that it belongs to $C^{1,2} \bigl([0,T] \times \mathcal{M}\bigr)$.
 To solve this semi-linear PDE by deterministic numerical strategies seems non-accessible due to the high dimension of the underlying manifold ${\mathcal M} $; we also want to avoid a direct probabilistic representation of its solution which would involve a backward SDE.
 Indeed, we demonstrate how the nonlinear HJB
 equation may be replaced with a {\em linear} parabolic PDE~\eqref{lHJB} by applying the Hopf-Cole transformation. The quadratic form resp. linearity of the control in the cost functional resp. in the equation~\eqref{sde-1} together with the geometric character of the problem then lead to an isotropic quadratic term in the HJB equation~\eqref{HJB}, which is crucial for this transformation. The regularity of the value function and the optimal policy mapping is explicitly expressed through the regularity of the terminal condition $h$. Furthermore, the solution~$w$ of the linear parabolic PDE can now be represented via a Feynman-Kac formula. This is the starting point for the numerical scheme proposed 
 in Section \ref{sec:numerics}. To approximate the optimal pair $({\bf m}^*, {\bf u}^*)$ numerically, a Monte-Carlo method for the solution $w$ of the linear equation \eqref{HJB} and its tangential gradient $\nabla_{\mathcal{M}}w$ is proposed, from which the optimal feedback function $\bar{\bf u}$ can be
 obtained directly via \eqref{eq:control-feedback}. To approximate $\nabla_{\mathcal{M}}w$ through a difference quotient with needed accuracy, we choose a stencil 
 diameter $\bar{h} = \mathcal{O}\big(1/\sqrt{M}\big)$ for a sufficiently large number of Monte-Carlo realizations $M$; see Remark~\ref{rem:h-M-parameters}. Importantly, this approach does not require larger data storage resources as \cite{dunst_majee_prohl_vallet,dunst_prohl} does, but an ample calculation of iterates from related SDEs. Computational studies for the switching dynamics of single and multiple ferromagnetic particles are reported in Section~\ref{sec:computational-experiments}.

\section{The stochastic Landau-Lifschitz-Gilbert equation}\label{sec:sllg}
The solution process ${\bf m}$ of \eqref{sde-1} attains values in ${\mathcal M} = (\mathbb{S}^2)^N$. For any 
\begin{align*}
\pmb{m} = ({\it m}_1, \ldots, {\it m}_N)^\top = \big({\it m}_{1,1}, {\it m}_{1,2}, {\it m}_{1,3}, \ldots, {\it m}_{N,1}, {\it m}_{N,2}, {\it m}_{N,3}\big)^\top \in {\mathcal M}
\end{align*}
we have
\begin{align*}
\| \pmb{m} \|^2_{(\R^3)^N} = \sum_{i=1}^N \| {\it m}_i \|^2_{{\mathbb R}^3} = 
\sum_{i=1}^N \sum_{\ell=1}^3 | {\it m}_{i,\ell} |^2 = N\, .
\end{align*}
For any $\pmb{m} = ({\it m}_1, \ldots, {\it m}_N)^\top \in \mathcal{M}$, denote $\pmb{\sigma}(\pmb{m}) = {\tt diag} \bigl(\sigma(m_1), \ldots, \sigma(m_N) \bigr)$,
where $\sigma({\it m}_i) \in {\R}^{3 \times 3}$ is the matrix 
\begin{align*}
\sigma(m_i) = 
\begin{pmatrix} 
0 & - {\it m}_{i,3} & {\it m}_{i,2}\\
{\it m}_{i,3} & 0 & -{\it m}_{i,1}\\ 
- {\it m}_{i,2} & {\it m}_{i,1} & 0
\end{pmatrix}\,.
\end{align*}
Again, for any $(\pmb{m}, \pmb{u})\in \mathcal{M}\times (\R^3)^N$, we define
\begin{align*}
{\bf f}(\pmb{m}, \pmb{u}) := \pmb{m} \times \mathbf{H}_{\tt eff} (\pmb{m},\pmb{u}) - \alpha \, \pmb{m} \times \bigl[\pmb{m} \times \mathbf{H}_{\tt eff}(\pmb{m},\pmb{u}) \bigr]\, .
\end{align*}
Then, using $\pmb{\sigma}$ also here, and  combining   $H_{\tt ani,i}({\it m}_i) + H_{\tt d,i}({\it m}_i)= - {\tt D}_i {\it m}_i$ with
some $ {\tt D}_i= {\tt B}_i - {\tt A}\in {\mathbb R}^{3 \times 3}_{\tt diag} $ as well as  ${\bf D}={\tt diag}\big( {\tt D}_1, \ldots, {\tt D}_N\big)$, we have
\begin{align}
{\bf f}(\pmb{m}, \pmb{u}) 
& = \pmb{\sigma}(\pmb{m}) \mathbf{H}_{\tt eff} (\pmb{m},\pmb{u})   - \alpha \pmb{m}\times \pmb{\sigma}(\pmb{m}) \mathbf{H}_{\tt eff} (\pmb{m},\pmb{u})= \big({\bf Id} - \alpha \, \pmb{\sigma}(\pmb{m}) \big)  \, \pmb{\sigma}(\pmb{m})  \mathbf{H}_{\tt eff} (\pmb{ m},\pmb{u})\notag \\
& = \pmb{\Sigma} \bigl(\pmb{m}) \big(-\mathbf{J} \pmb{m} - {\bf D} \pmb{m} + C_{\tt ext}\,\pmb{u} \bigr)\, , \label{eq:expression-drift}
\end{align}
where 
\begin{align}
\pmb{\Sigma}(\pmb{m}) = \bigl({\bf Id} - \alpha \, \pmb{\sigma}(\pmb{m}) \bigr) \pmb{\sigma}(\pmb{m})\,.\label{defi:sigma}
\end{align} 
The matrix $\pmb{\Sigma}(\pmb{m}) \in {\mathbb R}^{3N \times 3N}$ is block-diagonal, with its $i$-th block 
\begin{align*}
\Sigma_i({\it m}_i) = \begin{pmatrix}
\alpha\, ({\it m}_{i,2}^2 + {\it m}_{i,3}^2) & -{\it m}_{i,3} - \alpha\, {\it m}_{i,1} {\it m}_{i,2}      & {\it m}_{i,2} - \alpha\, {\it m}_{i,1} {\it m}_{i,3}\\
{\it m}_{i,3} - \alpha\, {\it m}_{i,1} {\it m}_{i,2}     & \alpha\, ({\it m}_{i,1}^2 + {\it m}_{i,3}^2) & -{\it m}_{i,1} - \alpha \,{\it m}_{i,2} {\it m}_{i,3}\\
- {\it m}_{i,2} - \alpha\, {\it m}_{i,1} {\it m}_{i,3}     & {\it m}_{i,1} -\alpha\, {\it m}_{i,2} {\it m}_{i,3}      & \alpha\, ({\it m}_{i,1}^2 + {\it m}_{i,2}^2)
\end{pmatrix}.
\end{align*}
For ${\it m}_i \in \mathbb{S}^2$, one has 
\begin{align*}
\Sigma_i({\it m}_i) = \bigl({\rm Id} - \alpha \, \sigma({\it m}_i) \bigr) \sigma({\it m}_i) = \sigma({\it m}_i) + \alpha\, \bigl({\rm Id}  - {\it m}_i \otimes {\it m}_i \bigr) = \sigma({\it m}_i) + \alpha \, {\mathcal P}({\it m}_i)\,,
\end{align*}
where ${\mathcal P}({\it m}_i)$ is the orthogonal projection onto the tangent plane of $\mathbb{S}^2$  at ${\it m}_i$. Note that the diffusion term $\nu\, m_i(s) \times (\circ {\rm d}W_i(s))$ in \eqref{sde-1} can be re-written as $\nu\, \sigma(m_i(s))\circ {\rm d}W_i(s)$. To state the dynamic programming equation, we introduce a family of stochastic
Landau-Lifschitz-Gilbert equations with different initial times $t \in [0,T]$ and states $\pmb{m} \in {\mathcal M}$:
\begin{equation} \label{SLLG}
\begin{aligned}
{\rm d} {\bf m}(s)&= {\bf f}({\bf m}(s), {\bf u}(s))\,{\rm d}s + \nu \, \pmb{\sigma}({\bf m}(s)) \circ \,{\rm d}\mathbf{W}(s) \qquad (t< s \le T)\, ,  \\
 {\bf m}(t)&= {\pmb m} \in {\mathcal M}\, ,
 \end{aligned}
\end{equation}
where ${\bf f}$ is defined in \eqref{eq:expression-drift}. The solutions ${\bf m} = {\bf m}^{t, {\pmb m}}$ of \eqref{SLLG} thus depend on $t$ and ${\pmb m}$; however, we
shall drop the superscript of ${\bf m}^{t, {\pmb m}}$ in the subsequent text for the ease of notation. For every $0\le t\le T$, ${\bf m}(t)=\pmb{m}$, and ${\bf u} \in L^2_{\{ {\mathcal F}_s\}}\bigl( \Omega; L^2(t,T; (\R^3)^N)\bigr)$, there exists a
 unique strong solution ${\bf m}=(m_1, \ldots, m_N) \in L^2_{\{\mathcal{F}_s\}}\big( \Omega; C(t,T; (\R^3)^N)\big)$ of \eqref{SLLG}. Indeed by considering the truncation of the control and then using the stochastic version of the Arzela-Ascoli theorem, Prohorov's
 lemma and Jakubowski-Skorokhod representation theorem (cf.~\cite{Jakubowski}), we have existence of a weak solution of \eqref{SLLG}. Moreover, an application of Gy\"{o}ngy-Krylov's characterization of convergence in probability introduced in \cite{gyongy-krylov} along with pathwise uniqueness of weak martingale solutions gives  existence of a unique strong solution, see \cite[Appendix]{dunst_majee_prohl_vallet}.  Furthermore, by applying It\^{o}'s formula to the functional ${\bf x}\goto \|{\bf x}\|_{\R^3}^2$ for $m_i$ and using the vector identity $\langle {\bf a}, {\bf a} \times {\bf b}\rangle =0$ for any ${\bf a}, {\bf b}\in \R^3$, we have $\mathbb{P}$-a.s.,
 \begin{align*}
 \|m_i(s)\|_{\R^3}^2&=  \|m_i(t)\|_{\R^3}^2 + \int_t^s\Big[  \nu^2\,\|m_i(r)\|_{\R^3}^2 
 + 2\,\big\langle {\bf f}_i({\bf m}(r),{\bf u}(r)), m_i(r)\rangle 
 -\nu^2\,\|m_i(r)\|_{\R^3}^2\Big]\,{\rm d}r \\
 &\hspace{1cm} + 2\nu\, \int_t^s \langle m_i(r)\times {\rm d}W_i(r), m_i(r)\rangle =  \|m_i(t)\|_{\R^3}^2\,.
 \end{align*}
 Since $m_i(t)\in \mathbb{S}^2$, we see that $\mathbb{P}$-a.s., each $m_i$ is $\mathbb{S}^2$-valued, and thus ${\bf m}\in L^2_{\{\mathcal{F}_s\}}\big( \Omega; C(t,T; (\mathbb{S}^2)^N)\big)$.
\vspace{.2cm}

Because the paths of the Landau-Lifschitz-Gilbert process stay on the manifold $\mathcal{M}$, the natural domain for the value function
of the control problem is $[0,T] \times \mathcal{M}$. In order to make the connection between the controlled process $\bf m$ on the one hand and
the Hamilton-Jacobi-Bellman PDE posed on $[0,T] \times \mathcal{M}$ on the other hand, it is convenient to describe properties of $\bf m$ purely in terms of quantities that are
intrinsically defined on $\mathcal{M}$, without referring to the ambient space $(\R^3)^N$. Of particular interest is Dynkin's formula.
\vspace{.2cm}

We begin by rewriting It\^o's formula with tangential derivatives.
For $1 \leq \ell \leq 3N$, let $\pmb{\sigma}(\pmb{m})_\ell$ be the $\ell$th row of $\pmb{\sigma}(\pmb{ m})$, for $\pmb{m} \in {\mathcal M}$. Then 
$\partial_{\pmb{\sigma}(\pmb{m})_\ell}$ denotes the tangential derivative in the direction $\pmb{\sigma}(\pmb{m})_\ell$, and $\partial_{\pmb{\sigma}(\pmb{m})} :=
\bigl( \partial_{\pmb{\sigma}(\pmb{m})_1}, \ldots, \partial_{\pmb{\sigma}(\pmb{ m})_{3N}}\bigr) \in [T_{\pmb{m}} {\mathcal M}]^{3N}$.
Similarly, $\partial_{{\bf f}(\pmb{m}, \pmb{u})}$ is the tangential derivative in the direction  ${\bf f}(\pmb{m}, \pmb{u})$. 
\vspace{.2cm}

We wish to apply It\^o's formula to $\psi (s,{\bf m}(s))$ for any $\psi \in C^{1,2}\big([0,T]\times \mathcal{M}\big)$. One may directly return to the standard formula on $(\R^3)^N$, cf.~\cite[Chapter V.1]{Watanabe}, by extending $\psi$ via
\begin{align*}
\hat{\psi}(s,\hat{\pmb{m}}) = \psi\Big( s,\big( \frac{\hat{{\it m}}_1}{\|\hat{{\it m}}_1\|_{\R^3}},\ldots, \frac{\hat{{\it m}}_N}{\|\hat{{\it m}}_N\|_{\R^3}}\big)\Big)
\end{align*}
to $[0,T] \times (\R^3\setminus \{0\})^N$, for any $\hat{\pmb{m}}=\big(\hat{{\it m}}_1,\ldots, \hat{{\it m}}_N\big)\in (\R^3\setminus \{0\})^N$. Then
${\mathbb P}$-a.s.:
\begin{align} \label{wSLLG}
\psi \bigl(s,{\bf m}(s) \bigr) - \psi\big(t, {\pmb m}\big) &= \int_t^s \partial_t \psi\big(r, {\bf m}(r)\big)\, {\rm d}r + \int_t^s \partial_{{\bf f} \bigl({\bf m}(r), \mathbf{u}(r) \bigr)} \psi\big(r,{\bf m}(r)\big) \, {\rm d} r \notag \\
& \hspace{0.4cm} + \nu \int_t^s \partial_{\pmb{\sigma} \bigl({\bf m}(r)\bigr)}\psi\big(r,{\bf m}(r)\big) \circ \,{\rm d}\mathbf{W}(r) \quad (t \leq s \leq T)\, .
\end{align}
 For $\psi \in C^{1,2}([0,T]\times {\mathcal M})$, we associate the generator of the Markov process ${\bf m}$ as
\begin{align*}
\mathcal{A}^{\pmb{u}} \, \psi(s, \pmb{m}) = \frac{\nu^2}{2}\, \Delta_{\mathcal M} \psi(s,\pmb{m}) + \partial_{{\bf f}\big({\pmb{m}}, \pmb{u}\big)} \psi(s,\pmb{m})\,,
\end{align*}
where $\Delta_{\mathcal M}$ denotes the Laplace-Beltrami operator on ${\mathcal M}$, and define the operator
\begin{align}
\mathcal{A}_1^{\pmb{u}} \, \psi(s, \pmb{m}):= \partial_t\psi(s,\pmb{m}) + \mathcal{A}^{\pmb{u}} \, \psi(s, \pmb{m})\,.  \label{eq:generator}
\end{align}
 Re-writing \eqref{wSLLG} with $s=T$ in It\^{o} form, we have 
\begin{align*}
&\psi\big(T,{\bf m}(T)\big) - \psi\big(t,\pmb{m}\big) \notag \\
& =\int_t^T \Bigl[ \partial_t \psi\big(r,{\bf m}(r)\big) + \partial_{{\bf f}\big({\bf m}(r), \mathbf{u}(r)\big)} \psi\big(r,{\bf m}(r)\big) + \frac{\nu^2}{2}\, \partial_{\pmb{\sigma}\big({\bf m}(r)\big)} \partial_{\pmb{\sigma}\big({\bf m}(r)\big)} \psi\big(r,{\bf m}(r)\big) \Bigl] \,{\rm d} r\\
& \hspace{2cm} +\nu\, \int_t^T \partial_{\pmb{\sigma}\big({\bf m}(r)\big)} \psi\big(r,{\bf m}(r)\big) \,{\rm d} \mathbf{W}(r)\, .
\end{align*}
As in the proof of \cite[Proposition~$3.2$]{NP} we conclude that
\begin{align*}
\partial_{\pmb{\sigma}\big({\bf m}(r)\big)} \partial_{\pmb{\sigma}\big({\bf m}(r)\big)} \psi\big(r,{\bf m}(r)\big) = \Delta_{\mathcal{M}} \psi\big(r,{\bf m}(r)\big)\, .
\end{align*}
Taking the expectation then leads to Dynkin's formula
\begin{align} \label{eq:Dynkin}
\mathbb{E}_{t,\pmb{m}}\Big[\psi\big(T, \mathbf{m}(T)\big)\Big] - \psi(t, \pmb{m}) = \mathbb{E}_{t,\pmb{m}}\Big[ \int_t^T \mathcal{A}_1^{\mathbf{u}(r)} \psi\big(r, \mathbf{m}(r)\big) \, {\rm d} r\Big]\,,
\end{align}
recalling that the It\^{o} integral is a martingale.

\section{Dynamic programming  and HJB equation}\label{sec:hjb}
For any $(t,\pmb{m})\in [0,T]\times \mathcal{M}$, we consider problem \eqref{SLLG} to now construct the associated Hamilton-Jacobi-Bellman equation, following the formal rules of 
 dynamic programming. We then use the Hopf-Cole transformation to replace the nonlinear HJB equation by a linear PDE and show the existence of a unique classical solution, which then implies
 existence of a unique classical solution of the original nonlinear HJB equation. Next, we present a verification theorem which shows that the value function is indeed equal to the solution 
 of the HJB equation. We describe the optimal control through an optimal feedback  function which is written explicitly in terms of the value function.
\vspace{.2cm}

Let us define the Lagrangian 
\begin{align}
L\big(\pmb{m}, \pmb{u}\big) = \delta \| \pmb{m} - \widetilde {\bf m} \|_{(\R^3)^N}^2 + \frac{\lambda}{2}\, \| \pmb{u} \|_{(\R^3)^N}^2\,, \label{eq:lagrangian}
\end{align}
where the parameters $\delta, \lambda$ are given in Problem~\ref{problem-1}.
\vspace{.2cm}
 
Let $\big(\Omega, \mathbb{P}, \mathcal{F},\{\mathcal{F}_s\}_{t \le s \le T}\big)$ be a given filtered probability space satisfying the usual hypotheses, and $\mathbf{W}$ is 
 a $\{\mathcal{F}_s\}_{t \le s \le T}$-adapted $(\R^3)^N$-valued Wiener process on it. 
We denote by $\mathcal{U}_{\pmb{m}}^{\pmb{s}}[t,T]$ the set of all admissible pairs $({\bf m}, {\bf u})$ such that $\mathbf{u} \in L^2_{\{{\mathcal F}_{s}\}}\bigl( \Omega; L^2( t,T; ({\mathbb R}^{3})^N )\bigr)$, and ${\bf m(\cdot)}$ is the unique $\{ \mathcal{F}_s \}_{t \le s \le T}$-adapted $ \mathcal{M}$-valued strong solution of \eqref{SLLG}. In fact, the superscript $\pmb{s}$ refers to the fact that we search an
optimal admissible pair on the given filtered probability space. It follows for admissible $({\bf m}, {\bf u})$ that $L\big({\bf m}(\cdot), {\bf u}(\cdot)\big)\in L^1_{\{{\mathcal F}_{s}\}}\bigl( \Omega; L^1( t,T; \R )\bigr)$ 
and $h({\bf m}(T))\in L^1_{\mathcal{F}_T}(\Omega; \R)$, recalling that $h$ is a continuous function on a compact manifold. As a further consequence,  Dynkin's formula \eqref{eq:Dynkin} then holds for all $\psi \in C^{1,2}([t,T]\times \mathcal{M})$ and $({\bf m}, {\bf u}) \in \mathcal{U}_{\pmb{m}}^{\pmb{s}}[t,T]$.
\vspace{.2cm}

 Our aim is to achieve ${\bf m}$ close to a
reference configuration $\widetilde {\bf m} \in C^2([t,T]; \mathcal{M})$ by selecting an optimal control ${\bf u}^*$. The cost functional on $\mathcal{U}_{\pmb{m}}^{\pmb{s}}[t,T]$ is
\begin{align} \label{costf}
\mathcal{J}\big(t,\,\pmb{m}; ({\bf m},{\bf u})\big) = \mathbb{E}_{t,\pmb{m}} \Bigl[ \int_t^T L({\bf m}(r), {\bf u}(r)) \,{\rm d} r + h\big({\bf m}(T)\big) \Bigr]\,.
\end{align}
We write the value function of $\mathcal{J}$ as
\begin{align}
 V(t,\pmb{m})=\inf_{ ({\bf m},{\bf u}) \in \mathcal{U}_{\pmb{m}}^{\pmb{s}}[t,T]} \mathcal{J}\big(t,\,\pmb{m};\,({\bf m},{\bf u})\big)\,. \label{eq:value-function}
\end{align} 
Note that, thanks to \cite[Proposition~$1.33$]{banas_brzezniak_neklyudov_prohl}, there exists a unique strong solution ${\bf m}$ of \eqref{SLLG} for ${\bf u}={\bf 0}$, and hence the value function is uniformly bounded since $h$ is a given continuous function and $\|{\bf m}\|_{(\R^3)^N}^2=N$.
\subsection{The Hamilton-Jacobi-Bellman equation}\label{subsec:hjb}
 We define the Hamiltonian
\begin{align*}
\mathcal{H} : \; C^2(\mathcal{M}) &\to C(\mathcal{M}) \notag \\
 \psi  &\mapsto - \frac{\nu^2}{2} \, \Delta_{\mathcal M} \psi - \delta\, \| \pmb{m} - \widetilde{\bf m} \|_{(\R^3)^N}^2 
+ \sup_{\pmb{u} \in (\R^3)^N} \Big( - \partial_{{\bf f}\big(\pmb{m},\pmb{u}\big)} \psi - \frac{\lambda}{2} \, \| \pmb{u} \|_{(\R^3)^N}^2 \Big)\,.
\end{align*}
Note that in the definition of $\mathcal{H}$, the letter $\pmb{m}$ stands for the identity map $\pmb{m} \mapsto \pmb{m}$. Using \eqref{eq:expression-drift}, we evaluate the supremum analytically. We have for the tangential gradient $\nabla_\mathcal{M} \psi$
\begin{align}
& \sup_{\pmb{u} \in (\R^3)^N} \Big( -  \partial_{{\bf f}\big(\pmb{m},\pmb{u}\big)} \psi- \frac{\lambda}{2} \, \| \pmb{u} \|_{(\R^3)^N}^2 \Big) \notag \\
&=\sup_{\pmb{u} \in (\R^3)^N} \Big(- {\bf f}(\pmb{m}, \pmb{u}) \cdot \nabla_\mathcal{M} \psi(\pmb{m}) - \frac{\lambda}{2} \, \| \pmb{u} \|_{(\R^3)^N}^2 \Big) \notag \\
&= \sup_{\pmb{u} \in (\R^3)^N} \Big(- \pmb{\Sigma}(\pmb{m})(- \,{\bf J} \pmb{m} - {\bf D}\pmb{m} + C_{\tt ext} \pmb{u}) \cdot \nabla_\mathcal{M} \psi(\pmb{m}) - \frac{\lambda}{2} \, \| \pmb{u} \|_{(\R^3)^N}^2 \Big)\notag \\
 &= \pmb{\Sigma}(\pmb{m}) \big( {\bf J} \pmb{m} + {\bf D}\pmb{m}\big) \cdot \nabla_\mathcal{M} \psi(\pmb{m}) + \sup_{\pmb{u} \in (\R^3)^N} \Big(- C_{\tt ext} \pmb{\Sigma}(\pmb{m}) \pmb{u} \cdot \nabla_\mathcal{M} \psi(\pmb{m}) - \frac{\lambda}{2} \, \| \pmb{u} \|_{(\R^3)^N}^2 \Big)\notag \\
 &= \pmb{\Sigma}(\pmb{m}) {\bf Q} \pmb{m}\cdot \nabla_\mathcal{M} \psi(\pmb{m}) + \sup_{\pmb{u} \in (\R^3)^N} \Big(\pmb{u} \cdot \big(- C_{\tt ext}\pmb{\Sigma}^\top(\pmb{m})\nabla_\mathcal{M} \psi(\pmb{m})\big) - \frac{\lambda}{2} \, \| \pmb{u} \|_{(\R^3)^N}^2 \Big)\notag \\
 &= \pmb{\Sigma}(\pmb{m}) {\bf Q} \pmb{m} \cdot \nabla_\mathcal{M} \psi(\pmb{m}) + \big(\frac{\lambda}{2} \, \| \cdot \|_{(\R^3)^N}^2\big)^* \big(-C_{\tt ext}\pmb{\Sigma}^\top(\pmb{m})\nabla_\mathcal{M} \psi(\pmb{m})\big)\notag \\
& = \pmb{\Sigma}(\pmb{m}) {\bf Q} \pmb{m} \cdot \nabla_\mathcal{M} \psi(\pmb{m}) + \frac{1}{2\lambda}\, \| - C_{\tt ext}\pmb{\Sigma}^\top(\pmb{m})\nabla_\mathcal{M} \psi(\pmb{m}) \|_{(\R^3)^N}^2\,, \label{Legendre}
\end{align}
where $^*$ denotes the convex conjugate function,
and ${\bf Q}$ is the $3N \times 3N$ matrix given by 
\begin{align}
{\bf Q}:=  {\bf J} + {\bf D}\, .\label{eq:Q}
\end{align}
Since $\nabla_\mathcal{M} \psi(\pmb{m})$ belongs to the tangent space $T_{\pmb{m}}\mathcal{M}$, it follows that
\begin{align*}
- \pmb{\Sigma}^\top(\pmb{m})\nabla_\mathcal{M} \psi(\pmb{m})
& = \pmb{\sigma}(\pmb{m})\nabla_\mathcal{M} \psi(\pmb{m}) + \alpha\, \pmb{\sigma}(\pmb{m}) \pmb{\sigma}(\pmb{m})\nabla_\mathcal{M} \psi(\pmb{m}) \\
&= \pmb{m} \times \nabla_{\mathcal{M}} \psi(\pmb{m}) + \alpha\,  \pmb{m} \times (\pmb{m}\times \nabla_\mathcal{M} \psi(\pmb{m})) 
=  \pmb{m} \times \nabla_\mathcal{M} \psi(\pmb{m})- \alpha\, \nabla_\mathcal{M} \psi(\pmb{m})\,.
\end{align*}
 Note that
for any $\pmb{m} \in \mathcal{M}$, 
 $\| \pmb{m} \times \nabla_\mathcal{M} \psi(\pmb{m}) \|_{(\R^3)^N} = \| \nabla_\mathcal{M} \psi(\pmb{m}) \|_{(\R^3)^N}$. Thus by Pythagoras' theorem, 
we have 
\begin{align}
&\|- \pmb{\Sigma}^\top(\pmb{m})\nabla_\mathcal{M} \psi(\pmb{m}) \|_{(\R^3)^N}^2  = \|\pmb{m} \times \nabla_\mathcal{M} \psi(\pmb{m})- \alpha \,\nabla_\mathcal{M} \psi(\pmb{m}) \|_{(\R^3)^N}^2 \notag  \\
& = \| \pmb{m} \times \nabla_\mathcal{M} \psi(\pmb{m}) \|_{(\R^3)^N}^2 + \| \alpha \nabla_\mathcal{M} \psi(\pmb{m}) \|_{(\R^3)^N}^2
 = (1 + \alpha^2) \| \nabla_\mathcal{M} \psi(\pmb{m}) \|_{(\R^3)^N}^2\, . \label{eq:pythagoras}
\end{align}
Let us denote 
\begin{align}
{\bf b}(\pmb{m}) := \pmb{\Sigma}(\pmb{m}) {\bf Q} \pmb{m}= \pmb{\Sigma}(\pmb{m}) \big( {\bf J} \pmb{m} + {\bf D} \pmb{m}\big)\, .\label{defi:b}
\end{align}
Then, in summary, we have 
\begin{align*}
\mathcal{ H} \psi(\pmb{m})=& - \frac{\nu^2}{2} \, \Delta_{\mathcal{M}} \psi(\pmb{m}) + \frac{C_{\tt ext}^2(1+ \alpha^2)}{2\lambda}\, \|\nabla_\mathcal{M} \psi(\pmb{m}) \|_{(\R^3)^N}^2\\
& \qquad +\, {\bf b}(\pmb{m}) \cdot \nabla_\mathcal{M}\psi(\pmb{m}) - \delta\, \| \pmb{m} - \widetilde{{\bf m}} \|_{(\R^3)^N}^2\,.
\end{align*}
We point out that \eqref{eq:pythagoras} ensures that the quadratic term in the Hamiltonian is isotropic, which is crucial for the Hopf-Cole transformation below.  
We now state the Hamilton-Jacobi-Bellman equation, whose solution we denote by $W$:

\begin{equation}\label{HJB}
\begin{aligned}
-\partial_t W(t,\pmb{m}) + \mathcal{H} W(t,\pmb{m}) &= 0 \hspace{2cm} \text{in}\quad [0,T) \times \mathcal{M}\, , \\
 W(T,\pmb{m}) &= h(\pmb{m}) \hspace{1.3cm} \text{on}\quad \mathcal{M}\,. 
\end{aligned}
\end{equation}

\noindent{\bf The Hopf-Cole transformation:}
The HJB equation \eqref{HJB} is a second-order nonlinear PDE on a high-dimensional domain and therefore without further understanding of its structure computationally expensive to 
solve; the study of its wellposedness as well as the regularity of its solution is non-trivial. We use the Hopf-Cole transformation $w = \exp \bigl(- \beta \, W \bigr)$, $\beta \in \R$ given below, to
substitute \eqref{HJB} by the linear PDE \eqref{lHJB}. 
\vspace{.2cm}

We span the tangent space of $\mathcal{M}$ at any point $\pmb{m}$ by the orthonormal \green{tangent} vectors
\begin{align*}
\partial_{1,1}, \partial_{1,2}, \ldots, \partial_{N,1}, \partial_{N,2}\,.
\end{align*}
It is convenient to conceptually let $\partial_{i,1}, \partial_{i,2}$ span the local coordinates associated to the $i$th sphere. Then we have the following relations: for $j\in \{1,2\}$ and $i\in \{1,2,\cdots, N\}$,
\begin{align*}
\partial_{i,j} w & = - \beta \, w \; \partial_{i,j} W && \Leftrightarrow \qquad \partial_{i,j} W = - \frac{1}{\beta \, w} \; \partial_{i,j} w\,,\\
\partial_{i,j;i,j}^2 w & = - \beta \, w \; \bigl( - \beta \, |\partial_{i,j} W|^2 + \partial_{i,j;i,j}^2 W\bigr) && \Leftrightarrow \qquad \partial_{i,j;i,j}^2 W = - \frac{1}{\beta \, w} \; \partial_{i,j;i,j}^2 w + \beta \, |\partial_{i,j} W|^2\,.
\end{align*}
 We see that
\begin{align*}
\Delta_{\mathcal{M}} W & = \sum_i \Bigl[ \partial_{i,1; i,1}^2 W + \partial_{i,2; i,2}^2 W \Bigr]\\
& = \sum_i \Bigl[ \Bigl( - \frac{1}{\beta \, w} \; \partial_{i,1; i,1}^2 w + \beta \, |\partial_{i,1} W|^2 \Bigr) 
+ \Bigl( - \frac{1}{\beta \, w} \; \partial_{i,2; i,2}^2 w + \beta \, |\partial_{i,2} W|^2 \Bigr) \Bigr]\\
& = - \frac{\Delta_{\mathcal{M}} w} {\beta \, w} + \beta\, \sum_i \Bigl[ |\partial_{i,1} W|^2 + |\partial_{i,2} W|^2 \Bigr]\,.
\end{align*}
Therefore, we have the following correspondences:

\medskip

\begin{center}
\begin{tabular}{c|c|c}
& $W$ variable  & $w$ variable\\ \hline & \\[-2mm]
(a) & $- \partial_t W$ & $ \displaystyle \frac{1}{\beta \, w}\,\partial_t w$\\[4mm]
(b) & $ \displaystyle- \frac{\nu^2}{2} \Delta_{\mathcal{M}} W$ & $ \displaystyle \frac{\nu^2}{2\beta \, w} \Delta_{\mathcal{M}} w - \frac{\nu^2}{2}\beta \sum_i \bigl[ |\partial_{i,1} W|^2 + |\partial_{i,2} W|^2 \bigr]$\\[5mm]
(c) & ${\bf b} \cdot \nabla_{\mathcal{M}} W$ & $- \displaystyle \frac{1}{\beta \, w}\,{\bf b}\cdot \nabla_\mathcal{M}w $\\[1mm]
\end{tabular}
\end{center}

Recalling $\| \nabla_\mathcal{M} W \|_{(\R^3)^N}^2 = \sum_i \bigl[ |\partial_{i,1} W|^2 + |\partial_{i,2} W|^2 \bigr]$, we choose $ \displaystyle \beta = \frac{ C_{\tt ext}^2(1 + \alpha^2)}{\lambda \nu^2}$ to obtain a cancellation of the quadratic nonlinearity via ${\rm (b)}$. Substituting the respective terms in \eqref{HJB}
and multiplying by $ - \beta \, w \neq 0$, we arrive at the following second-order linear  equation
\begin{align} \label{lHJB}
- &\partial_t w(t,\pmb{m}) - \frac{\nu^2}{2}\Delta_{\mathcal{M}} w(t,\pmb{m}) + {\bf b}(\pmb{m}) \cdot \nabla_\mathcal{M} w(t,\pmb{m}) + c(t,\pmb{m}) \, w(t, \pmb{m}) = 0 \quad \text{in}\quad [0,T) \times \mathcal{M}\,, \notag \\
&w(T,\pmb{m})= \exp\big( -\beta\, h(\pmb{m})\big) \quad \text{on}\quad \mathcal{M}\,,
\end{align}
where $c(t,\pmb{m})=\beta\,\delta \, \| \pmb{m} - \widetilde {{\bf m}}(t) \|_{(\R^3)^N}^2$.
The following theorem shows that weak solutions $w$ may be examined in the Sobolev-Bochner space
\begin{align*}
W^{1,2,2}\big([0,T]; H^1({\mathcal{M}}), H^{-1}({\mathcal{M}})\big) = \left\{ u \in L^2\left([0,T]; H^1({\mathcal{M}})\right); \; \frac{{\rm d}u}{{\rm d}t} \in L^2\left([0,T]; H^{-1}({\mathcal{M}})\right)\right\}.
\end{align*}
A weak solution of \eqref{lHJB} is a $w \in W^{1,2,2}\big([0,T]; H^1({\mathcal{M}}), H^{-1}({\mathcal{M}})\big)$ such that for all $\psi \in L^2\big(0,T; H^1(\mathcal{M})\big)$:
\begin{align} \label{wHJB}
\int_0^T \int_{\mathcal{M}} - \partial_t w \, \psi + \frac{\nu^2}{2} \, \nabla_\mathcal{M} w \cdot \nabla_\mathcal{M} \psi + {\bf b} \cdot \nabla_\mathcal{M} w \, \psi + c \, w \, \psi \,{\rm d}\pmb{m} \, {\rm d}t = 0\, .
\end{align}
\begin{thm}\label{thm:equiweak}
There exists a unique classical solution $w \in C^{1,2}\big([0,T] \times \mathcal{M}\big)$ of \eqref{lHJB} which is also the unique weak solution in $W^{1,2,2}\big([0,T]; H^1({\mathcal{M}}), H^{-1}({\mathcal{M}})\big)$. The function 
\begin{align} \label{eq:reverseHC}
W(t,\pmb{m}) = - \frac{1}{\beta}\log \big(w(t,\pmb{m})\big)
\end{align}
is the unique classical solution of the Bellman equation \eqref{HJB}.
\end{thm}
\begin{proof}
The existence of a unique solution $w$ of \eqref{wHJB} in $W^{1,2,2}\big([0,T]; H^1({\mathcal{M}}), H^{-1}({\mathcal{M}})\big)$ is for instance given in \cite[p.~224]{Roubicek}. Using
charts of $\mathcal{M}$ and a partition of unity to represent \eqref{wHJB} locally on the flat space, and an application of parabolic regularity theory ensures that $w$ belongs
to $C^{1,2}\big([0,T] \times \mathcal{M}\big)$. This implies that the nonlinear HJB equation \eqref{HJB} has a classical solution $W$. Moreover, the above construction of the Hopf-Cole transformation directly ensures that a function $w \in C^{1,2}\big([0,T] \times \mathcal{M}\big)$ is a classical solution of \eqref{lHJB} if
and only if $W$ of \eqref{eq:reverseHC} solves \eqref{HJB} classically, which then implies the existence of a unique classical solution $W$, given by \eqref{eq:reverseHC}, of the HJB equation \eqref{HJB}.
\end{proof}
It is easy to see that additional smoothness of the terminal condition $h$ directly translates into additional regularity of $w$ and $W$.
\subsection{The verification theorem}\label{subsec:VW}
In this subsection, we show that $V = W$, i.e.,~the value function of the optimal control problem is equal to the solution of the above HJB equation \eqref{HJB}. The following verification theorem also 
provides an explicit formula for the optimal control, which inherits the smoothness of $\nabla_{\mathcal{M}} W$. 
\begin{thm}\label{thm:verification}
The value function $V(t,\pmb{m})$ in \eqref{eq:value-function} is the unique classical solution of the nonlinear HJB equation \eqref{HJB}:
\begin{align*}
V(t,\pmb{m}) = W(t,\pmb{m}) = - \frac{\log w(t,\pmb{m})}{\beta}, \qquad \forall\,(t,\pmb{m})\in [0,T]\times \mathcal{M}. 
\end{align*}
Problem~\ref{problem-1} admits a minimizer $({\bf m}^*, {\bf u}^*)\in \mathcal{U}_{\pmb{m}}^{\pmb{s}}[t,T]$ such that $\mathcal{J}\big(t, \pmb{m};({\bf m}^*, {\bf u}^*)\big)=V(t,\pmb{m})$ and $\mathbf{u}^*(s)=\bar{\mathbf{u}}\big(s,{\bf m}^*(s)\big)$, where
\begin{align}
 \bar{\mathbf{u}}(t,\pmb{m})= \frac{C_{\tt ext}}{\lambda} \Big(\pmb{m} \times \nabla_\mathcal{M} W(t,\pmb{m}) - \alpha\, \nabla_\mathcal{M} W(t,\pmb{m})\Big)\,.\label{eq:control-feedback}
\end{align}
\end{thm}

\begin{proof} The proof is divided into two steps. 
\vspace{.2cm}

 \noindent{\em Step 1:} First we show that $W \le V$. Let $({\bf m}, {\bf u})\in \mathcal{U}_{\pmb{m}}^{\pmb{s}}[t,T]$ be any admissible pair. Now, for any  $\psi \in C^{1,2}\big([0,T) \times \mathcal{M}\big)$, we have, thanks to the definition of the Hamiltonian $\mathcal{H}$, \eqref{eq:lagrangian}, and \eqref{eq:generator},
\begin{align} 
 \mathcal{H} \psi(t,\pmb{m}) & \ge - \frac{\nu^2}{2} \, \Delta_{\mathcal{M}} \psi(t,\pmb{m}) - \delta \| \pmb{m} - \widetilde{\bf m} \|_{(\R^3)^N}^2 - \partial_{{\bf f}(\pmb{m},\pmb{u})} \psi(t,\pmb{m}) - \frac{\lambda}{2} \, \| \pmb{u} \|_{(\R^3)^N}^2 \notag \\
 & = - L( \pmb{m}, \pmb{u}) - \frac{\nu^2}{2} \, \Delta_{\mathcal{M}} \psi(t,\pmb{m}) - \partial_{{\bf f}(\pmb{m},\pmb{u})} \psi(t,\pmb{m}) \notag \\
  & = \partial_t \psi(t,\pmb{m}) - \mathcal{A}_1^{\pmb{u}} \psi(t,\pmb{m}) - L(\pmb{m}, \pmb{u})\,, \label{verification_identity}
\end{align}
and therefore one has 
\begin{align}
-\partial_t \psi(t,\pmb{m}) + \mathcal{H} \psi(t,\pmb{m}) \ge  - \mathcal{A}_1^{\pmb{u}} \psi(t,\pmb{m}) - L(\pmb{m}, \pmb{u})\,. \label{eq:gen}
\end{align}
Because $W$ is a smooth classical solution, we may substitute $\psi$ by $W$, in which case the left-hand side of \eqref{eq:gen} vanishes. In other words,
\begin{align}
\mathcal{A}_1^{\mathbf{u}(s)} W\big( s,{\bf m}(s)\big) + L\big({\bf m}(s), \mathbf{u}(s)\big) \ge 0\,. \label{eq:Lagbound}
\end{align}
Indeed, the existence of a classical solution $W$ avoids a more complicated construction to arrive at a bound like \eqref{eq:Lagbound}, which would be necessary in a
setting with viscosity solutions.  Applying now Dynkin's formula \eqref{eq:Dynkin} with $W$ in place of $\psi$ and recalling the final time conditions, we conclude from \eqref{eq:Lagbound},
\begin{align} \nonumber
W(t,\pmb{m}) &= \mathbb{E}_{t,\pmb{m}} \Bigl[\int_t^T - \mathcal{A}_1^{{\bf u}(r)} W\big(r,{\bf m}(r)\big) \,{\rm d}r + W\big(T, {\bf m}(T)\big) \Bigr] \\
&\le \mathbb{E}_{t,\pmb{m}} \Bigl[ \int_t^T L\big({\bf m}(r), {\bf u}(r)\big) \,{\rm d} r +h\big({\bf m}(T)\big) \Bigr] = \mathcal{J}\big(t,\,\pmb{m};({\bf m},{\bf u})\big)\, .  \label{eq:Dynkin_verification} 
\end{align}
Because this result holds for all admissible pairs $ ({\bf m},{\bf u}) \in \mathcal{U}_{\pmb{m}}^{\pmb{s}}[t,T]$, it follows that $W \le V$.
\vspace{.2cm}

\noindent{\em Step 2:} Now we show that there exists an admissible pair $ ({\bf m}^*, {\bf u}^*) \in \mathcal{U}_{\pmb{m}}^{\pmb{s}}[t,T]$ such that $W(t,\pmb{m})= \mathcal{J}\big(t,\,\pmb{m};({\bf m}^*, {\bf u}^*) \big)$. Recalling \eqref{Legendre}, we find that the supremum in the definition 
of the Hamilton-Jacobi-Bellman operator is attained by the control
\begin{align*}
\bar{\mathbf{u}}(s,\pmb{m}) := - \frac{C_{\tt ext}}{\lambda}\pmb{\Sigma}^\top(\pmb{m}) \, \nabla_\mathcal{M} W(s,\pmb{m}) = \frac{C_{\tt ext}}{\lambda}\Big(\pmb{m} \times \nabla_\mathcal{M} W(s,\pmb{m}) - \alpha \nabla_\mathcal{M} W(s,\pmb{m})\Big)\,.
\end{align*}
Since $W$ is a $C^{1,2}$-solution of \eqref{HJB}, we see that $\bar{\mathbf{u}}$ is continuously differentiable and bounded on $[0,T] \times \mathcal{M}$. Moreover, $\mathcal{M}\ni {\pmb m}\mapsto \bar{\bf u}(t,\pmb{m})$ is Lipschitz. Indeed, for any $\pmb{m}_1,\, \pmb{m}_2 \in \mathcal{M}$
\begin{align*}
& \| \bar{\bf u}(t, \pmb{m}_1)- \bar{\bf u}(t, \pmb{m}_2)\|_{(\R^3)^N} \\
&= \frac{C_{\tt ext}}{\lambda}\big\| 
(\pmb{m}_1 - \pmb{m}_2)\times \nabla_{\mathcal{M}}W(t,\pmb{m}_1) + \pmb{m}_2 \times \big(\nabla_{\mathcal{M}}W(t,\pmb{m}_1)- \nabla_{\mathcal{M}}W(t,\pmb{m}_2)\big) \\
& \hspace{2cm}- \alpha \big(\nabla_{\mathcal{M}}W(t,\pmb{m}_1)- \nabla_{\mathcal{M}}W(t,\pmb{m}_2)\big) \big\|_{(\R^3)^N}\,.
\end{align*}
Since $W\in C^{1,2}\big( [0,T]\times \mathcal{M}\big)$, by the mean-value theorem, either applied 
in combination with Whitney's extension theorem \cite[Section 6.5]{EG} in the ambient space $(\R^3)^N$ or directly to the geodesics of $\mathcal{M}$, we have
\begin{align*}
\|\nabla_{\mathcal{M}} W(t,\pmb{m}_1)- \nabla_{\mathcal{M}}W(t,\pmb{m}_2)\|_{(\R^3)^N} \le C\|\pmb{m}_1 - \pmb{m}_2\|_{(\R^3)^N}\,.
\end{align*}
Thus, 
\begin{align*}
\| \bar{\bf u}(t, \pmb{m}_1)- \bar{\bf u}(t, \pmb{m}_2)\|_{(\R^3)^N}  \le  C \|\pmb{m}_1 - \pmb{m}_2\|_{(\R^3)^N} \quad \forall\, \pmb{m}_1,\, \pmb{m}_2 \in \mathcal{M}\,.
\end{align*}
Now, on the given stochastic basis $(\Omega, \mathbb{P}, \mathcal{F},\{\mathcal{F}_s\}_{t \le s \le T})$ and for the $\{\mathcal{F}_s\}_{t \le s \le T}$-adapted  Brownian motion ${\bf W}(s)$, the stochastic differential equation 
\begin{align*} 
{\rm d} {\bf m}^*(s) ={\bf f}({\bf m}^*(s), \bar{\bf u}(s,{\bf m}^*(s)))\,{\rm d}s + \nu \, \pmb{\sigma}({\bf m}^*(s)) \circ \,{\rm d}\mathbf{W}^*(s)\qquad s\in (t, T]\,, \quad 
{\bf m}^*(t) = \pmb{m}\, 
\end{align*}
has a pathwise unique $\mathcal{M}$-valued solution. 
 Then the process $$ {\bf u}^*=\mathbf{u}^*(s) := \bigl\{\bar{\mathbf{u}}(s, {\bf m}^*(s));\, t \le s \le T\bigr\}$$ belongs to $ L^2_{\{{\mathcal F}_{s}\}}\bigl( \Omega; L^2( t,T; (\R^3)^N)\bigr)$. Thus $({\bf m}^*, {\bf u}^*) \in \mathcal{U}_{\pmb{m}}^{\pmb{s}}[t,T]$. 
With this admissible pair $({\bf m}^*, {\bf u}^*)$, the inequality in \eqref{verification_identity} turns into equality. Again, by using Dynkin's formula along with initial data ${\bf m}^*(t)=\pmb{m}$, we see that 
\begin{align*}
 W(t,\pmb{m})= \mathbb{E}_{t,\pmb{m}} \Bigl[ \int_t^T L\big({\bf m}^*(r), \mathbf{u}^*(r)\big) \,{\rm d}r + h\big({\bf m}^*(T)\big) \Bigr]= \mathcal{J}\big(t,\, \pmb{m}; ({\bf m}^*, {\bf u}^*)\big)= V(t,\pmb{m})\, .
\end{align*}
Recalling that by Theorem \ref{thm:equiweak} the HJB equation \eqref{HJB} has a unique solution $W$, we conclude from the above that $V = W$. 
\end{proof}

Now we show the uniqueness of the optimal admissible pair $({\bf m}^*, {\bf u}^*)$, and thus in particular of the strong solution of Problem~\ref{problem-1}. We remark that the uniqueness of the optimal admissible pair is not automatically provided from the uniqueness of solutions of the HJB equation, which instead was used to characterize
the value function through the differential equation~\eqref{SLLG}. 

\begin{thm}\label{thm:uniqueness-optimal-pair}
The pair $({\bf m}^*, {\bf u}^*) \in \mathcal{U}_{\pmb{m}}^{\pmb{s}}[t,T]$ constructed in Theorem \ref{thm:verification} is the unique minimizer of $\mathcal{J}\big(t,\pmb{m}; (\cdot,\cdot)\big)$ in the sense that if there exists any other optimal pair $({\bf m}_1^*, {\bf u}_1^*) \in \mathcal{U}_{\pmb{m}}^{\pmb{s}}[t,T]$, then $
{\bf m}^*(s)= {\bf m}_1^*(s)$ and ${\bf u}^*(s)={\bf u}_1^*(s)$ for a.e.\,$s\in [t,T],\,\mathbb{P}$-a.s.
\end{thm}
\begin{proof}
\green{\noindent{\em Step 1:} Let $({\bf m}^*, {\bf u}^*) \in \mathcal{U}_{\pmb{m}}^{\pmb{s}}[t,T]$ be an optimal pair. Similar to \cite[Chapter~$5$, Theorem~$5.1$]{yong-zhou}, we note that then \eqref{eq:Dynkin_verification} holds as equality with $({\bf m}^*, {\bf u}^*)$ in place of $({\bf m}, {\bf u})$. This implies that also \eqref{verification_identity} holds with equality for a.e.~$s\in [t,T]$, $\mathbb{P}$-a.s. Hence $({\bf m}^*, {\bf u}^*)$ and thus every optimal pair satisfies
\begin{align}\label{eq:optimal-control-uniqueness}
{\bf u}^*(s) &=\frac{C_{\tt ext}}{\lambda}\Big({\bf m}^*(s) \times \nabla_\mathcal{M} V(s,{\bf m}^*(s)) - \alpha\, \nabla_\mathcal{M} V(s,{\bf m}^*(s))\Big)\,,
\end{align}
for a.e.~$s\in [t,T]$, $\mathbb{P}$-a.s.}
\vspace{.1cm}

\green{\noindent{\em Step 2:}} Suppose there exists another optimal pair $({\bf m}_1^*, {\bf u}_1^*) \in \mathcal{U}_{\pmb{m}}^{\pmb{s}}[t,T]$ and let $\widetilde{\bf m}^*= {\bf m}^* - {\bf m}_1^*$. Then $\widetilde{\bf m}^*$ is a strong solution to the following SDE: for  $s\in (t,T]$
\begin{align}
{\rm d}\widetilde{\bf m}^*(s) &= \Big\{ \big(\pmb{\Sigma}({\bf m}^*(s))-\pmb{\Sigma}({\bf m}_1^*(s))\big)\big[ -{\bf Q}{\bf m}_1^*(s) + C_{\tt ext}{\bf u}_1^*(s)\big] + \pmb{\Sigma}({\bf m}^*(s))\big[ -{\bf Q}\widetilde{\bf m}^*(s) \notag \\
 &\hspace{1cm}+ C_{\tt ext} \big({\bf u}^*(s)-{\bf u}_1^*(s)\big)\big] \Big\}\,{\rm d}s
+ \nu\,\big( \pmb{\sigma}({\bf m}^*(s))-\pmb{\sigma}({\bf m}_1^*(s)) \big) \circ \,{\rm d} {\bf W}(s) \label{eq:sde-for-uni-optimal}
\end{align} 
with $\widetilde{\bf m}^*(t)={\bf 0}$. In view of \eqref{defi:sigma}, we observe that
\begin{align*}
\pmb{\Sigma}({\bf m}^*(s))-\pmb{\Sigma}({\bf m}_1^*(s))&= \big( \pmb{\sigma}({\bf m}^*(s))-\pmb{\sigma}({\bf m}_1^*(s)) \big) - \alpha\, \pmb{\sigma}({\bf m}^*(s)) \big( \pmb{\sigma}({\bf m}^*(s))-\pmb{\sigma}({\bf m}_1^*(s)) \big) \\
& \hspace{1cm} -\alpha\,\big(\pmb{\sigma}({\bf m}^*(s))- \pmb{\sigma}({\bf m}_1^*(s))\big)  \pmb{\sigma}({\bf m}_1^*(s)) \\
& = \big[{\bf Id}- \alpha\,  \pmb{\sigma}({\bf m}^*(s))\big] \pmb{\sigma}(\widetilde{\bf m}^*(s)) -\alpha \, \pmb{\sigma}(\widetilde{\bf m}^*(s)) \pmb{\sigma}({\bf m}_1^*(s))\,.
\end{align*}
Thus, the equation \eqref{eq:sde-for-uni-optimal} reduces to
\begin{align*}
{\rm d}\widetilde{\bf m}^*(s) &= \Big[\pmb{\Sigma}({\bf m}^*(s))\big[ -{\bf Q}\widetilde{\bf m}^*(s) + C_{\tt ext} \big({\bf u}^*(s)-{\bf u}_1^*(s)\big)\big] +  \Big\{\big[{\bf Id}- \alpha\,  \pmb{\sigma}({\bf m}^*(s))\big] \pmb{\sigma}(\widetilde{\bf m}^*(s)) \notag \\
& \quad-\alpha \, \pmb{\sigma}(\widetilde{\bf m}^*(s)) \pmb{\sigma}({\bf m}_1^*(s))\Big\}\big( -{\bf Q}{\bf m}_1^*(s) + C_{\tt ext} {\bf u}_1^*(s)\big)\Big]\,{\rm d}s + \nu\, \pmb{\sigma}(\widetilde{\bf m}^*(s)) \circ \,{\rm d} {\bf W}(s)\,, \\
\widetilde{\bf m}^*(t)&={\bf 0}\,.
\end{align*} 
We now apply It\^{o}'s formula to the functional ${\bf x}\mapsto \|{\bf x}\|_{(\R^3)^N}^2$ for the above equation, and then use \eqref{eq:optimal-control-uniqueness} to have
\begin{align*}
\mathbb{E}\big[\|\widetilde{\bf m}^*(s)\|_{(\R^3)^N}^2\big] \le  C\int_t^s \mathbb{E}\big[\|\widetilde{\bf m}^*(r)\|_{(\R^3)^N}^2\big]\,{\rm d}r
\end{align*}
for some constant $C>0$. An application of Gronwall's lemma then implies $\mathbb{P}$-a.s., 
${\bf m}^*(s)={\bf m}_1^*(s)$, and therefore from \eqref{eq:optimal-control-uniqueness} we get ${\bf u}^*(s)={\bf u}_1^*(s)$ for a.e.\,$s\in [t,T]$. Thus, the optimal control problem admits a unique strong solution, which is an improvement over \cite{dunst_prohl}.
\end{proof}

\begin{rem}\label{rem:optimal-pair-orthogonality}
In \cite{dunst_prohl}, $\mathbb{P}$-a.s.~the orthogonality of an optimal state and control was shown both theoretically and numerically. In our approach, we also have $\mathbb{P}$-a.s.~the orthogonality of ${\bf m}^*$ and ${\bf u}^*$. Indeed, by using the vector identity 
$\langle {\bf a}, {\bf a} \times {\bf b}\rangle =0$ for any ${\bf a}, {\bf b}\in (\R^3)^N$, and the fact that ${\pmb m}\in \mathcal{M}$ and $\nabla_{\mathcal{M}}W(\cdot,\pmb{m})$ are orthogonal, we have $\mathbb{P}$-a.s., from \eqref{eq:control-feedback}
\begin{align*}
\langle {\bf m}^*(s), {\bf u}^*(s)\rangle= \frac{C_{\tt ext}}{\lambda} \big\langle {\bf m}^*(s),\, {\bf m}^*(s) \times \nabla_\mathcal{M} W(s,{\bf m}^*(s)) - \alpha \nabla_\mathcal{M} W(s,{\bf m}^*(s))\big\rangle=0\,.
\end{align*}
For its computational evidence, see Figure~\ref{fig:state-control-3-spin} for an ensemble of $N=3$ particles. 

\end{rem}


\section{Probabilistic representation of the value function} \label{sec:probRep}
To solve the linear PDE \eqref{lHJB} numerically by deterministic methods is still demanding since it is posed on $\mathcal{M}\subset (\R^3)^N$. Therefore, we choose a probabilistic representation of the solution of \eqref{lHJB} which requires to solve
the following forward stochastic differential equation, defined on a given stochastic basis $(\Omega, \mathbb{P}, \mathcal{F},\{\mathcal{F}_s\}_{t \le s \le T})$ with $3N$-dimensional Brownian motion $\mathbf{W}$,

\begin{equation} \label{eq:auxiliary-sllg}
\begin{aligned}
{\rm d}\pmb{\mathfrak m}(s)&= -{\bf b}\big(\pmb{\mathfrak m}(s)\big)\,{\rm d}s + \nu\, \pmb{\sigma}\big(\pmb{\mathfrak m}(s)\big)  \circ \,{\rm d} {\bf W}(s) \quad t< s\le T\,, \\
\pmb{\mathfrak m}(t)&= \pmb{m}\in \mathcal{M}\,,
\end{aligned}
\end{equation}
where ${\bf b}(\cdot)$ is defined in \eqref{defi:b}.
Equation \eqref{eq:auxiliary-sllg} has a strong solution $\pmb{\mathfrak m}$ taking values in $\mathcal{M}$. 
Let $G(s)= \displaystyle \exp\Big(-\int_t^s R\big(r, \pmb{\mathfrak m}(r)\big)\,{\rm d}r\Big)$
where $\displaystyle R(r,\pmb{\mathfrak m}(r))= \beta\,\delta \| \pmb{\mathfrak m}(r) - \widetilde {\bf m}(r) \|_{(\R^3)^N}^2$. By using the It\^{o} product rule applied to $G(s)\,w\big(s, \pmb{\mathfrak m}(s)\big)$, where $w$ is the classical solution of the linear parabolic PDE \eqref{lHJB}, 
we arrive at the following Feynman-Kac representation ~\cite[Theorem $7.6$]{karatzas}
for the solution of \eqref{lHJB} with terminal datum $ \exp\big( -\beta\, h(\pmb{m})\big)$: 
\begin{align}
w(t,\pmb{m})= \mathbb{E}_{t,\pmb{m}}\Big[ \exp\big( -\beta\, h(\pmb{\mathfrak m}(T))\big) \exp\Big(- \frac{\delta\,C_{\tt ext}^2(1+\alpha^2)}{\lambda \nu^2} \int_t^T \|\pmb{\mathfrak m}(r)-\widetilde {\bf m}(r) \|_{(\R^3)^N}^2\,{\rm d}r\Big) \Big]\,. \label{eq:prob-represnt-lhjb}
\end{align}
\green{We note that because of the linearity of \eqref{lHJB} the Feynman-Kac representation can be used in place of a backward SDE.}

\subsection{ A numerical scheme for \eqref{eq:auxiliary-sllg}}\label{subsec:eq:auxiliary-sllg}
  To approximate the solution $\pmb{\mathfrak m}$ of \eqref{eq:auxiliary-sllg}, we use the semi-implicit method proposed in \cite{mentink}. Now ${\bf b}(\pmb{m})$ from \eqref{defi:b} can be re-written as
\begin{align*}
{\bf b}(\pmb{m})= \pmb{m} \times {\bf Q} \pmb{m} - \alpha \, \pmb{m} \times \big(  \pmb{m} \times {\bf Q}\pmb{m}\big)\,,
\end{align*}
where ${\bf Q}$ is defined in \eqref{eq:Q}.
Let $T>0$ be fixed. For $J\in \mathbb{N}$, let $I_J^0:= \{ t_j\}_{j=0}^J$ be a partition of $[0,T]$ with time step size $ \displaystyle \tau= \green{T/J} >0$.  Let  $I_J^{\ell} \subset I_J^0$ be the sub-partition on $[t_{\ell},T]$, where $\ell \in \{0,1,\ldots, J-1 \}$. Let $\{\pmb{\xi}^j\}_{j=\ell}^J$ is a $({\mathbb R}^{3})^N$-valued random walk of $I_J^{\ell}$ with $\pmb{\xi}^j := (\xi^j_1, \ldots, \xi^j_{N})$, where~($1\le i\le N$)  $\xi^j_i=\big( \xi^j_{i,l}\big)_{1\le l\le 3}$ are i.i.d. ${\mathbb R}^3$-valued~(discrete) random variables such that each
 \begin{itemize}
 \item [i)] ${\xi}^j_{i,l}$ satisfies $\mathbb{E}\big[{\xi}^j_{i,l}\big]=0$ and 
 $\mathbb{E}\big[\big|{\xi}^j_{i,l}\big|^2\big|\big]=\tau$,
 \item [ii)] for every integer $p\ge 1$, there exists $C_p>0$ such that $\mathbb{E}\big[ |{\xi}^j_{i,l}|^{2p}\big] \le C_p \tau^p$.
 \end{itemize} Let now $\ell$ be fixed and $\pmb{\mathfrak m}^{\ell}= \pmb{m}\in \mathcal{M}$ be given. We determine the 
$\mathcal{M}$-valued random variables $\{\pmb{\mathfrak m}^j \}_{j=\ell}^J$ via  $\big(i=1,2,\ldots, N\big)$
\begin{subequations}
 \begin{align}
 {\mathfrak e}_i^j &= {\mathfrak m}_i^j + \tau\, \frac{{\mathfrak e}_i^j + {\mathfrak m}_i^j}{2} \times \bar{\tt a}_i\big(\pmb{\mathfrak m}^j\big) + \nu\, \frac{{\mathfrak e}_i^j + {\mathfrak m}_i^j}{2} \times \xi_i^j, \label{eq:scheme-auxi-prob-sde-0-intermediate}  \\
 {\mathfrak m}_i^{j+1}&= {\mathfrak m}_i^j + \tau \,\frac{{\mathfrak m}_i^{j+1} + {\mathfrak m}_i^j}{2} \times \bar{\tt a}_i\big(\frac{\pmb{\mathfrak m}^j + \pmb{\mathfrak e}^j}{2}\big)  + \nu\, \frac{{\mathfrak m}_i^{j+1} + {\mathfrak m}_i^j}{2} \times \xi_i^j\,,  \label{eq:scheme-auxi-prob-sde-0-iteration}
 \end{align}
 \end{subequations}
 where $\pmb{\mathfrak e}^j=\big({\mathfrak e}_1^j, {\mathfrak e}_2^j,\ldots, {\mathfrak e}_N^j\big)$, and the function $\bar{\tt a}_i$ for each $i=1,2,\ldots, N$ is given by 
 \begin{align*}
 \bar{\tt a}_i(\pmb{\mathfrak m}):= -({\bf Q}\pmb{\mathfrak m})_i + \alpha\, {\mathfrak m}_i \times ({\bf Q}\pmb{\mathfrak m})_i\,.
 \end{align*}

Note that \eqref{eq:scheme-auxi-prob-sde-0-intermediate}-\eqref{eq:scheme-auxi-prob-sde-0-iteration} is a system of linear equations, leading to short simulation times. \green{Furthermore, the numerical schemes ensures $\pmb{\mathfrak m}^j$ takes values on $\mathcal{M}$. This is exploited when applying arguments from \cite{banas_brzezniak_neklyudov_prohl} in this finite ensemble setting} to conclude convergence of iterates $\{\pmb{\mathfrak m}^j \}_{j=\ell +1}^J$ towards a weak solution of \eqref{eq:auxiliary-sllg} for $\tau \rightarrow 0$.

\section{An algorithm to approximate an optimal pair $({\bf m}^*, {\bf u}^*)$}\label{sec:numerics}
In order to simulate the optimal pair $({\bf m^*}, {\bf u}^*)$, we need to solve the equations 
\eqref{eq:auxiliary-sllg}, \eqref{eq:control-feedback} and \eqref{SLLG} with the function $\bar{\bf u}(s,\cdot)$ numerically. 

\subsection{HJB solution}\label{subsec:numeric-hjb}
 The classical solution $w$ of \eqref{lHJB} is given by \eqref{eq:prob-represnt-lhjb}. In order to approximate it, and hence the classical solution $W$ of the nonlinear HJB equation \eqref{HJB}, we proceed as follows:
 \begin{itemize}
 \item [a)] Compute all the iterates $\{\pmb{\mathfrak m}^j\}_{j=\ell}^J$ via \eqref{eq:scheme-auxi-prob-sde-0-intermediate}-\eqref{eq:scheme-auxi-prob-sde-0-iteration} along $I_J^{\ell}$ and store them.
 \item [b)] Approximate the integral in~\eqref{eq:prob-represnt-lhjb} by Gauss-Legendre quadrature ~\cite[Section $10.3$]{krylov}, where we use the piecewise affine interpolation of the iterates $\{\pmb{\mathfrak m}^j\}_{j=\ell}^J$ via
 \begin{align}
 \pmb{\mathfrak m}(r):= \frac{r-t_j}{\tau} \pmb{\mathfrak m}^{j+1} + \frac{t_{j+1}-r}{\tau}\pmb{\mathfrak m}^j \qquad  \big(r\in [t_j, t_{j+1})\big)\,.\label{eq:intermideate-point}
 \end{align}
 \item[c)]Since $\widetilde{\bf m}\in C^2([t_{\ell},T]; \mathcal{M})$, we use the  piecewise affine interpolation $\widetilde{\bf m}(r)$ of the iterates
 $\widetilde{\bf m}^j\equiv \widetilde{\bf m}(t_j)$.
 \item[d)] Approximate $\mathbb{E}_{t,\pmb{m}}$ in \eqref{eq:prob-represnt-lhjb}
 via Monte-Carlo estimation along with the variance reduction method of {\em antithetic variates}~(see e.g.~\cite[Subsection $4.2$]{glasserman}).
 \end{itemize}
 Thus, we can simulate the quantities $w(t_{\ell}, \pmb{\mathfrak m}^\ell)$, and hence 
 $W(t_{\ell}, \pmb{\mathfrak m}^\ell)= \displaystyle -\frac{1}{\beta} \log(w(t_{\ell}, \pmb{\mathfrak m}^\ell))$.
 \subsection{Optimal feedback transformation}\label{subsec:method-a-b} To approximate the function $\bar{\bf u}$ at any point $(t,\pmb{m})$, we need to approximate $\nabla_\mathcal{M} W(t,\pmb{m})$,  which again demands to
 approximate $\nabla_\mathcal{M} w(t,\pmb{m})$ thanks to the Hopf-Cole transformation. For the latter, we may proceed in two different ways:
 \begin{itemize}
  \item [i)]{\bf Method A}: We take the expectation first and then use the central difference quotient to approximate the tangential gradient. More precisely, for any $\pmb{m}=\big({\it m}_1,{\it m}_2,\ldots, {\it m}_N\big)\in \mathcal{M}$ and $\bar{h}>0$, define for $i=1,2,\ldots, N$,
  \begin{equation}\label{eq:diff-quotient-point}
   \begin{aligned}
   \pmb{m}_{\bar{h},i,l}^+: = \Big( {\it m}_1, \ldots, \frac{{\it m}_i + \bar{h} {\bf e}_l}{\|{\it m}_i + \bar{h} {\bf e}_l\|_{\R^3}}, \ldots, {\it m}_N\Big)\,,\\ \pmb{m}_{\bar{h},i,l}^-: = \Big( {\it m}_1, \ldots, \frac{{\it m}_i - \bar{h}  {\bf e}_l}{\|{\it m}_i - \bar{h}{\bf e}_l\|_{\R^3}}, \ldots, {\it m}_N\Big)\,,
   \end{aligned} 
   \end{equation}
   where ${\bf e}_l$ is the $l$-th identity vector in $\R^3,\, 1\le l\le 3$. 
   Recall that, by using Subsection~\ref{subsec:numeric-hjb}, we can calculate $w(t,  \pmb{m}_{\bar{h},i,l}^+)$ and $w(t,  \pmb{m}_{\bar{h},i,l}^-)$ for any $t\in [0,T]$;
   therefore, we approximate $\nabla_\mathcal{M} w(t,\pmb{m})$ by $
    \nabla_\mathcal{M} w(t,\pmb{m})\simeq \Big( \partial_1 w(t,\pmb{m}),\, \ldots,\, \partial_N w(t,\pmb{m})\Big)$, where for any $i=1,2,\ldots, N$, and $l=1,2,3$,
    \begin{align*}
    &\partial_i w(t,\pmb{m}):= \Big(d_{1,i} w(t, \pmb{m}),\, d_{2,i} w(t, \pmb{m}),\, d_{3,i}w(t, \pmb{m})\Big)^\top\,, \\
   & d_{l,i} w(t,\pmb{m}):= \frac{1}{2\bar{h}}\Big[ w(t,\pmb{m}_{\bar{h},i,l}^+ )- w(t, \pmb{m}_{\bar{h},i,l}^-)\Big]\,.
   \end{align*} 
   Hence, we approximate $\nabla_{\mathcal{M}}W(t_{\ell},\pmb{\mathfrak m}^\ell)$ by
   \begin{align}
       \nabla_\mathcal{M} W(t_{\ell},\pmb{\mathfrak m}^{\ell})\simeq -\frac{1}{\beta w(t_{\ell},\pmb{\mathfrak m}^{\ell})}\Big( \partial_1 w(t_{\ell},\pmb{\mathfrak m}^{\ell}),\, \ldots,\, \partial_N w(t_{\ell},\pmb{\mathfrak m}^{\ell})\Big)\,.\label{aprox:method-a-gradient}
   \end{align}
  \item [ii)]{\bf Method B:} In contrast to {\bf Method A},  we first use the central difference quotient and then take the expectation to approximate the gradient $\nabla_\mathcal{M} w$. 
   For all $(t, \pmb{m})\in [0,T)\times \mathcal{M}$, we define \green{the random variable}
   \begin{align}
   H(t,{\pmb{m}}):= \exp\big( -\beta\, h(\pmb{\mathfrak m}(T))\big)\exp\Big(- \beta\,\delta \int_t^T \|\pmb{\mathfrak m}(r)-\widetilde {\bf m}(r) \|_{(\R^3)^N}^2\,{\rm d}r\Big)\,, \label{eq:integrand-method-B}
   \end{align}
   where $\pmb{\mathfrak m}$ solves \eqref{eq:auxiliary-sllg} with $\pmb{\mathfrak m}(t)=\pmb{m}$. Let $\pmb{m}_{\bar{h},i,l}^+$ and $\pmb{m}_{\bar{h},i,l}^-$ be the points in $\mathcal{M}$ as defined above. We compute the central difference quotients component-wise, and use \eqref{eq:scheme-auxi-prob-sde-0-intermediate}-\eqref{eq:scheme-auxi-prob-sde-0-iteration} \green{to approximate related solutions from \eqref{eq:auxiliary-sllg} in \eqref{eq:integrand-method-B} for $i=1,2,\ldots, N$, and $l=1,2,3$,
    \begin{equation}\label{aprox:method-b-gradient-0}
    \begin{aligned}
    d_{l,i} H(t,\pmb{m}) &:= \frac{1}{2\bar{h}}\Big[ H(t,\pmb{m}_{\bar{h},i,l}^+ )- H(t, \pmb{m}_{\bar{h},i,l}^-)\Big] \,, \\
    d_i H(t,\pmb{m}) &:= \Big(\mathbb{E}\big[ d_{1,i}H(t, \pmb{m})\big],\, \mathbb{E}\big[ d_{2,i}H(t, \pmb{m})\big],\, \mathbb{E}\big[ d_{3,i}H(t, \pmb{m})\big]\Big)^\top\,.
    \end{aligned}
    \end{equation}
    so that $\partial_i w(t,\pmb{m}) \simeq d_i H(t,\pmb{m})$.} We approximate the expectation in \eqref{aprox:method-b-gradient-0} via Monte-Carlo estimation together with the method of {\em antithetic variates}. We then approximate $\nabla_\mathcal{M} w(t,\pmb{m})$~(hence $\nabla_\mathcal{M} W(t,\pmb{m})$) as
    \begin{align*}
    & \nabla_\mathcal{M} w(t,\pmb{m})\simeq \Big( d_1 H(t,\pmb{m}),\, \ldots,\, d_N H(t,\pmb{m})\Big)^\top\,, \\
 &     \nabla_\mathcal{M} W(t,\pmb{m})\simeq  -\frac{1}{\beta w(t,\pmb{m})}\Big( d_1 H(t,\pmb{m}),\, \ldots,\, d_N H(t,\pmb{m})\Big)^\top\,.
    \end{align*}
 \end{itemize}
 Thus, we simulate the transformation function $\bar{\bf u}(t_{\ell}, \pmb{\mathfrak m}^\ell)$ from \eqref{eq:control-feedback} by using one of the above methods, where the sequence $(t_{\ell}, \pmb{\mathfrak m}^\ell)$ is described in Subsection~\ref{subsec:eq:auxiliary-sllg}.
 
 \subsection{Optimal state}\label{subsec:optimal-state-numeric}
  We use again the semi-implicit method proposed in \cite{mentink} to approximate the solution in \eqref{SLLG} in which each realization takes values in ${\mathcal M}$. For any $\pmb{m}=({\it m}_1,\ldots, {\it m}_N)\in \mathcal{M}$, and $i=1,2,\ldots,N$, define 
 \begin{align}
 {\rm a}_{i}\big(\pmb{m}, \bar{\bf u}(t, \pmb{m})\big) &:= -({\bf Q}\pmb{m})_i+ C_{\tt ext}\bar{\bf u}_i(t, \pmb{m}) - \alpha\, {\it m}_i \times \big[ - ({\bf Q}\pmb{m})_i+ C_{\tt ext}\bar{\bf u}_i(t, \pmb{m})\big]\,, \label{defi:mentink-drift}
 \end{align}
 where ${\bf Q}$ is defined in \eqref{eq:Q}. We use again the scheme \eqref{eq:scheme-auxi-prob-sde-0-intermediate}-\eqref{eq:scheme-auxi-prob-sde-0-iteration}, and Subsection~\ref{subsec:method-a-b} to find a $\mathcal{M}$-valued random variables $\{{\bf m}^j \}_{j=\ell}^J$ along $I_J^{\ell}$ with ${\bf m}^{\ell}= \pmb{m}\in \mathcal{M}$ and $\ell \in \{0,1,\ldots, J-1 \}$, where $\bar{\tt a}_i\big(\pmb{\mathfrak m}^j\big)$ resp. $\bar{\tt a}_i\big(\frac{\pmb{\mathfrak m}^j + \pmb{\mathfrak e}^j}{2}\big)$ in \eqref{eq:scheme-auxi-prob-sde-0-intermediate} resp.  \eqref{eq:scheme-auxi-prob-sde-0-iteration} is replaced by ${\rm a}_i\big({\bf m}^j, \bar{\bf u}(t_j, {\bf m}^j)\big)$ resp. 
\[
\displaystyle {\rm a}_i\Big(\frac{{\bf m}^j + \pmb{\rm e}^j}{2}, \bar{\bf u}\big(t_j, \big( \frac{g_1^j}{\|g_1^j\|_{\R^3}}, \ldots, \frac{g_N^j}{\|g_N^j\|_{\R^3}}\big)\big)\Big) \text{ with } \displaystyle g_i^j:=\frac{{\rm e}_i^j + {m}_i^j}{2} \quad ~(1 \leq i \leq N),
\]
such that the iterates $\{{\bf m}^j\}_{j=\ell}^J$ converge towards a weak solution of \eqref{SLLG} for $\tau \goto 0$.
 Moreover, the iterates $\{{\bf u}^j=\bar{\bf u}(t_j, {\bf m}^j)\}_{j={\ell}}^J$ defines the discrete optimal control along $I_J^\ell$.
 \vspace{.2cm}
 
 In summary, we have the following algorithm to compute the optimal solution and control along with {\bf Method B}.
 
 \begin{alg}\label{algoHJB}
 	Let ${\bf m}^{0} \in (\mathbb{S}^{2})^{N}$, $T > 0$,
 	$M \in \mathbb{N}$ be given. For $J \in \mathbb{N}$,
 	let $I_{J}^{0} := \{ t_{j} \}_{j=0}^{J}$ be a
 	partition of $[0,T]$ with time step size
 	$\tau = \frac{T}{J} > 0$. Denote by $I_{J}^{\ell} \subset I_{J}^{0}$ the sub-partition
 	on $[t_{\ell},T]$, where $\ell \in \{ 0,1,\ldots,J-1 \}$.
 	Let now $\ell$ be fixed and $\mathbf{m}^{\ell} = \pmb{m}
 	\in \mathcal{M}$ be given.
 	\begin{itemize}
\item[(I)]
 Compute $M$-samples $\mathcal{S}_{\pmb{\xi}}^{\ell} := \{ \mathcal{S}_{\pmb{\xi}}^{\ell,k} \}_{k=1}^{M}$, $\mathcal{S}_{\pmb{\xi}}^{\ell,k} := \{ \pmb{\xi}^{j}(\omega_{k}) \}_{j=\ell}^{J}$ on $I_{J}^{\ell}$.
 \item[(II)]
For $i = 1,\ldots,N$ do:
\par
For $l = 1,2,3$ do:
\begin{itemize}[leftmargin=3\parindent]
\item[(1)] Based on $\mathbf{m}^{\ell}=\pmb{m}$, compute $\pmb{m}_{\bar{h},i,l}^{+}$
and $\pmb{m}_{\bar{h},i,l}^{-}$ as in \eqref{eq:diff-quotient-point}.
\item[(2)] For $k = 1,\ldots,M$ do:
\begin{itemize}
\item[(a)] Compute $\mathcal{S}_{\pmb{\mathfrak m}}^{+,\ell,k} := \{\pmb{\mathfrak m}^{j}(\omega_{k})\}_{j=\ell}^{J}$ resp.\, $\mathcal{S}_{\pmb{\mathfrak m}}^{-,\ell,k}
 := \{\pmb{ \mathfrak m}^{j}(\omega_{k}) \}_{j=\ell}^{J}$ on $I_{J}^{\ell}$ via
scheme~\eqref{eq:scheme-auxi-prob-sde-0-intermediate}-\eqref{eq:scheme-auxi-prob-sde-0-iteration} for $\bar{\mathtt{a}}_{i}(\pmb{\mathfrak m}^{j})$ and $\bar{\tt a}_i\big(\frac{\pmb{\mathfrak m}^j + \pmb{\mathfrak e}^j}{2}\big)$ using
$\{ +\pmb{\xi}^{j}(\omega_{k}) \}_{j=\ell}^{J}$ resp.\,
$\{ -\pmb{\xi}^{j}(\omega_{k}) \}_{j=\ell}^{J}$.
\item[(b)] Compute $H(t_{\ell}, \pmb{m}_{\bar{h},i,l}^{+},\omega_{k})$ resp.\,
$H(t_{\ell},\pmb{m}_{\bar{h},i,l}^{-},\omega_{k})$ in~\eqref{eq:integrand-method-B} based on $\mathcal{S}_{\pmb{\mathfrak m}}^{+,\ell,k}$ resp.\,   $\mathcal{S}_{\pmb{\mathfrak m}}^{-,\ell,k}$ to determine $d_{l,i}H(t_{\ell},\pmb{m},\omega_{k})$
in~\eqref{aprox:method-b-gradient-0} and store it.
\end{itemize}
\item[(3)] Approximate $\mathbb{E}[d_{l,i}H(t_{\ell},\pmb{m})]$
in~\eqref{aprox:method-b-gradient-0} via Monte-Carlo estimation along with the variance reduction method of antithetic variates.
\end{itemize}
\item[(III)] Set $\nabla_{\mathcal{M}} w(t_{\ell},\pmb{m}) \approx \bigl( d_{1}H(t_{\ell},\pmb{m}),\ldots,d_{N}H(t_{\ell},\pmb{m}) \bigr)^{\top}$
with $d_{i}H(t_{\ell},\pmb{m})$ as in~\eqref{aprox:method-b-gradient-0}, and compute 
$\bar{\mathbf{u}}(t_{\ell},{\bf m}^\ell)$ as in~\eqref{eq:control-feedback}.
\item[(IV)] Compute $\mathbf{e}^\ell \in (\mathbb{R}^{3})^{N}$
via Scheme~\eqref{eq:scheme-auxi-prob-sde-0-intermediate} with ${\rm a}_{i}\bigl(
\mathbf{m}^{\ell}, \bar{\mathbf{u}}(t_{\ell},{\bf m}^{\ell}) \bigr)$ defined in \eqref{defi:mentink-drift}.
\item[(V)] Define $\mathbf{g}^{\ell} := \left( \frac{g_{1}^{\ell}}{{\| g_{1}^{\ell} 
\|}_{\mathbb{R}^{3}}},\ldots, \frac{g_{N}^{\ell}}{{\| g_{N}^{\ell}  \|}_{\mathbb{R}^{3}}}
\right)^{\top}$ with $g_{i}^{\ell} := \frac{m_{i}^{\ell}+e_{i}^{\ell}}{2}\,\,(1\le i\le N)$.
\item[(VI)] 
For $i = 1,\ldots,N$ do:
\par
For $l= 1,2,3$ do:
\begin{itemize}[leftmargin=3\parindent]
\item[(4)] Based on $\mathbf{g}^{\ell}:=\pmb{g}$, compute $\pmb{g}_{\bar{h},i,l}^{+}$
and $\pmb{g}_{\bar{h},i,l}^{-}$ as in \eqref{eq:diff-quotient-point}.
\item[(5)] For $k = 1,\ldots,M$ do:
 \begin{itemize}
 \item[(c)] Compute $\mathcal{S}_{\pmb{\mathfrak m}}^{+,\ell,k}
:= \{\pmb{\mathfrak m}^{j}(\omega_{k})\}_{j=\ell}^{J}$ resp.\,						
$\mathcal{S}_{\pmb{\mathfrak m}}^{-,\ell,k}:= \{\pmb{
\mathfrak m}^{j}(\omega_{k})\}_{j=\ell}^{J}$ on $I_{J}^{\ell}$ via
scheme~\eqref{eq:scheme-auxi-prob-sde-0-intermediate}-\eqref{eq:scheme-auxi-prob-sde-0-iteration} for $\bar{\mathtt{a}}_{i}(\pmb{\mathfrak m}^{j})$ and $\bar{\tt a}_i\big(\frac{\pmb{\mathfrak m}^j + \pmb{\mathfrak e}^j}{2}\big)$
using $\{ +\pmb{\xi}^{j}(\omega_{k}) \}_{j=\ell}^{J}$ resp.\,
$\{ -\pmb{\xi}^{j}(\omega_{k})\}_{j=\ell}^{J}$ with initial condition $\pmb{\mathfrak m}^{\ell}=\pmb{g}$.
\item[(d)] Compute $H(t_{\ell},\pmb{g}_{\bar{h},i,l}^{+},\omega_{k})$ resp.\,
$H(t_{\ell},\pmb{g}_{\bar{h},i,l}^{-},\omega_{k})$
in~\eqref{eq:integrand-method-B} based on $\mathcal{S}_{\pmb{\mathfrak m}}^{+,\ell,k}$
resp.\, $\mathcal{S}_{\pmb{\mathfrak m}}^{-,\ell,k}$ to determine
$d_{l,i}H(t_{\ell}, \pmb{g},\omega_{k})$ in~\eqref{aprox:method-b-gradient-0} and store it.
\end{itemize}
\item[(6)] Approximate $\mathbb{E}[d_{l,i}H(t_{\ell},\pmb{g})]$
in~\eqref{aprox:method-b-gradient-0} via Monte-Carlo estimation along with the variance reduction method of antithetic variates.
\end{itemize}
\item[(VII)] Set $\nabla_{\mathcal{M}} w(t_{\ell},\pmb{g}) \approx \bigl( d_{1}H(t_{\ell},\pmb{g}),\ldots,d_{N}H(t_{\ell},\pmb{g}) \bigr)^{\top}$
with $d_{i}H(t_{\ell},\pmb{g})$ as in~\eqref{aprox:method-b-gradient-0}, and compute 
$\bar{\mathbf{u}}(t_{\ell},{\bf g}^\ell)$ as in~\eqref{eq:control-feedback}, and 
$\mathbf{m}^{\ell+1}$ via scheme~\eqref{eq:scheme-auxi-prob-sde-0-iteration} with
${\rm a}_{i}\bigl(\mathbf{g}^{\ell},\bar{\mathbf{u}}(t_{\ell},{\bf g}^{\ell})\bigr)$.
\end{itemize}
\end{alg}
 
 \section{Computational Experiments}\label{sec:computational-experiments}
 We computationally study the behavior of the optimal state and control for the switching dynamics of an ensemble of $N$ particles by using the algorithm from Section~\ref{sec:numerics}. For this purpose,
 we employ discretely distributed random numbers from the GNU Scientific Library~\cite{galassi}. All computations are
 performed on an Intel Core i5-4670 3.40GHz processor with 16GB RAM in double precision arithmetic. The arising linear algebraic systems are solved by the Gaussian elimination method~\cite{galassi}.
 
 \subsection{Test studies} We start with some test problems to compare the two methods from Subsection~\ref{subsec:method-a-b}. For this purpose, we omit certain energy contributions in \eqref{eq:effective-field}, and allow only one or two spins such that an exact solution of \eqref{lHJB} becomes available. 
 \vspace{.2cm}

 \noindent{\bf Test problem~$1$:} Consider the controlled problem for a single spin (${\bf J}={\bf 0}$) of an isotropic material (${\bf D}={\bf 0}$), and $\delta=0$ in the cost functional; all other parameters are equal to $1$. Then \eqref{lHJB} is the backward heat equation 
 \begin{align} \label{eq:test-problem-1}
 \green{-\partial_t w(t,\pmb{m})- \frac{1}{2} \Delta_{\mathbb{S}^2} w(t,\pmb{m})=0\,.}
 \end{align} 
We shall use spherical harmonics to describe the exact solution of \eqref{eq:test-problem-1}. Note that for any $\pmb{m}= ({\it m}_1, {\it m}_2, {\it m}_3)\in \mathbb{S}^2$, the spherical harmonic $w_{0,1}(\pmb{m})= {\it m}_3$ is an eigen-function of the Laplace-Beltrami operator with eigenvalue $-2$ \cite[Lemma 4.3.26]{FG}, i.e., 
\begin{align*}
\Delta_{\mathbb{S}^2} w_{0,1}(\pmb{m}) + 2 w_{0,1}(\pmb{m})=0 \quad \forall \, \pmb{m}\in \mathbb{S}^2\,.
\end{align*}
 As $w= \exp(-W)$, we know that $w$ has to be positive, while $w_{0,1}$ may also take negative values. Therefore we use the constant spherical harmonic $w_{0,0}(\pmb{m}):=1$. Consider the problem \eqref{eq:test-problem-1} with final time condition $w(T,\pmb{m})=w_{0,1}(\pmb{m}) + 2$. We obtain this terminal condition by choosing the terminal payoff $h(\pmb{m})= - \frac{1}{2}\log\big(w_{0,1}(\pmb{m}) + 2\big)$. Then  the solution of \eqref{eq:test-problem-1} is
 \begin{align*}
 w(t,\pmb{m})= \exp(t-T) w_{0,1}(\pmb{m}) + 2\,.
 \end{align*}
 Moreover, we have the explicit formula for $\nabla_{\mathbb{S}^2} w(t,\pmb{m})$, and hence for $\nabla_{\mathbb{S}^2} W(t,\pmb{m})$:
 \begin{equation}\label{eq:gradient-test-problem-1}
 \begin{aligned}
 \nabla_{\mathbb{S}^2} w(t,\pmb{m})&= \exp(t-T)\big(-{\it m}_1{\it m}_3, -{\it m}_2{\it m}_3, 1-{\it m}_3^2\big)^\top\, , \\
  \nabla_{\mathbb{S}^2} W(t,\pmb{m})&= - \frac{1}{2}\frac{\exp(t-T)}{\exp(t-T) w_{0,1}(\pmb{m}) + 2}\big(-{\it m}_1 {\it m}_3, -{\it m}_2 {\it m}_3, 1-{\it m}_3^2\big)^\top\,.
 \end{aligned}
 \end{equation}
 Let $\bar{\bf u}_{\tt exct}(t,\pmb{m})$ resp. $\bar{\bf u}_{\tt app}(t,\pmb{m})$ be the function defined in \eqref{eq:control-feedback} associated to \eqref{eq:gradient-test-problem-1} resp. \eqref{eq:prob-represnt-lhjb}, and
 ${\bf m}_{\tt exct}^*(t)$ resp.  ${\bf m}_{\tt app}^*(t)$ be the solution of \eqref{SLLG} with $\bar{\bf u}_{\tt exct}\big(t,{\bf m}_{\tt exct}^*(t)\big)$
 resp. $\bar{\bf u}_{\tt app}\big(t,{\bf m}_{\tt app}^*(t)\big)$. By denoting the error
 \begin{align*}
  {\tt err}(t):= \|{\bf m}_{\tt exct}^*(t)-{\bf m}_{\tt app}^*(t)\|_{(\R^3)^N}^2\,,
  \end{align*}
   we show the behavior of ${\tt err}(t)$ for different values of Monte-Carlo realization $M$. For its simulation, we have used $T=0.5$, $\tau=10^{-2}$, $\bar{h}= \frac{1}{\sqrt{M}}$, $\bar{\bf m}= {\bf e}_1$, and other parameters as specified in the beginning of this subsection. We 
   observe that the error ${\tt err}(t)$ for {\bf Method B} is significantly smaller~(by a factor of $\frac{1}{20}$ in our simulations) if compared to {\bf Method A}, see Figure~\ref{fig:error-test-1}. Moreover, at least $M\approx 10^6$ 
   realizations are needed to balance the approximate computation via {\bf Method B}
   with remaining error sources.
   \begin{rem}\label{rem:h-M-parameters}
   Computational studies with respect to the parameters $\big(\tau,\bar{h},M\big)$ show that, independent of $\tau$, we have to choose $\bar{h} = \mathcal{O}\big(\frac{1}{\sqrt{M}}\big)$ to approximate the $\nabla_\mathcal{M} w$~(hence $\nabla_\mathcal{M} W$) accurately. For choice $\bar{h} \ll \mathcal{O}(\frac{1}{\sqrt{M}})$, irrespective of the {\bf Method A} or {\bf Method B}, we observe a strongly oscillatory behavior of the solution in case $M\le 10^4$. 
   \end{rem}

\noindent{\bf Test problem~$2$:} We study the interaction of two isotropic (${\bf D}={\bf 0}$) spins for $\alpha=0=\delta$, and all other parameters are equal to $1$. Let us first recall how the spherical harmonics on a single sphere $\mathbb{S}^2$ generalizes to the manifold $\mathcal{M}=(\mathbb{S}^2)^N= (\mathbb{S}^2)^2$ most naturally. Indeed, because $\mathcal{M}$ is a tensor product of spheres, and the spherical harmonics form an orthogonal basis on the \green{single} sphere, the tensor products of spherical harmonics form an orthogonal basis on $\mathcal{M}$. It is therefore reasonable to expect that the simplest meaningful test problems on $(\mathbb{S}^2)^2$ can be constructed with terminal time conditions which are products of low-order spherical harmonics and which are eigen-functions of the transformed Bellman equation. Because of the spin interaction, the first order coefficient ${\bf b}$ in \eqref{lHJB} does not vanish any more. Therefore, we combine the functions $w_{0,0}(\pmb{m}_1) w_{0,0}(\pmb{m}_2), w_{0,0}(\pmb{m}_1)w_{0,1}(\pmb{m}_2)$ and $w_{0,1}(\pmb{m}_1)w_{0,0}(\pmb{m}_2)$ with $\pmb{m}_1=({\it m}_{1,1},\, {\it m}_{1,2},\,{\it m}_{1,3})\in \mathbb{S}^2$ and $\pmb{ m}_2=({\it m}_{2,1},\, {\it m}_{2,2},\,{\it m}_{2,3})\in \mathbb{S}^2$ to \green{pose} a test problem on $(\mathbb{S}^2)^2$. Denoting by
\begin{align*}
w_{00,00}(\pmb{m}_1, \pmb{m}_2):=1\, ,\quad  w_{01,00}(\pmb{m}_1, \pmb{m}_2)= {\it m}_{1,3}\, , \quad \text{and} \quad
 w_{00,01}(\pmb{m}_1, \pmb{m}_2)= {\it m}_{2,3}\,,
\end{align*}
we consider the following version of \eqref{lHJB},
\begin{equation}\label{eq:test-problem-2}
\begin{aligned}
 -&\partial_t w\big(t,\pmb{m}_1, \pmb{m}_2\big) - \frac{1}{2}\Delta_{(\mathbb{S}^2)^2} w\big(t,\pmb{m}_1, \pmb{m}_2\big) + {\bf b}(\pmb{m}_1, \pmb{m}_2)\cdot \nabla_{(\mathbb{S}^2)^2} w\big(t,\pmb{m}_1, \pmb{m}_2\big)=0\,,  \\
& w\big(T, \pmb{m}_1, \pmb{m}_2\big)= w_{01,00}(\pmb{m}_1, \pmb{m}_2) +  w_{00,01}(\pmb{m}_1, \pmb{m}_2) + 2\,w_{00,00}(\pmb{m}_1, \pmb{m}_2)\,,
\end{aligned}
\end{equation}
with the positive \green{semi-}definite matrix ${\bf J}$: for any $\mu>0$,
\begin{align*}
{\bf J}=
\begin{pmatrix} 
\mu & 0 & 0 & -\mu & 0 & 0\\
0 & \mu & 0 & 0 & -\mu & 0\\ 
0& 0& \mu & 0 & 0& -\mu \\
-\mu & 0 & 0 & \mu & 0 & 0\\
0& -\mu & 0 & 0& \mu & 0 \\
0 & 0 & -\mu & 0& 0& \mu
\end{pmatrix}\,.
\end{align*}
Since $\alpha=0$, we have 
\begin{align*} 
{\bf b}(\pmb{m}_1, \pmb{m}_2) &= \Big( -\mu\big({\it m}_{1,2}{\it m}_{2,3}-{\it m}_{1,3}{\it m}_{2,2}\big),\, \mu\big({\it m}_{1,1}{\it m}_{2,3}-{\it m}_{1,3}{\it m}_{2,1}\big),\, -\mu\big({\it m}_{1,1}{\it m}_{2,2}-{\it m}_{1,2}{\it m}_{2,1}\big), \\
& \quad\, \mu\big({\it m}_{1,2}{\it m}_{2,3} -{\it m}_{1,3}{\it m}_{2,2}\big), 
 -\mu\big({\it m}_{1,1}{\it m}_{2,3} - {\it m}_{1,3}{\it m}_{2,1}\big),\, \mu\big({\it m}_{1,1}{\it m}_{2,2}-{\it m}_{1,2}{\it m}_{2,1}\big)\Big)^\top\,.
\end{align*}
We compute the tangential gradient of the functions $w_{00,00},\,w_{01,00}$ and 
$w_{00,01}$:
\begin{align*}
\nabla_{(\mathbb{S}^2)^2} w_{00,00}(\pmb{m}_1, \pmb{m}_2)&= \big( 0,\,0,\,0,\,0,\,0,\,0\big)^\top \, , \\
 \nabla_{(\mathbb{S}^2)^2} w_{01,00}(\pmb{m}_1, \pmb{m}_2)&= \big( -{\it m}_{1,1}{\it m}_{1,3},\,-{\it m}_{1,2}{\it m}_{1,3},\, 1-{\it m}_{1,3}^2, 0,0,0\big)^\top \,, \\
  \nabla_{(\mathbb{S}^2)^2} w_{00,01}(\pmb{m}_1, \pmb{m}_2)&= \big(0,\,0,\,0,\,-{\it m}_{2,1}{\it m}_{2,3},\,-{\it m}_{2,2}{\it m}_{2,3},\, 1-{\it m}_{2,3}^2\big)^\top \,.
\end{align*}
Observe that $${\bf b}(\pmb{m}_1, \pmb{m}_2)\cdot \nabla_{(\mathbb{S}^2)^2}[w_{01,00}(\pmb{m}_1, \pmb{m}_2) + w_{00,01}(\pmb{m}_1, \pmb{m}_2)]=0\, ,$$ and $w_{01,00}$ and $w_{00,01}$ are eigen-functions of the Laplace-Beltrami operator on $(\mathbb{S}^2)^2$ with eigenvalue $-2$. Thus the exact solution of \eqref{eq:test-problem-2} is given by 
\begin{align*}
w\big(t,\pmb{m}_1, \pmb{m}_2\big)= \exp(t-T)\Big\{w_{01,00}(\pmb{m}_1, \pmb{m}_2) + w_{00,01}(\pmb{m}_1, \pmb{m}_2)\Big \} + 2\,w_{00,00}(\pmb{m}_1, \pmb{m}_2).
\end{align*}
Moreover, we compute $\nabla_{(\mathbb{S}^2)^2} W\big(t, \pmb{m}_1, \pmb{m}_2\big)$ as 
\begin{align}
\nabla_{(\mathbb{S}^2)^2} W \big(t,\pmb{m}_1, \pmb{m}_2\big) &= - \frac{\exp(t-T)}{\exp(t-T)\big({\it m}_{1,3} + {\it m}_{2,3}\big) + 2}\Big( -{\it m}_{1,1}{\it m}_{1,3},\,-{\it m}_{1,2}{\it m}_{1,3},\, 1-{\it m}_{1,3}^2,\,\notag \\
&\hspace{4cm} -{\it m}_{2,1}{\it m}_{2,3},\,-{\it m}_{2,2}{\it m}_{2,3},\, 1-{\it m}_{2,3}^2\Big)^\top \in \R^6\,.\label{eq:gradient-test-problem-2}
\end{align}
Note that test problem $2$ corresponds to the terminal payoff 
\begin{align*}
h({\bf m}(T))=-\frac{1}{2}\log\Big(w_{01,00}({\bf m}_1(T),{\bf m}_2(T)) +  w_{00,01}({\bf m}_1(T), {\bf m}_2(T)) + 2 \Big) \quad \text{for}\,\, {\bf m}=({\bf m}_1, {\bf m}_2)\,.
\end{align*}
Similar to test problem $1$, we define ${\tt err}(t)$ and study its behavior in time $t$ for different values of the Monte-Carlo realization $M$ by 
using {\bf Method B}, see Figure~\ref{fig:error-test-1}, ${\rm (C)}$. The simulation is made for the following choice of parameters: $T=0.5$, $\tau=10^{-2}$, $\bar{h}= \frac{1}{\sqrt{M}}$, $\bar{\bf m}=\big({\bf e}_1, {\bf e}_2\big)$ and other parameters as
specified in the problem. We observe that the error ${\tt err}(t)$ decreases if one increases the Monte-Carlo realization $M$.

  \begin{figure}
 \centering
 \begin{subfigure}[b]{0.3\textwidth}
 \includegraphics[width=0.99\textwidth]{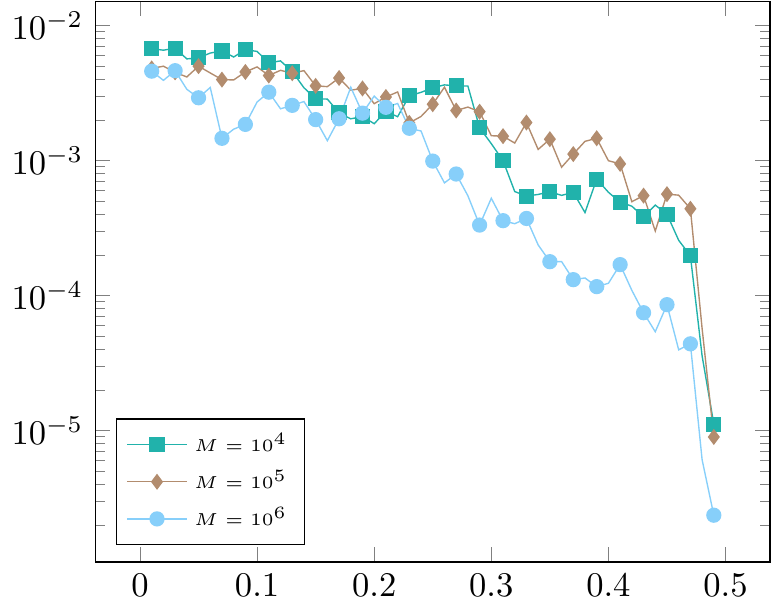}
  \caption{{\bf Method A}: $t\mapsto {\tt err}(t)$}
 \end{subfigure}
  \begin{subfigure}[b]{0.3\textwidth} 
 \includegraphics[width=0.99\textwidth]{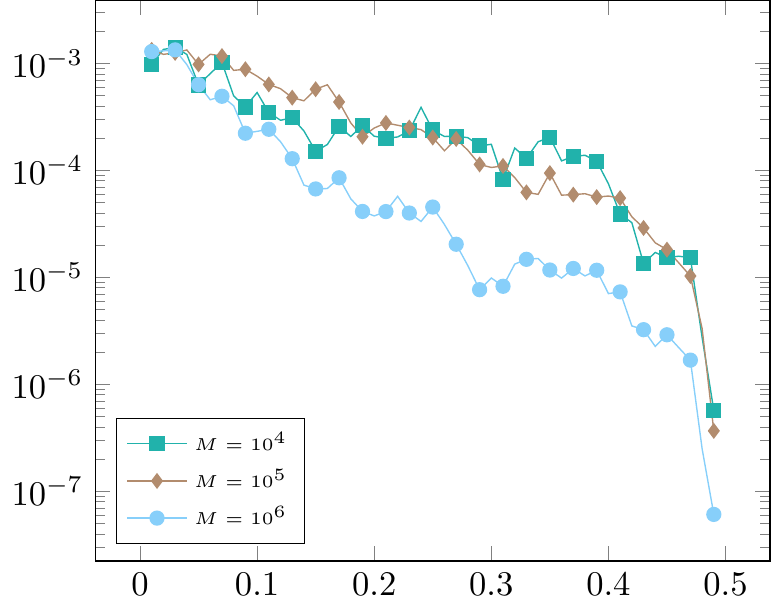}
 \caption{{\bf Method B}: $t\mapsto {\tt err}(t)$}
  \end{subfigure}
 \begin{subfigure}[b]{0.3\textwidth} 
 \includegraphics[width=0.99\textwidth]{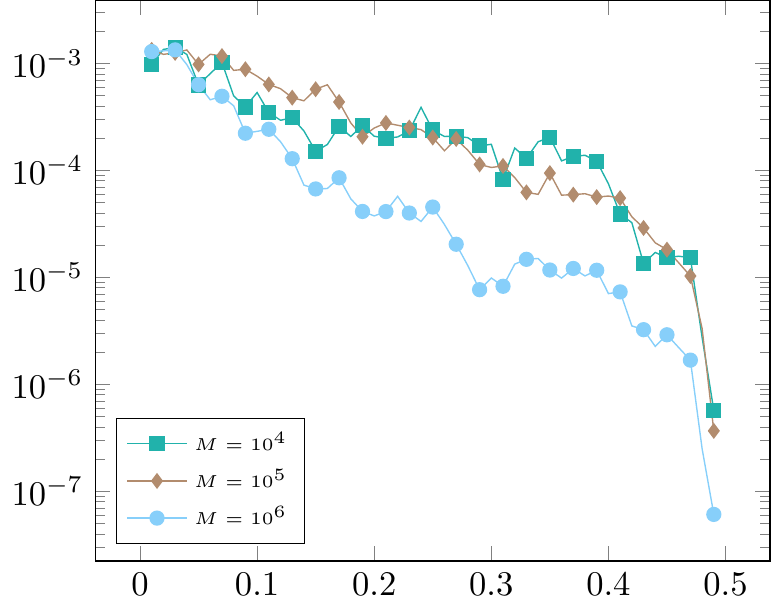}
 \caption{{\bf Method B}: $t\mapsto {\tt err}(t)$}
  \end{subfigure}
 \caption{ Time evolution of a single trajectory of the error $t\mapsto {\tt err}(t)$ : ${\rm (A)}$ by using {\bf Method A}, ${\rm (B)}$ by using {\bf Method B} for test problem $1$, and ${\rm (C)}$ by using {\bf Method B} for test problem $2$.}
 \label{fig:error-test-1}
 \end{figure}
We observe that the error ${\tt err}(t)$ for the both test problems $1$ and $2$ is of the same magnitude as the error made in the approximation of $\nabla_\mathcal{M} w$~(hence $\nabla_\mathcal{M} W$).
\vspace{.2cm}

\noindent{\bf Optimal control of two interacting isotropic spins.}
\green{In the setting of test problem $2$ we next study} the time evolution of a single trajectory of the optimal state, as well as the magnitude and direction of the optimal control. In this case, the trajectory of the optimal control lies in $x_1$-$x_2$ plane to balance the random influences; see Figure~\ref{fig:state-control-test-2}.
  \begin{figure}
 \centering
 \begin{subfigure}[b]{0.23\textwidth}
 \includegraphics[width=0.99\textwidth]{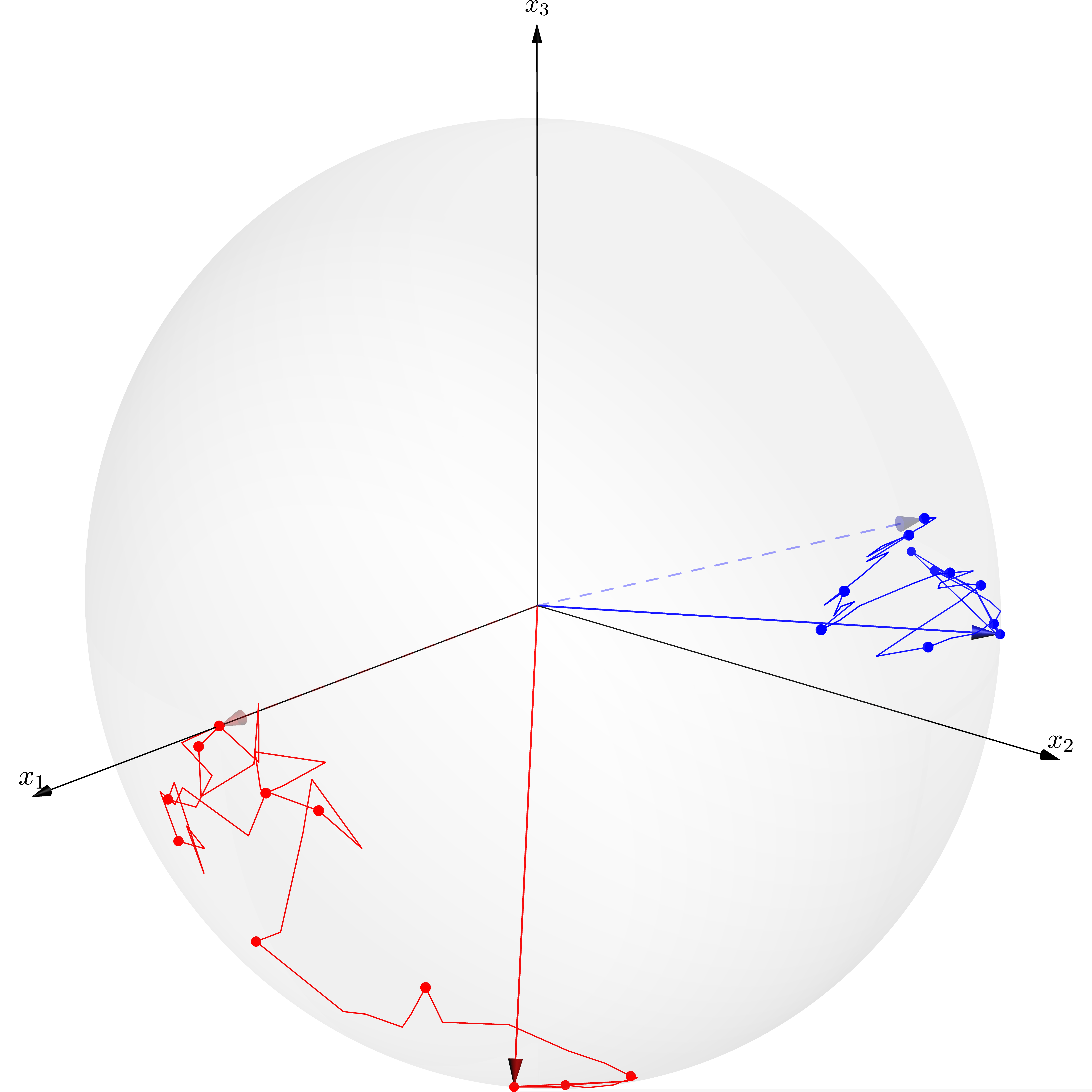}
 \caption{1st spin}
 \end{subfigure}
 \begin{subfigure}[b]{0.23\textwidth}
    \includegraphics[width=0.99\textwidth]{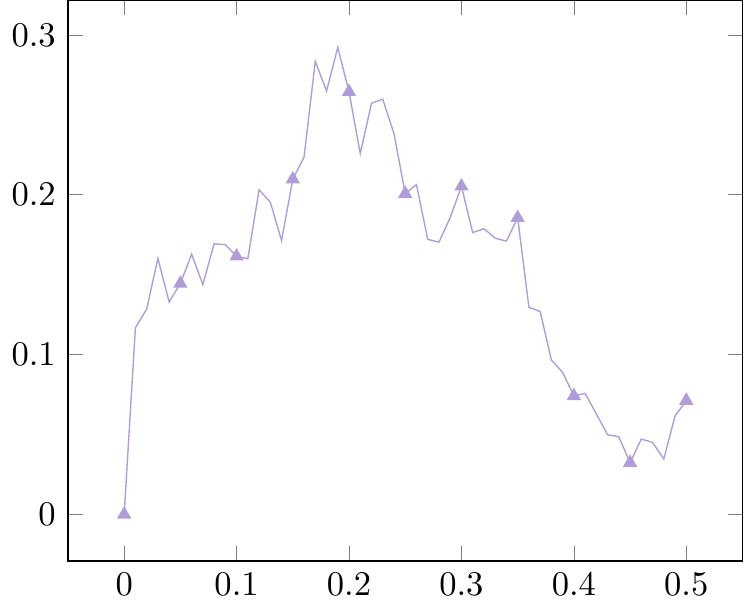}
    \caption{$t\mapsto \|u_1^*(t)\|_{\R^3}^2$}
    \end{subfigure}
  \begin{subfigure}[b]{0.23\textwidth} 
 \includegraphics[width=0.99\textwidth]{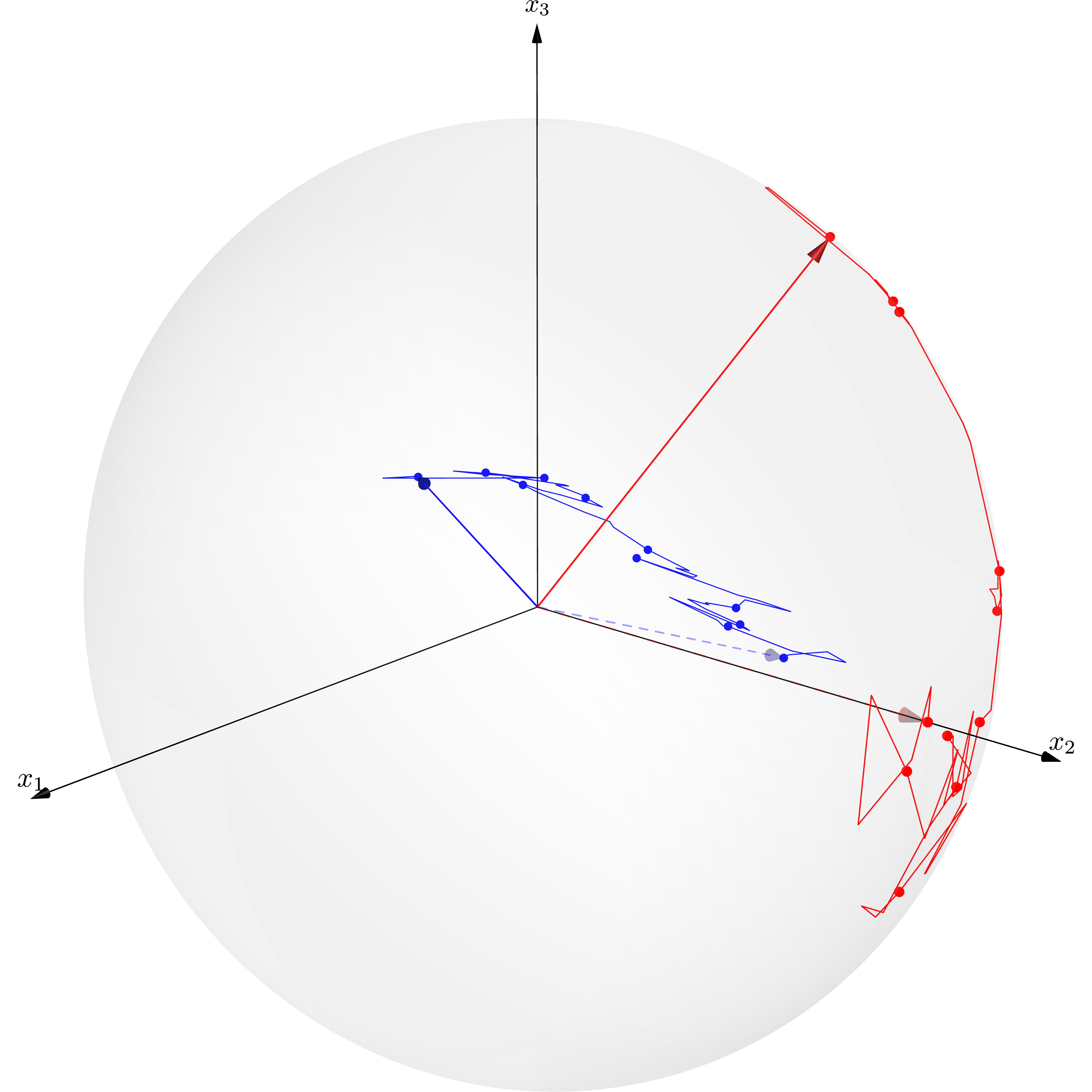}
 \caption{2nd spin}
  \end{subfigure} 
\begin{subfigure}[b]{0.23\textwidth} 
   \includegraphics[width=0.99\textwidth]{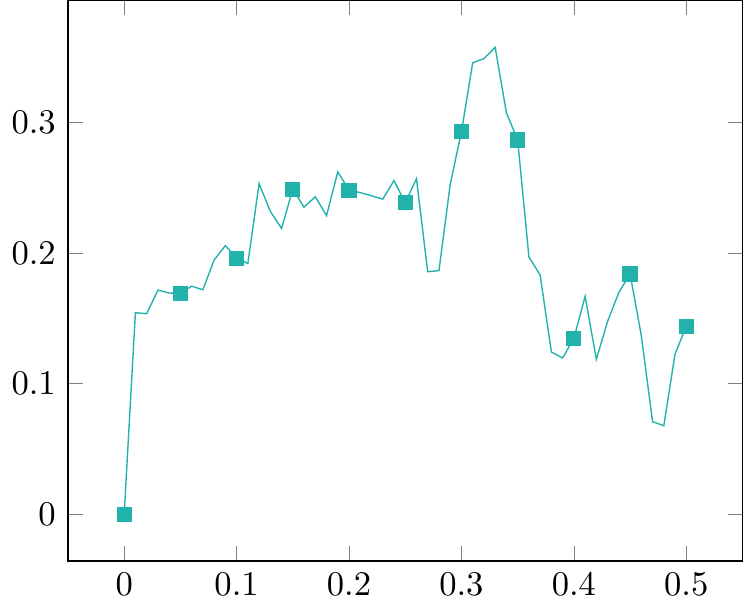}
   \caption{$t\mapsto \|u_2^*(t)\|_{\R^3}^2$}
 \end{subfigure}
 \caption{ Test problem $2$: time evolution of a single trajectory of the optimal state $t\mapsto m^*_i(t)$~(red), the direction of the optimal control $t\mapsto u_i^*(t)\|u_i^*(t)\|_{\R^3}^{-1}$~(blue), and the magnitude of the optimal control  $t\mapsto \|u_i^*(t)\|_{\R^3}^{2}$ for $i=1,2$ .}
 \label{fig:state-control-test-2}
 \end{figure}

 \begin{rem}
Computational studies for both test problems suggest stability of the scheme \eqref{eq:scheme-auxi-prob-sde-0-intermediate}-\eqref{eq:scheme-auxi-prob-sde-0-iteration}. However, convergence resp. termination of the scheme  depends crucially on the given parameters in Problem~\ref{problem-1}. For choices $$\lambda \nu^{2} \ll \min\{ \delta,1\} C_{\tt ext}^{2}(1+\alpha^{2}),$$ an exponential {\em overflow} occurs during truncation in simulations, and therefore the computed value of $w\big(t,\pmb{m}\big)$ in~\eqref{eq:prob-represnt-lhjb} is set to zero then. Hence, in this case, $\log \big(w(t,\pmb{m})\big)$ is not defined, and thus the approximation procedure to approximate $\nabla_{\mathcal{M}} W(t,\pmb{m})$ terminates. This is the reason that Examples~$5.1$ and $5.2$ from \cite{dunst_prohl} may not directly be simulated here. Notice that no exponential overflow occurs for both test problems above, since $\delta = 0$.
\end{rem}

\subsection{ Optimal control of three interacting spins}\label{subsec:three-spin}
We now study an ensemble of $N=3$ particles, which additionally are subjected to exchange
forces. We are mainly interested in the switching control for one ($i=2$) of these particles from $\bar{\bf m}_{2}$ (at initial time) to $-\bar{\bf m}_{2}$ at given final time $T$. Take
$h({\bf m}(T))= \frac{1}{2}\| {\bf m}(T)- \widetilde{\bf m}(T)\|_{(\R^3)^N}^2$, where the deterministic target profile $\widetilde{\bf m}: [0,T]\goto (\mathbb{S}^2)^3$ is given by
\begin{align*}
\widetilde{\bf m}_1(t)= {\bf e_1}\, , \quad \widetilde{\bf m}_2(t)=\begin{pmatrix} -\cos(\pi \frac{t}{T}) \\ \sin(\pi \frac{t}{T}) \\ 0 \\ \end{pmatrix}\, , \quad \widetilde{\bf m}_3(t)={\bf e}_1\, . 
\end{align*}
We use again {\bf Method B} to approximate $\nabla_\mathcal{M} W$. To simulate the optimal pair of the underlying problem,  we have used the methodology described in Subsections~\ref{subsec:eq:auxiliary-sllg} and  \ref{subsec:numeric-hjb}-\ref{subsec:optimal-state-numeric}, along with the following set-up of parameters:
\begin{center}
\begin{tabular}{ | m{0.5cm} | m{1.2cm}| m{1.8cm}  | m{2.1cm} | m{.7cm} | m{2.9cm} |m{3.5cm}| } 
\hline
$T$& $(\alpha, \delta)$ & $(\lambda, \nu)$&  $\bar{\bf m}$ & $ C_{\tt ext}$ & $\big( \bar{h}, \tau, M\big)$& ${\tt D}_i(i=1,2,3)$ \\[1mm] 
\hline
$0.5$ & $(0.1,\, 0)$ & $(10^{-3}, \, 0.3)$& $({\bf e}_1,\, -{\bf e}_1,\, {\bf e}_1)$& $0.1$& $\big(10^{-3},\, 10^{-2},\, 10^{6}\big)$ & ${\rm diag}(-5.0,\, 1.0,\, 3.5)$ \\ [1mm]
\hline
\end{tabular}
\end{center}
with the positive semi-definite matrix ${\bf J}$ such that for any ${\bf m}=(m_1, m_2,\ldots, m_N)\in (\mathbb{S}^2)^N$
\begin{equation}\label{eq:matrix-J}
 \begin{aligned}
  & ({\bf J}{\bf m})_i= -m_{i+1} + 2 m_i - m_{i-1} \quad (i=1,2,\ldots, N)\, , \\
  & m_{N+1}= m_1\,,\quad m_0= m_N\, .
 \end{aligned}
\end{equation}

  \begin{figure}
 \centering
 \begin{subfigure}[b]{0.31\textwidth}
 \includegraphics[width=0.99\textwidth]{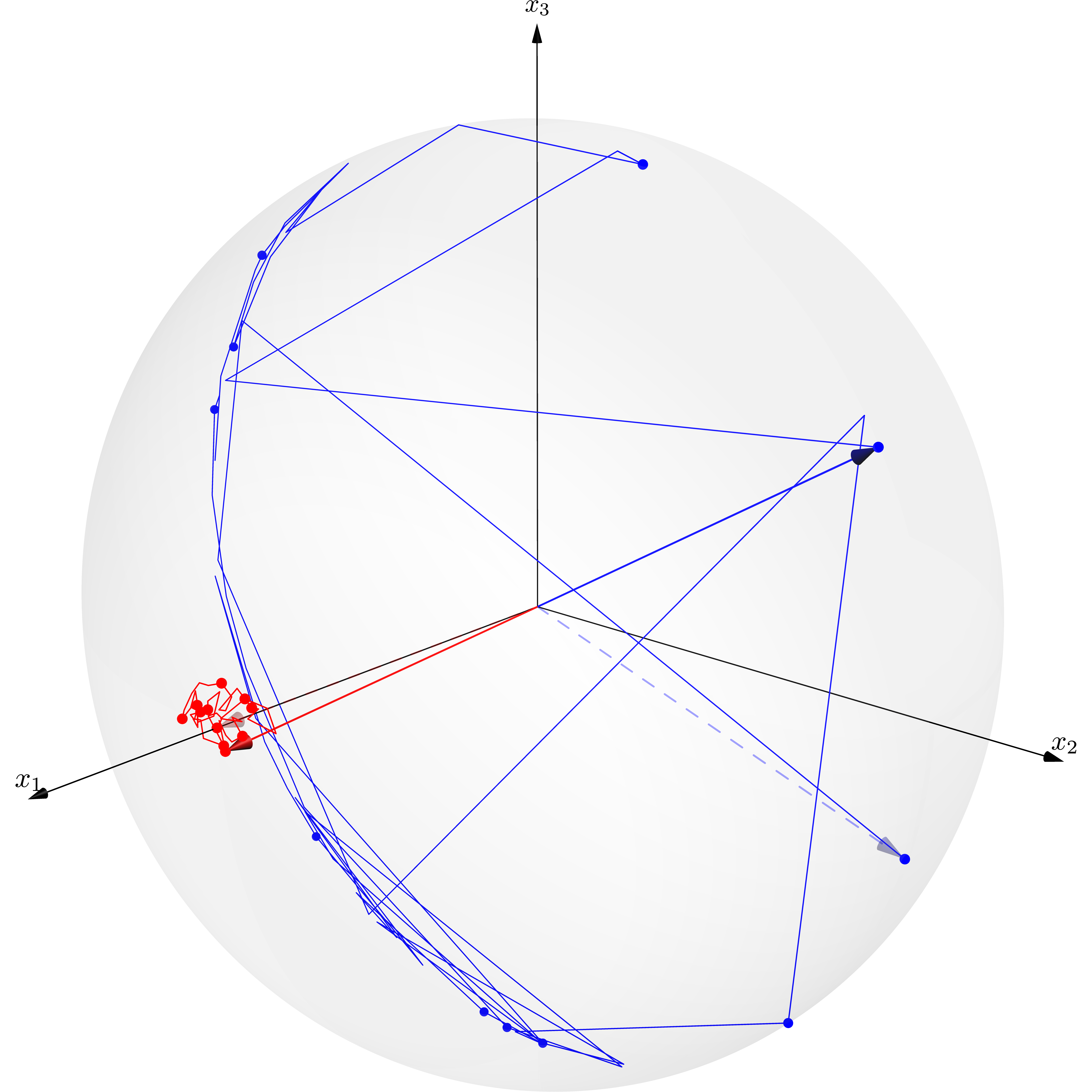}
 \caption{1st spin}
 \end{subfigure}
  \begin{subfigure}[b]{0.31\textwidth} 
 \includegraphics[width=0.99\textwidth]{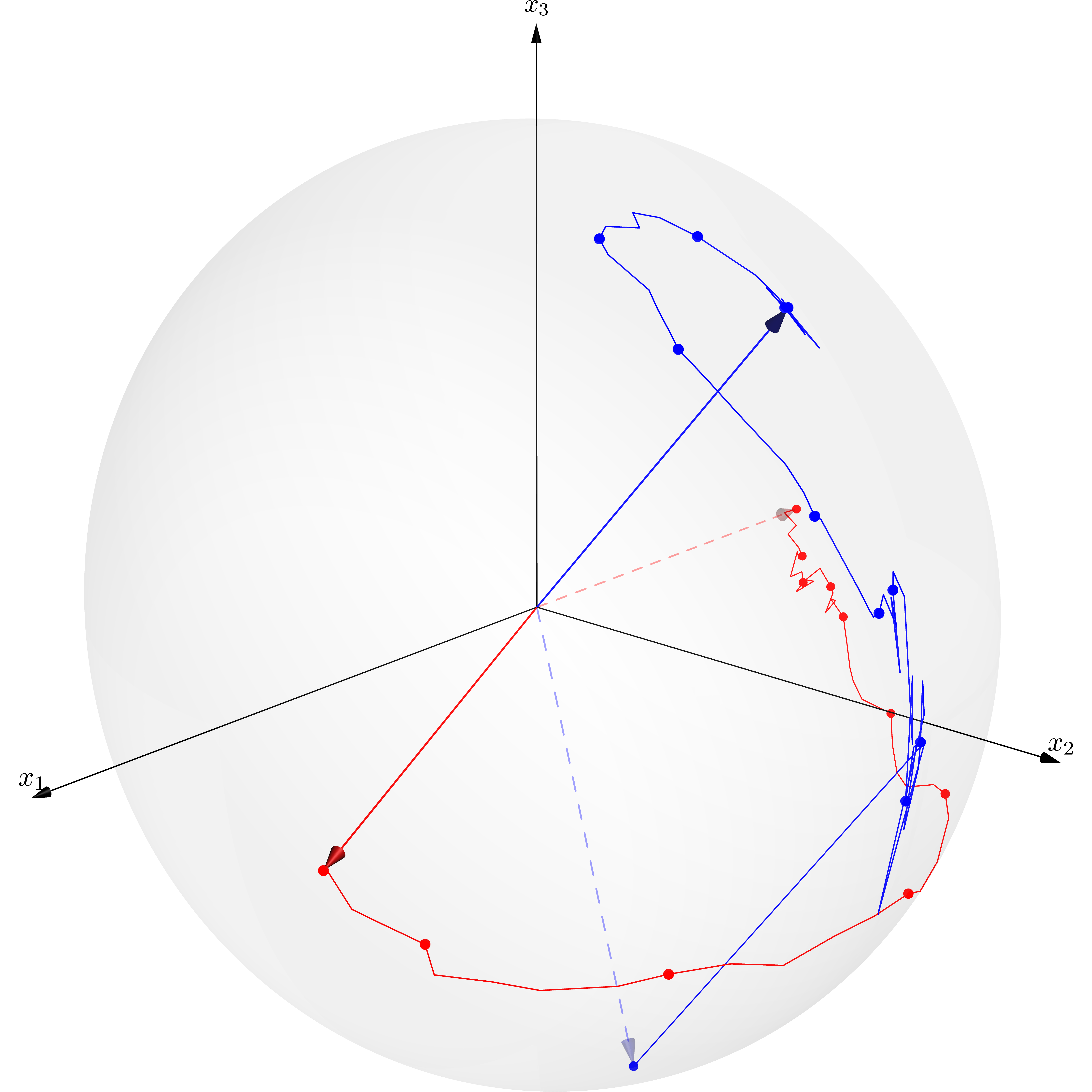}
 \caption{2nd spin}
  \end{subfigure}
    \begin{subfigure}[b]{0.31\textwidth} 
 \includegraphics[width=0.99\textwidth]{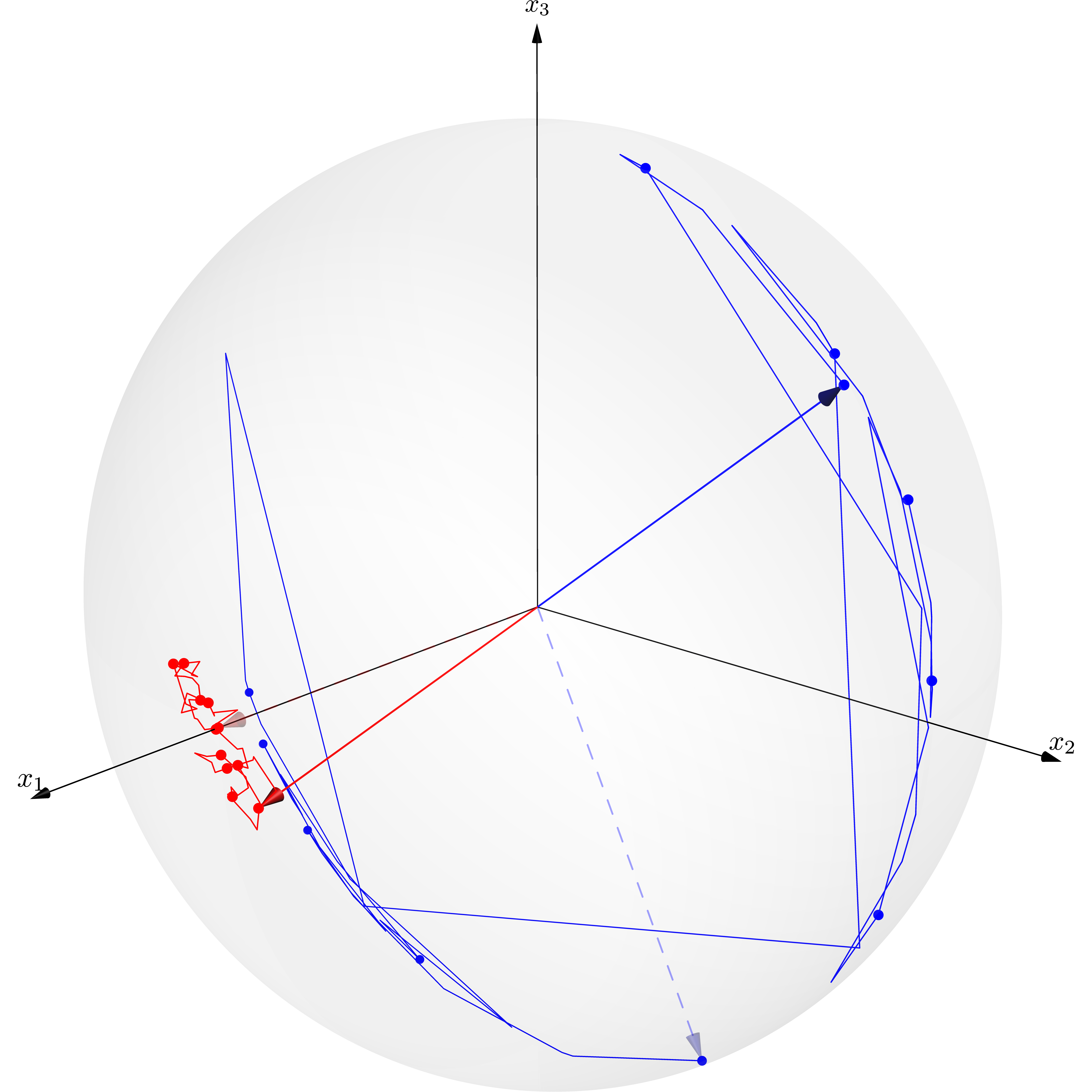}
 \caption{3rd spin}
  \end{subfigure} \\
  
   \begin{subfigure}[b]{0.30\textwidth}
   \includegraphics[width=0.99\textwidth]{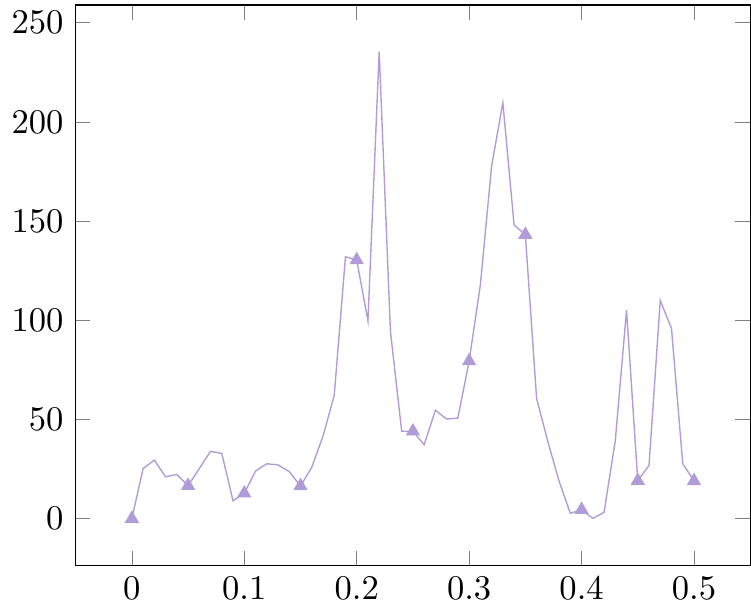}
   \caption{$t\mapsto \|u_1^*\|_{\R^3}^2$}
   \end{subfigure}
    \begin{subfigure}[b]{0.30\textwidth} 
   \includegraphics[width=0.99\textwidth]{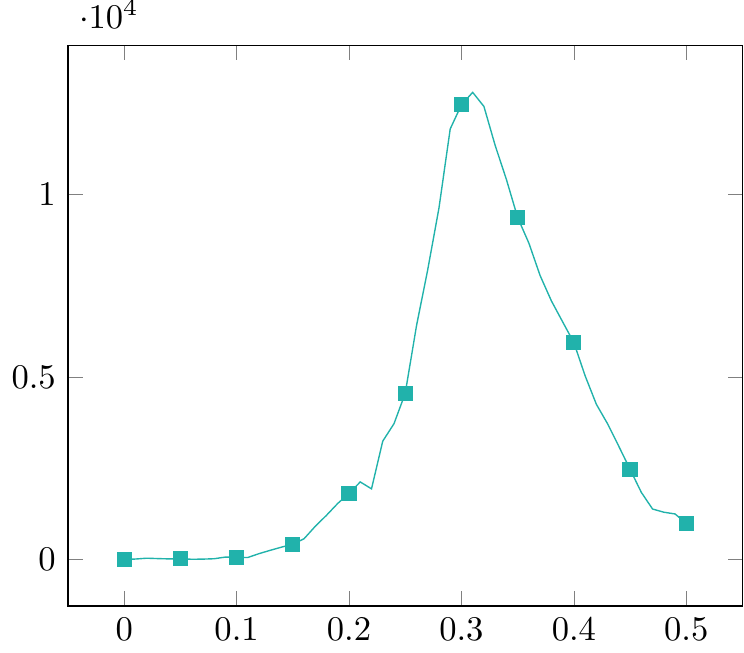}
   \caption{$t\mapsto \|u_2^*\|_{\R^3}^2$}
    \end{subfigure}
      \begin{subfigure}[b]{0.30\textwidth} 
   \includegraphics[width=0.99\textwidth]{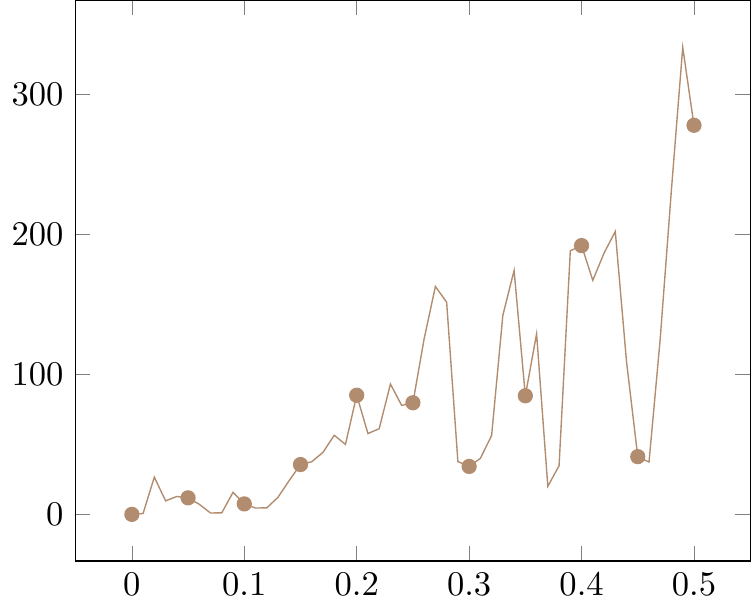}
   \caption{$t\mapsto \|u_3^*\|_{\R^3}^2$}
    \end{subfigure} \\
     \begin{subfigure}[b]{0.30\textwidth}
       \includegraphics[width=0.99\textwidth]{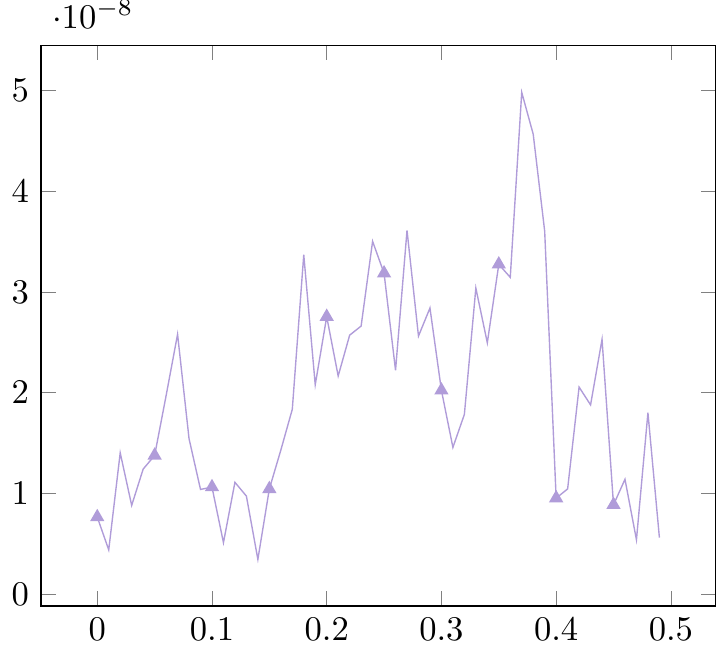}
       \caption{$t\mapsto \ll u_1^*(t), m_1^*(t)\gg $}
       \end{subfigure}
        \begin{subfigure}[b]{0.30\textwidth} 
       \includegraphics[width=0.99\textwidth]{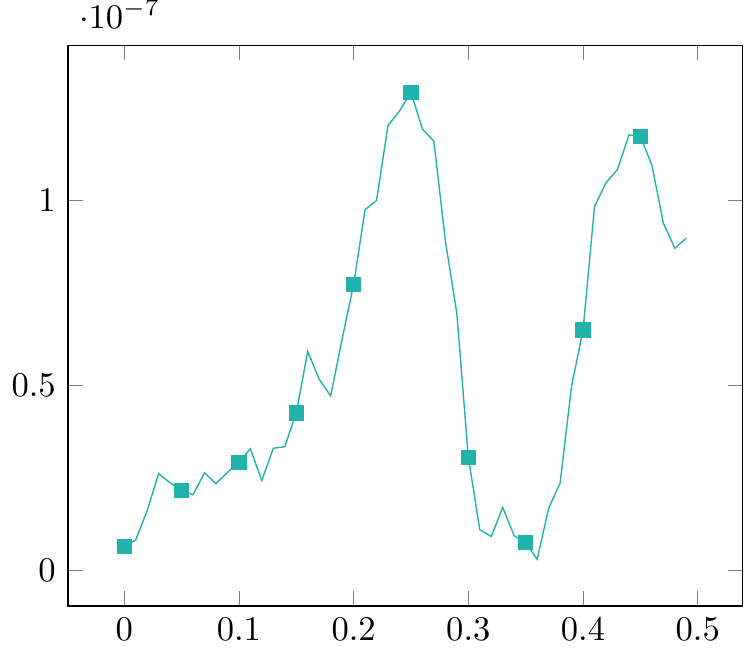}
       \caption{$t\mapsto \ll u_2^*(t), m_2^*(t)\gg$}
        \end{subfigure}
          \begin{subfigure}[b]{0.30\textwidth} 
       \includegraphics[width=0.99\textwidth]{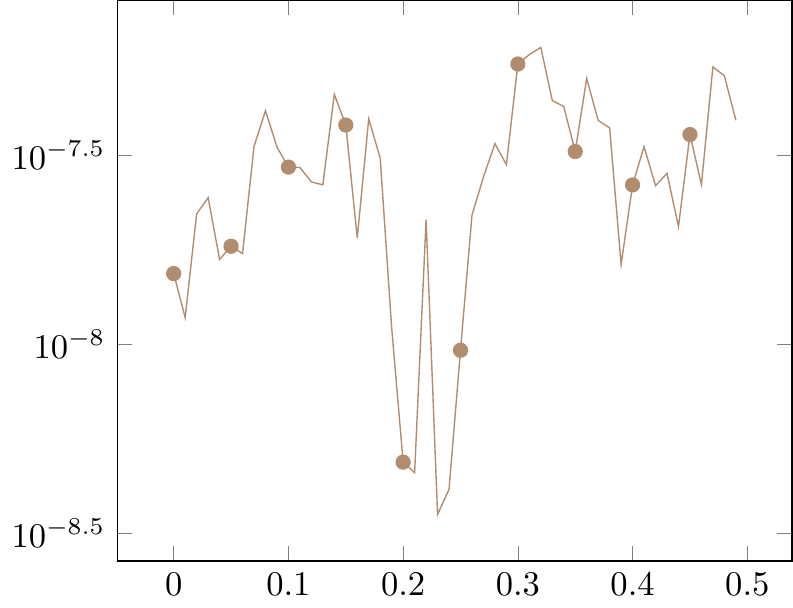}
       \caption{$t\mapsto \ll u_3^*(t), m_3^*(t)\gg$}
        \end{subfigure} 
  
 \caption{ Time evolution of a single trajectory of the optimal state $t\mapsto m^*_i(t)$~(red), the direction of the optimal control $t\mapsto u_i^*(t)\|u_i^*(t)\|_{\R^3}^{-1}$~(blue), the magnitude of the optimal control $t\mapsto \|u_i^*(t)\|_{\R^3}^{2}$, and the angle between optimal pair $t\mapsto \ll u_i^*(t), m_i^*(t)\gg$ for $i=1,2,3$.}
 \label{fig:state-control-3-spin}
 \end{figure}

In this case, the minimum value of the cost functional is $\mathcal{J}_{\tt sto}^*\equiv 0.9078$. 
Though the first and third spins start already at the desired state, it is due to the noise, and the exchange forces in particular, \green{that the optimal control is acting on the whole time interval and on all spins.} For the second spin, we observe that at the beginning and end, less control is needed opposed to the applied control at the mean time; see Figure~\ref{fig:state-control-3-spin}. The orthogonality of the optimal pair $({\bf m}^*, {\bf u}^*)$~(e.g. Remark~\ref{rem:optimal-pair-orthogonality}) is shown in Figure~\ref{fig:state-control-3-spin}, ${\rm (G)}$-${\rm (I)}$ by displaying the temporal evolution 
\begin{align*}
t\mapsto \ll m_i^*(t), u_i^*(t) \gg:= \frac{\big|\langle m_i^*(t), u_i^*(t)\rangle_{(\R^3)} \big|}{\|m_i^*(t)\|_{\R^3}\, \|u_i^*(t)\|_{\R^3}} \qquad (i=1,2,3)\,.
\end{align*}

\subsection{ Optimal control of four interacting spins} We consider here the switching control for an ensemble of $N=4$ particles. 
\vspace{.2cm}

\noindent{\bf Set-up $1$:} We use the parameters as in Subsection~\ref{subsec:three-spin} with $ \bar{\bf m}=\big( {\bf e}_1, -{\bf e}_1, {\bf e_1}, -{\bf e}_1\big)$, and $ \widetilde{\bf m}(t)=\big({\bf e}_1,\widetilde{\bf m}_2(t), {\bf e}_1, \widetilde{\bf m}_2(t)\big)$. In this case, the first and third spins start already at the desired state; the associated optimal controls are acting on the whole time interval. Moreover, for the second and fourth spins, significant controls are required to approximately meet the terminal state profile. The time evolution of $t\mapsto \|u_2^*(t)\|_{\R^3}^2$ is similar to the results for $N=3$ spin constellations~(see Figure~\ref{fig:state-control-3-spin}, ${\rm (E)}$), while $\|u_4^*(t)\|_{\R^3}^2$ is delayed in time for the fourth spin. We observe a loop of the orientation of $u_i^*(t)\|u_i^*(t)\|_{\R^3}^{-1}~(i=2,4)$ close to the terminal time; see Figure~\ref{fig:state-control-4-spin-set-up-1}.

\begin{figure}
 \centering
  \begin{subfigure}[b]{0.22\textwidth}
 \includegraphics[width=0.99\textwidth]{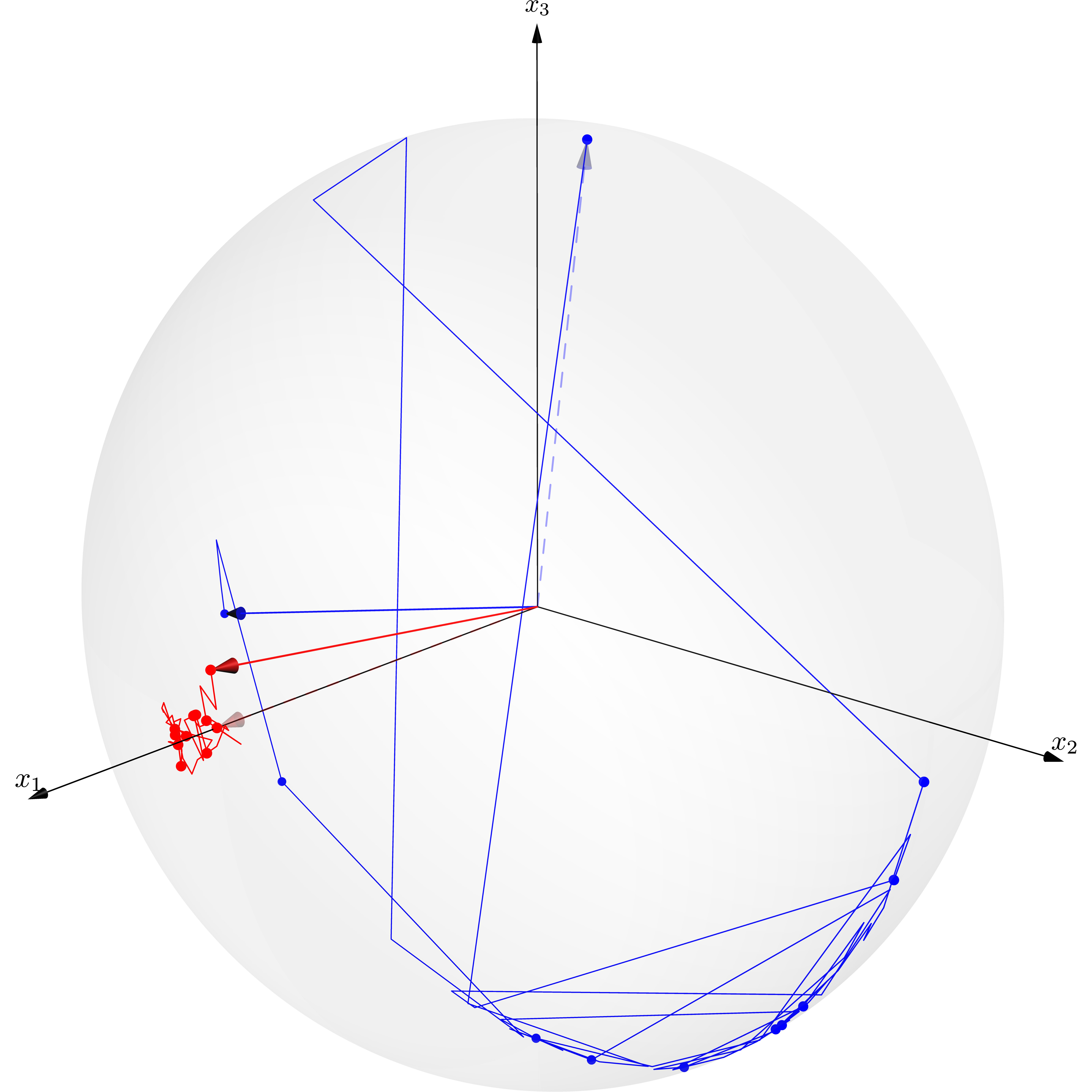}
  \caption{1st spin}
 \end{subfigure}
   \begin{subfigure}[b]{0.22\textwidth} 
 \includegraphics[width=0.99\textwidth]{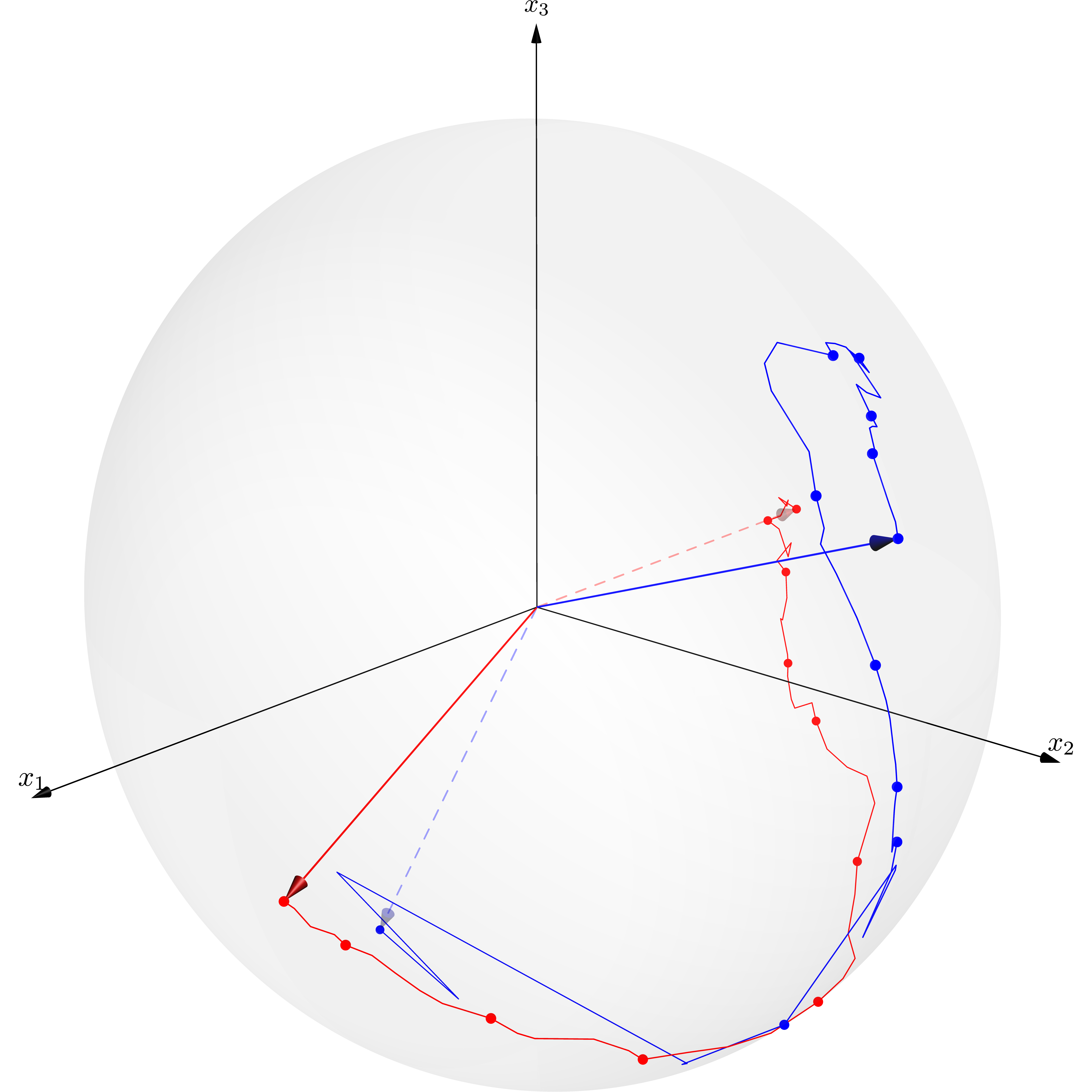}
  \caption{2nd spin}
   \end{subfigure}
     \begin{subfigure}[b]{0.22\textwidth} 
 \includegraphics[width=0.99\textwidth]{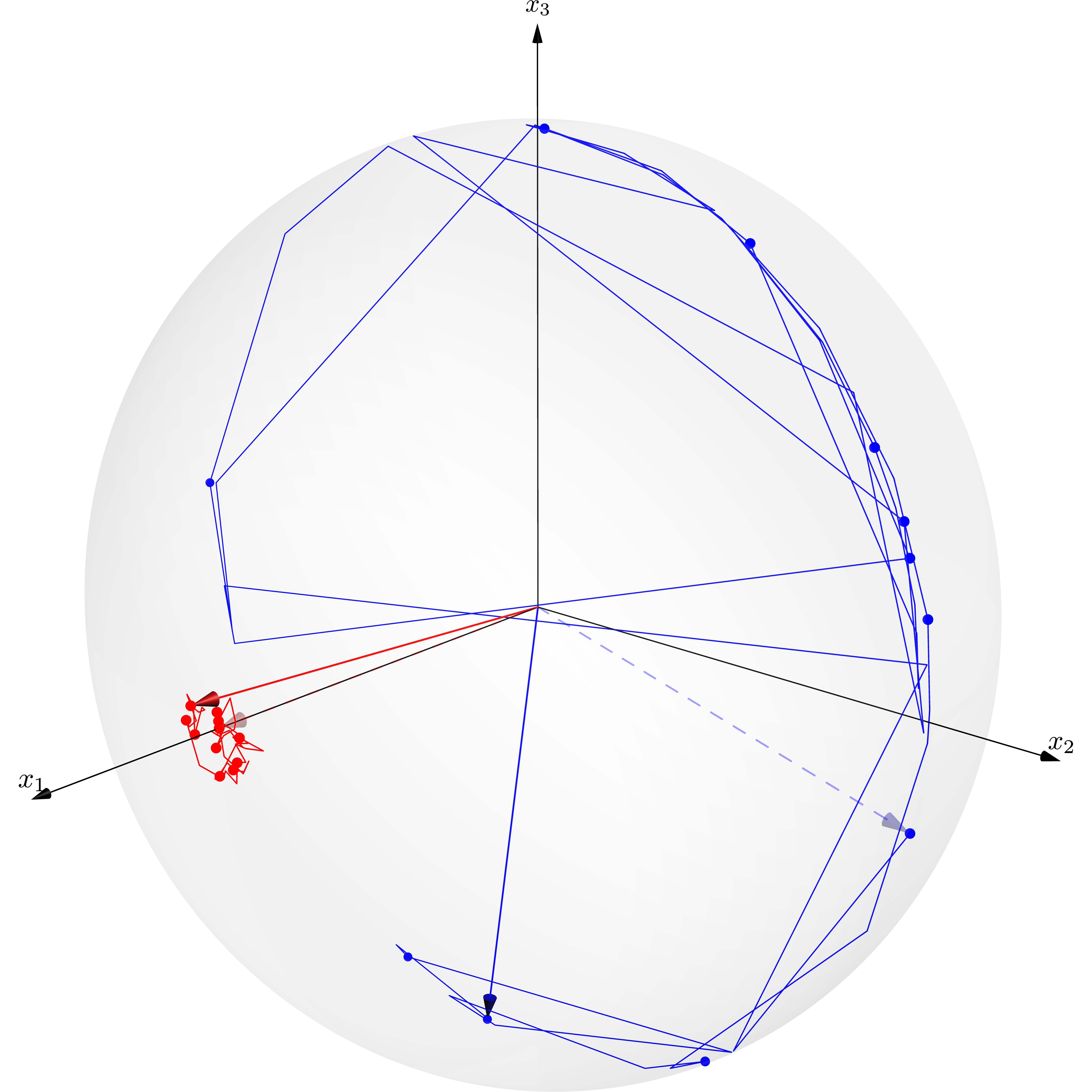}
  \caption{3rd spin}
  \end{subfigure}
   \begin{subfigure}[b]{0.22\textwidth} 
    \includegraphics[width=0.99\textwidth]{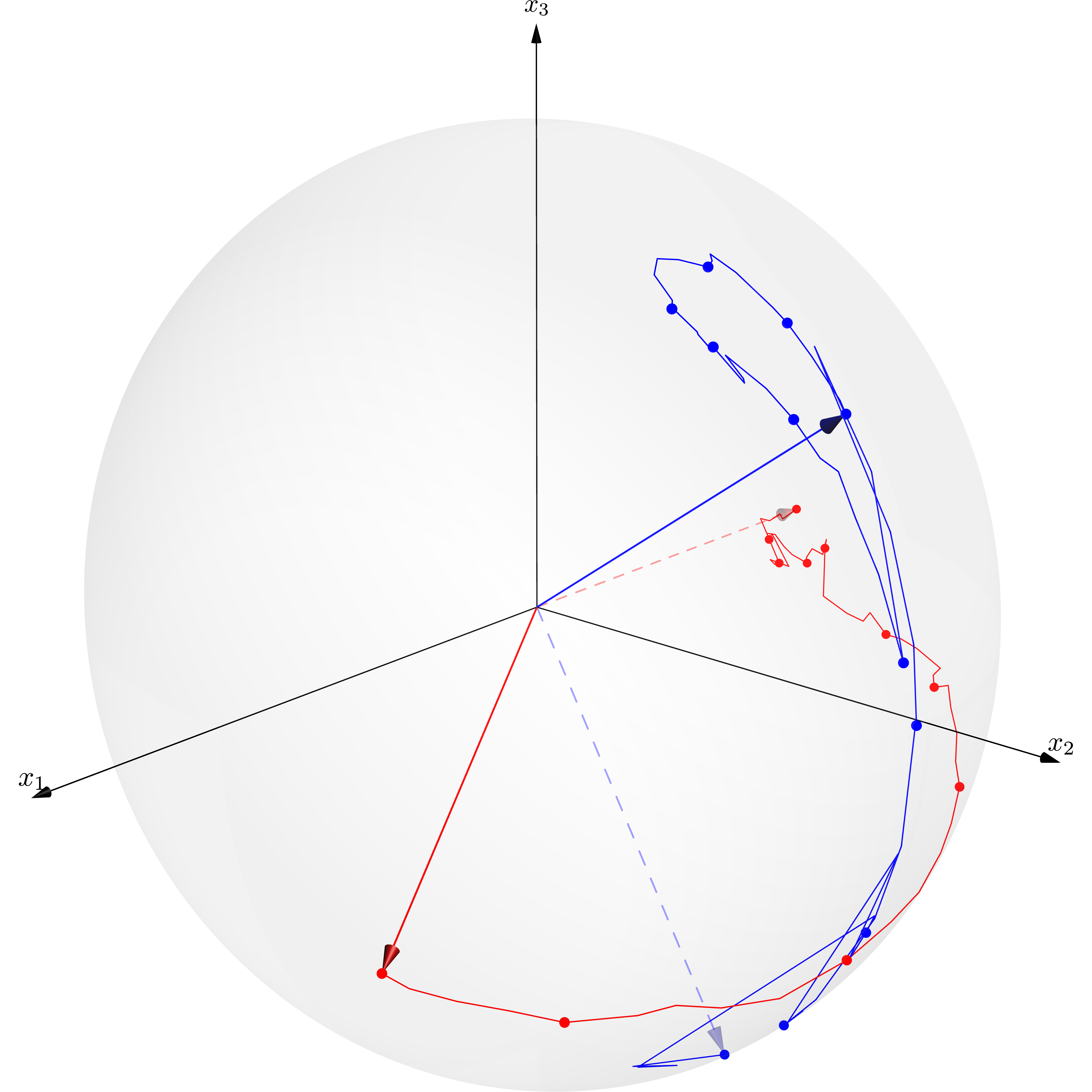}
     \caption{4th spin}
     \end{subfigure}\\
   \begin{subfigure}[b]{0.21\textwidth}
    \includegraphics[width=0.99\textwidth]{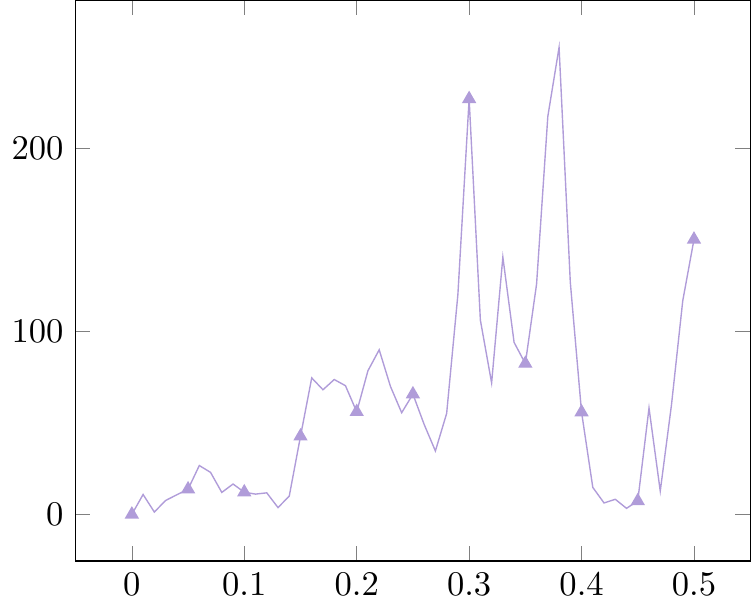}
   \caption{$t\mapsto \|u_1^*(t)\|_{\R^3}^2$}
    \end{subfigure}
    \begin{subfigure}[b]{0.21\textwidth} 
    \includegraphics[width=0.99\textwidth]{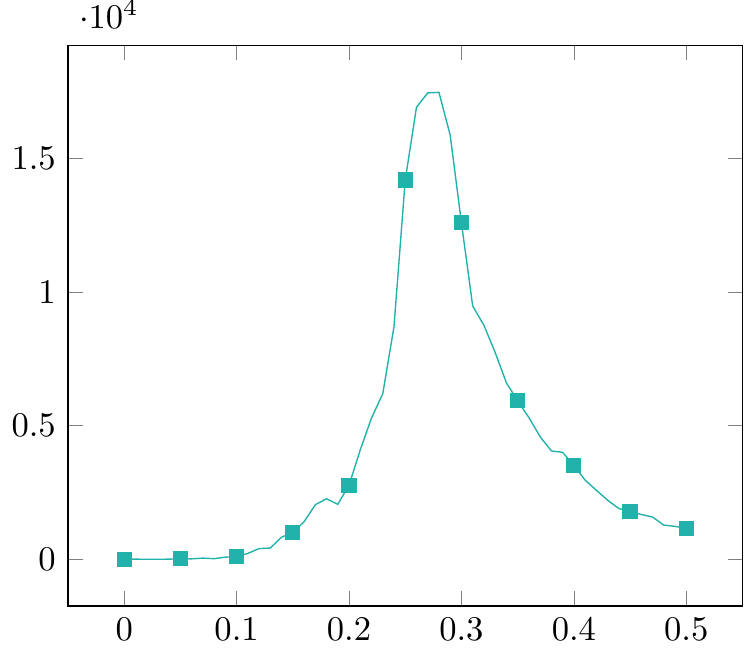}
   \caption{$t\mapsto \|u_2^*(t)\|_{\R^3}^2$}
     \end{subfigure}
      \begin{subfigure}[b]{0.21\textwidth} 
   \includegraphics[width=0.99\textwidth]{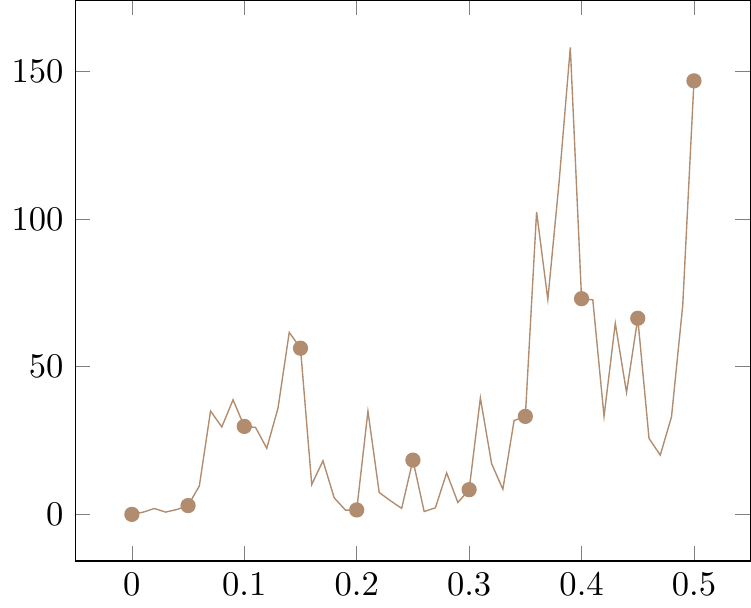}
   \caption{$t\mapsto \|u_3^*(t)\|_{\R^3}^2$}
    \end{subfigure} 
     \begin{subfigure}[b]{0.21\textwidth} 
        \includegraphics[width=0.99\textwidth]{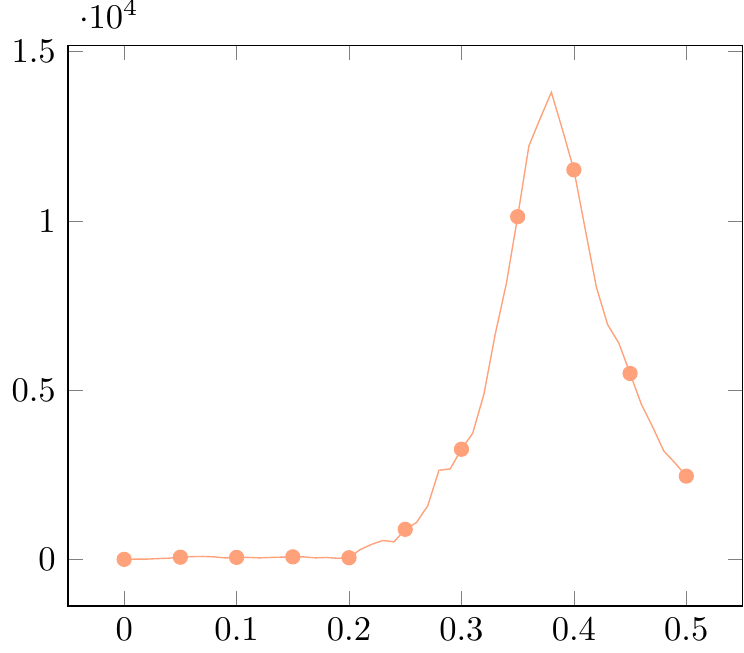}
        \caption{$t\mapsto \|u_4^*(t)\|_{\R^3}^2$}
         \end{subfigure} 
  \caption{ Time evolution of a single trajectory of the optimal state $t\mapsto m^*_i(t)$~(red), the direction of the optimal control $t\mapsto u_i^*(t)\|u_i^*(t)\|_{\R^3}^{-1}$~(blue), and the magnitude of the optimal control $t\mapsto \|u_i^*(t)\|_{\R^3}^{2}$ with set-up $1$ for $i=1,2,3, 4$.}
  \label{fig:state-control-4-spin-set-up-1}
  \end{figure}

\noindent{\bf Set-up $2$:} We use same parameters as in {\bf set-up $1$} with $
 \bar{\bf m}=\big( {\bf e}_1, -{\bf e}_1, -{\bf e_1}, {\bf e}_1\big)$ and $\widetilde{\bf m}(t)=\big({\bf e}_1,\widetilde{\bf m}_2(t), \widetilde{\bf m}_2(t), {\bf e}_1\big)$.
For the second and third spins, significantly synchronous controls at mean times are required to meet approximately the  desired target profile. Like in Set-up $1$, we also observe the formation of loops of the orientation of $u_i^*(t)\|u_i^*(t)\|_{\R^3}^{-1}~(i=2,3)$ close to the terminal time; see Figure~\ref{fig:state-control-4-spin-set-up-2}.

 \begin{figure}
  \centering
   \begin{subfigure}[b]{0.22\textwidth}
  \includegraphics[width=0.99\textwidth]{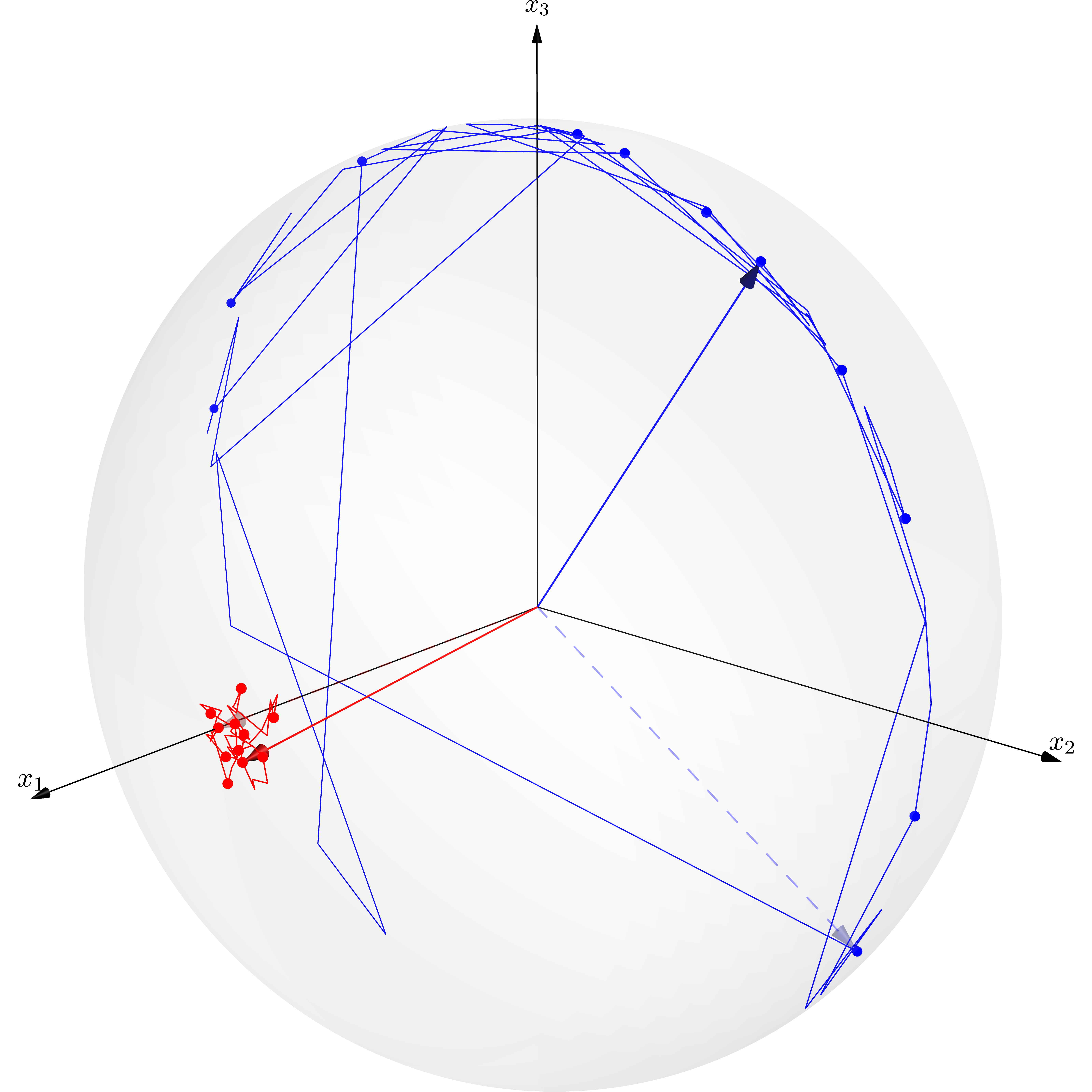}
   \caption{1st spin}
  \end{subfigure}
    \begin{subfigure}[b]{0.22\textwidth} 
  \includegraphics[width=0.99\textwidth]{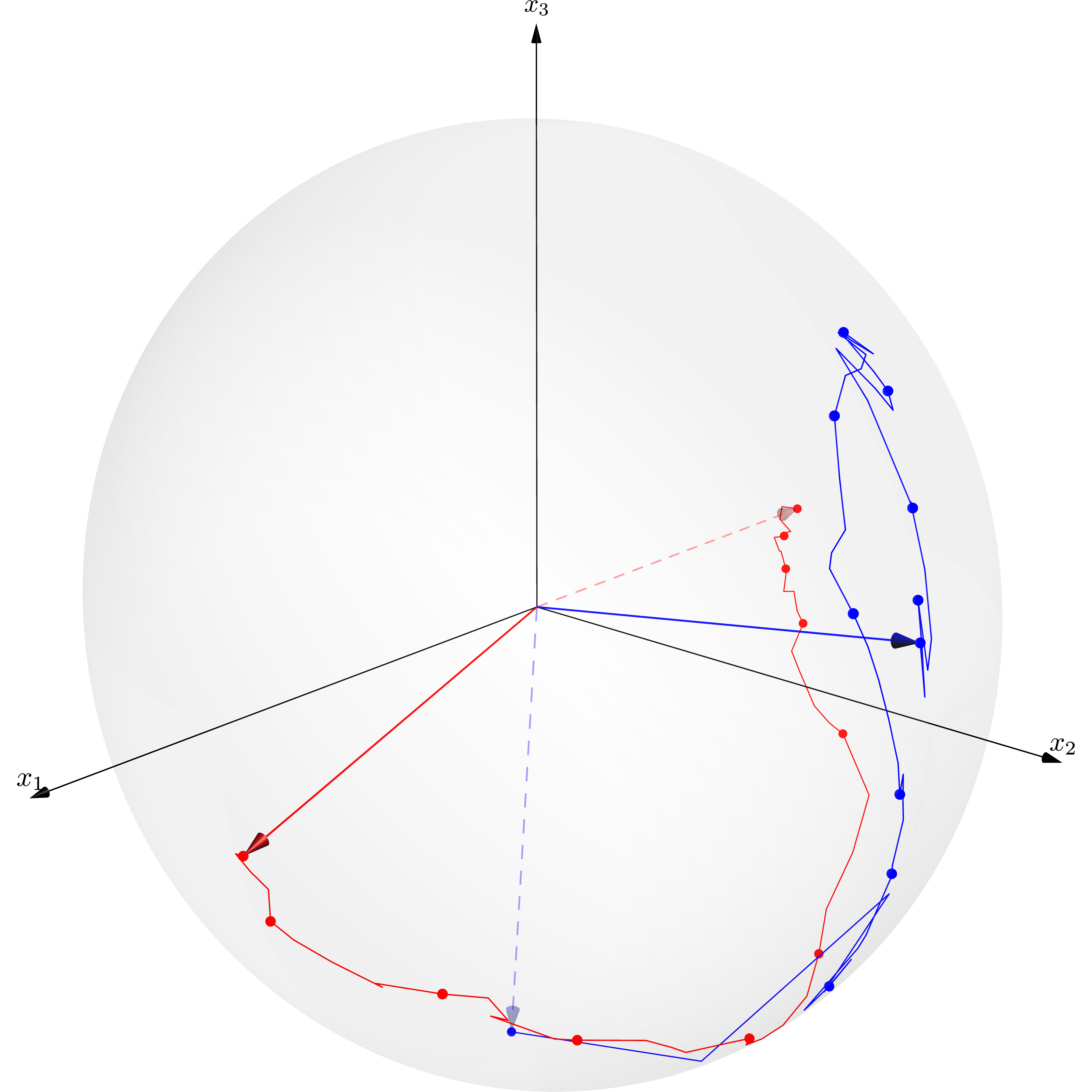}
   \caption{2nd spin}
    \end{subfigure}
      \begin{subfigure}[b]{0.22\textwidth} 
  \includegraphics[width=0.99\textwidth]{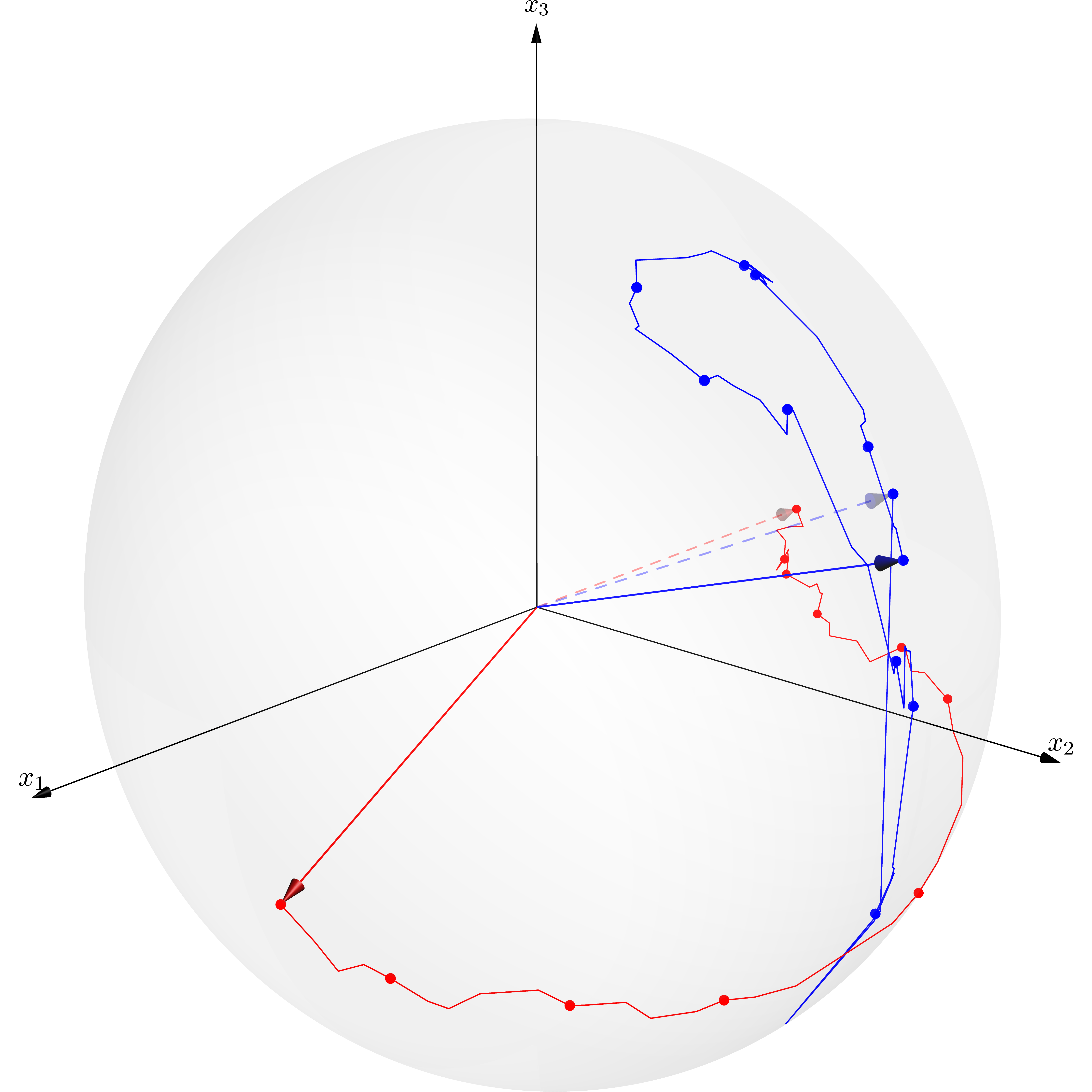}
   \caption{3rd spin}
   \end{subfigure}
    \begin{subfigure}[b]{0.22\textwidth} 
     \includegraphics[width=0.99\textwidth]{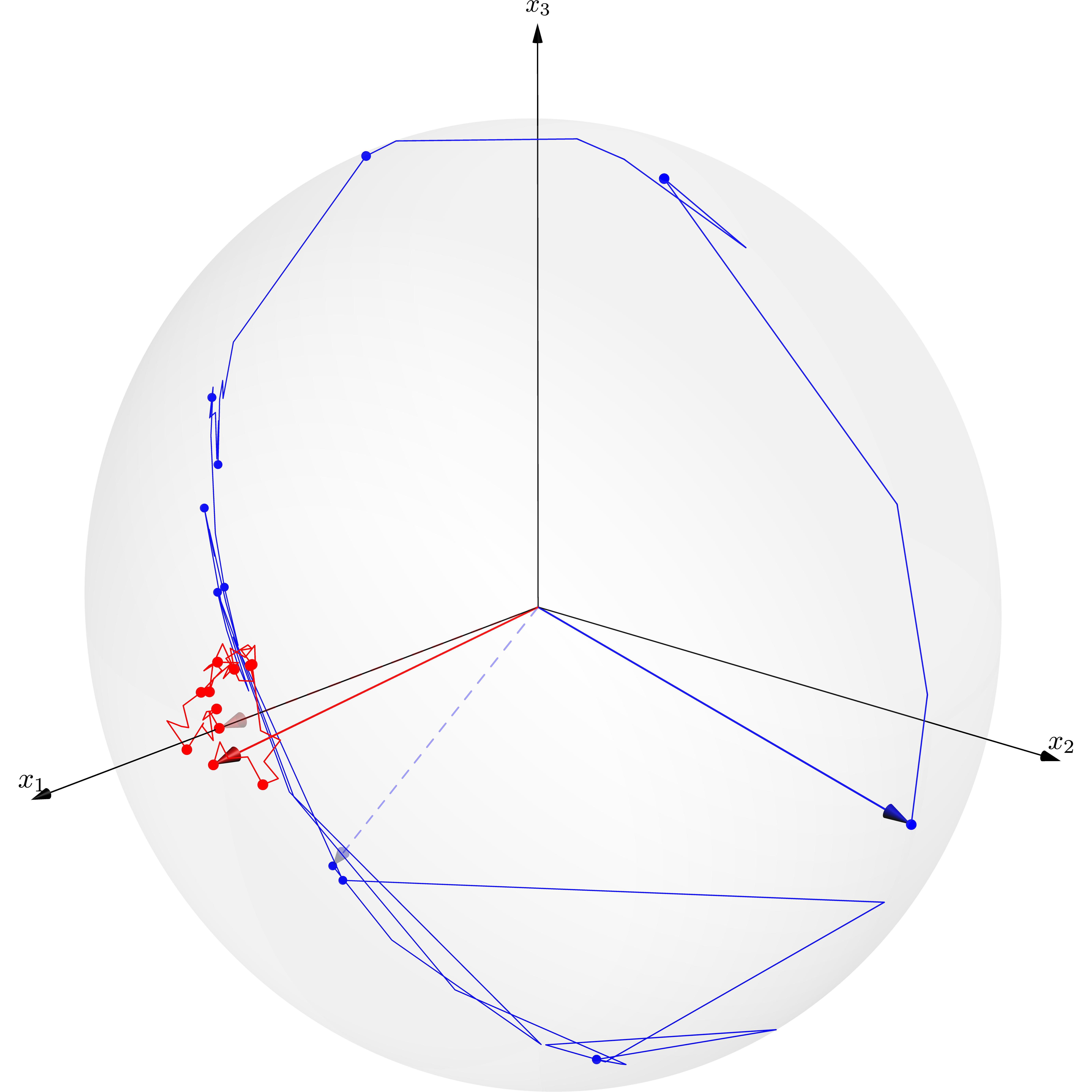}
      \caption{4th spin}
      \end{subfigure}\\
    \begin{subfigure}[b]{0.21\textwidth}
     \includegraphics[width=0.99\textwidth]{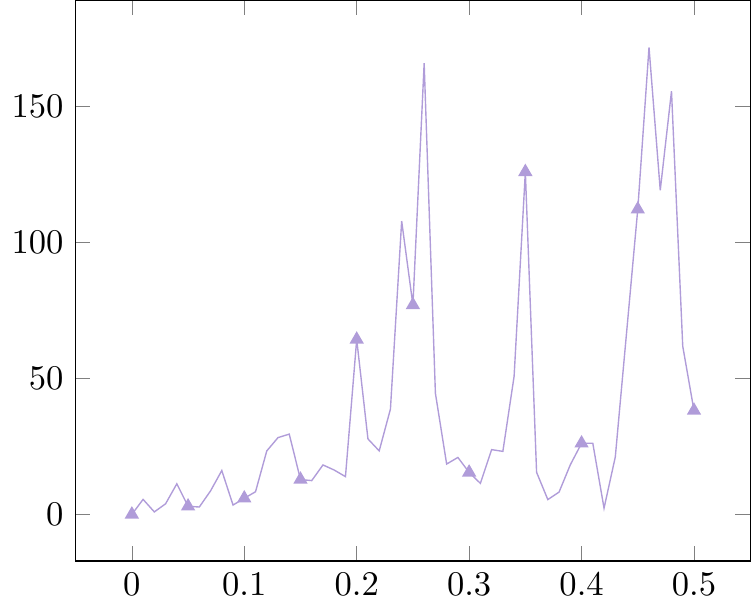}
    \caption{$t\mapsto \|u_1^*(t)\|_{\R^3}^2$}
     \end{subfigure}
     \begin{subfigure}[b]{0.21\textwidth} 
     \includegraphics[width=0.99\textwidth]{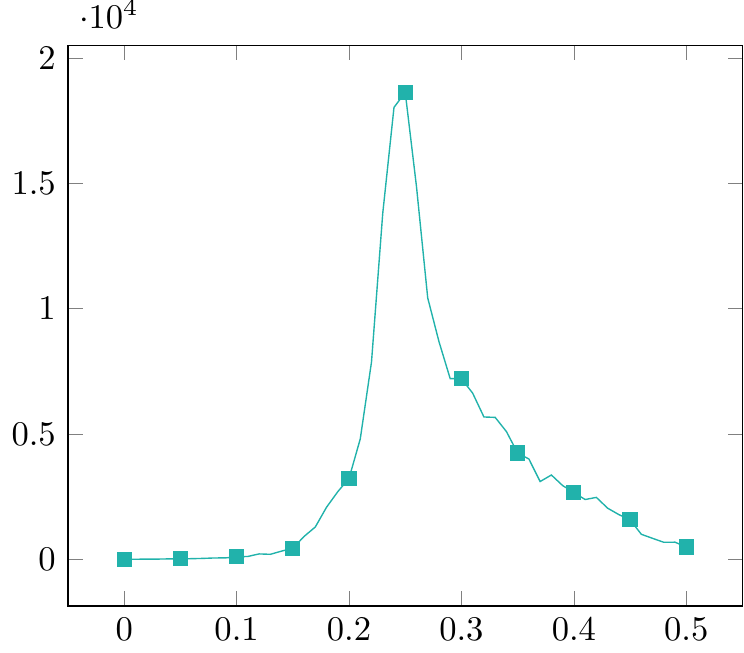}
    \caption{$t\mapsto \|u_2^*(t)\|_{\R^3}^2$}
      \end{subfigure}
       \begin{subfigure}[b]{0.21\textwidth} 
    \includegraphics[width=0.99\textwidth]{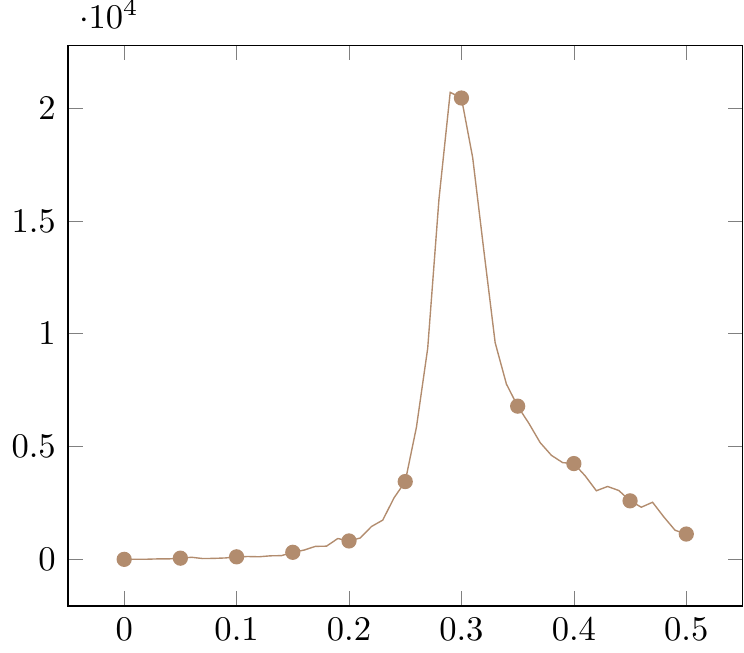}
    \caption{$t\mapsto \|u_3^*(t)\|_{\R^3}^2$}
     \end{subfigure} 
      \begin{subfigure}[b]{0.21\textwidth} 
         \includegraphics[width=0.99\textwidth]{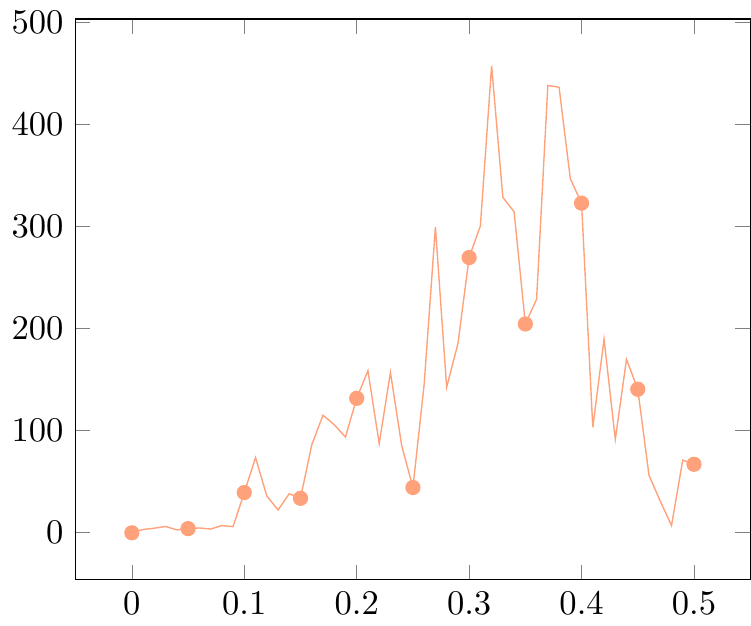}
         \caption{$t\mapsto \|u_4^*(t)\|_{\R^3}^2$}
          \end{subfigure} 
   \caption{ Time evolution of a single trajectory of the optimal state $t\mapsto m^*_i(t)$~(red), the direction of the optimal control $t\mapsto u_i^*(t)\|u_i^*(t)\|_{\R^3}^{-1}$~(blue), and the magnitude of the optimal control $t\mapsto \|u_i^*(t)\|_{\R^3}^{2}$ with set-up $2$ for $i=1,2,3, 4$.}
   \label{fig:state-control-4-spin-set-up-2}
   \end{figure}

 \subsection{Optimal control of ten interacting spins}
 We consider here an ensemble of $N=10$ particles to optimally control the dynamics to reach a deterministic target profile 
 \begin{align*}
 \widetilde{\bf m}=(\widetilde{\bf m}_1,\ldots,\widetilde{\bf m}_{10}):[0,T]\goto (\mathbb{S}^2)^{10} \quad \text{with}  \quad \widetilde{\bf m}_i(t)={\bf e}_1 \quad (i=1,2,\ldots, 10)
 \end{align*}
  within finite time $T$ at minimized expected external energy with initial configuration
  \begin{align*}
  \bar{\bf m}=\big( \bar{\bf m}_1,\ldots, \bar{\bf m}_{10}\big)\,,\quad \text{where}\quad \bar{\bf m}_i= \Big(0,\, \sin(\frac{2\pi i}{10}),\, \cos(\frac{2\pi i}{10}) \Big)^\top \quad (i=1,\ldots,10)\,.
  \end{align*}
   To simulate the optimal pair of the underlying problem, we take again ${\bf D}={\bf 0}$, ${\bf J}$ as in \eqref{eq:matrix-J}, $h({\bf m}(T))= \displaystyle \frac{1}{2} \| {\bf m}(T)- \widetilde{\bf m}(T)\|_{(\R^3)^N}^2$,   and the following set-up of parameters:
 
 \begin{center}
 \begin{tabular}{ | m{0.5cm} | m{1.5cm}| m{1.5cm} | m{1.0cm} | m{2.9cm} | } 
 \hline
 $T$& $(\alpha,\, \delta)$ & $(\lambda,\, \nu)$ & $C_{\tt ext}$ & $\big( \bar{h},\, \tau,\, M\big)$ \\[1mm] 
 \hline
 $0.5$ & $(1.0,\, 0)$ & $(1.0,\, 0.5)$& $1.0$& $\big(10^{-2},\, 10^{-2},\, 10^{4}\big)$\\ [1mm]
 \hline
 \end{tabular}
 \end{center}
 In Figure~\ref{fig:state-control-10-spin}, ${\rm (A)}$, we visualize the behavior of the optimal state ${\bf m}^*$. Due to the large damping coefficient $\alpha=1.0$, we observe fast switching dynamics of the optimal state. With the choice $\lambda=1$ the control is penalised more strongly than in the previous experiments, which has a noticeable effect on the magnitude of ${\bf u}^*$, compare Figures~\ref{fig:state-control-4-spin-set-up-2} (E) -- (H) and \ref{fig:state-control-10-spin} (C) \& (E). At the beginning a stronger control is applied to move towards the desired target profile. Because of the large noise intensity $\nu$, and the less control, some particles of this ensemble do not reach the target profile appropriately. For illustration, we plotted the behavior of the optimal state $t\mapsto m^*_i(t)$~(red), the direction of the optimal control $t\mapsto u_i^*(t)\|u_i^*(t)\|_{\R^3}^{-1}$~(blue), and the magnitude of the optimal control $t\mapsto \|u_i^*(t)\|_{\R^3}^{2}$ for $i=3,7$; see Figure~\ref{fig:state-control-10-spin}.
 
  \begin{figure}
   \centering
    \begin{subfigure}[b]{0.80\textwidth}
   \includegraphics[width=0.99\textwidth,height=3.5cm]{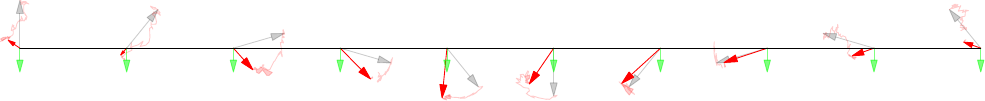}
    \caption{Initial state~(gray), target profile~(green) and optimal solution~(red)}
       \end{subfigure}\\
     \begin{subfigure}[b]{0.21\textwidth}
      \includegraphics[width=0.99\textwidth,height=3.5cm]{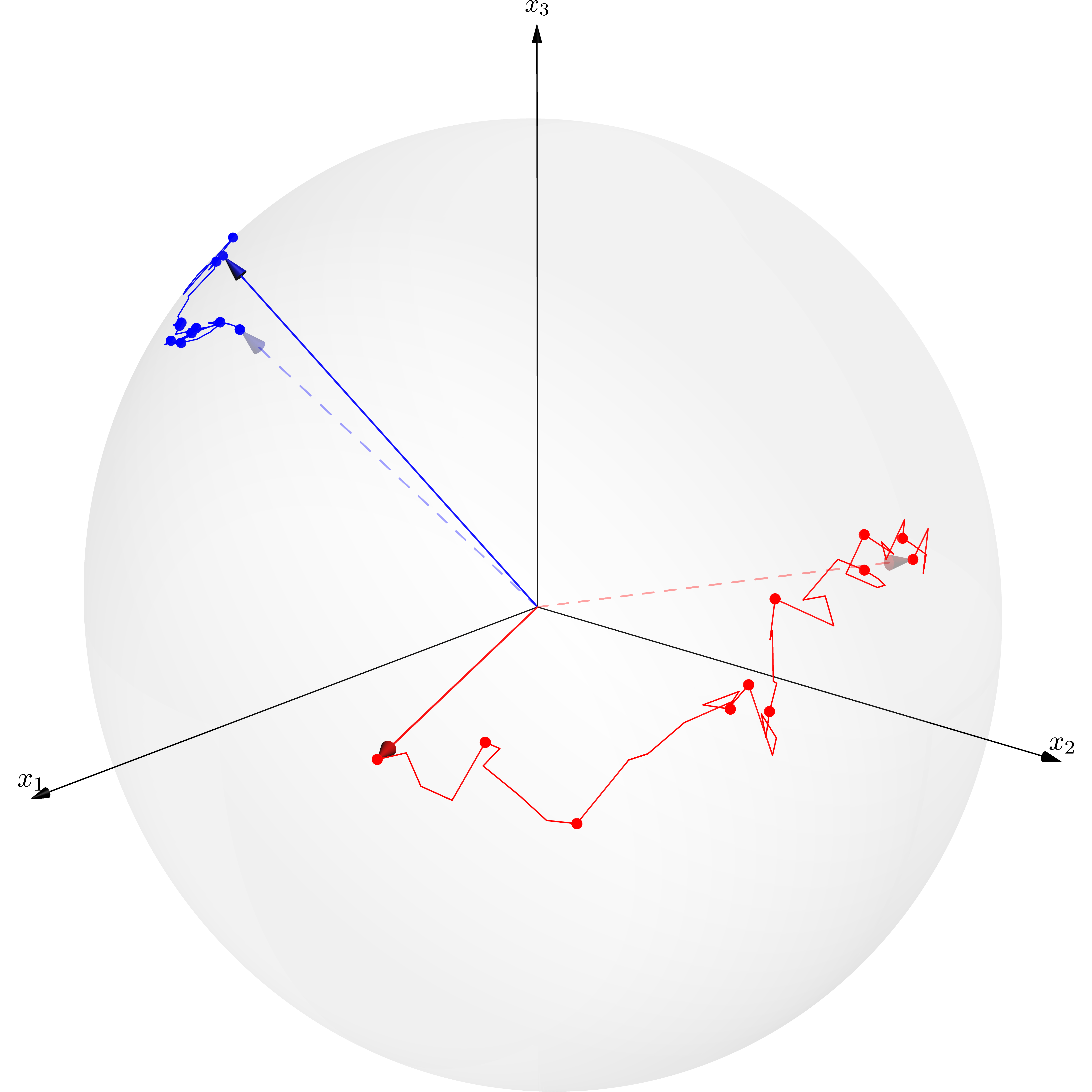}
     \caption{3rd spin}
   \end{subfigure}
\begin{subfigure}[b]{0.21\textwidth} 
 \includegraphics[width=0.99\textwidth,height=3.5cm]{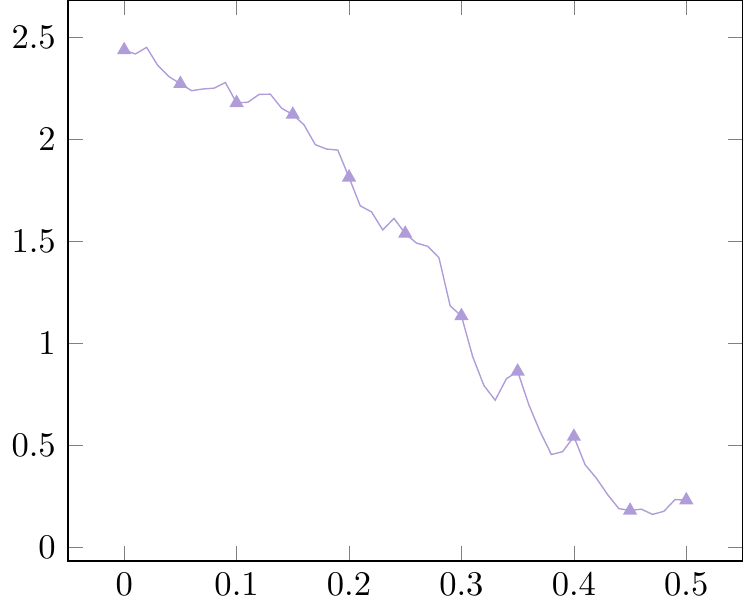}
\caption{$t\mapsto \|u_3^*(t)\|_{\R^3}^2$}
 \end{subfigure} 
      \begin{subfigure}[b]{0.21\textwidth} 
      \includegraphics[width=0.99\textwidth,height=3.5cm]{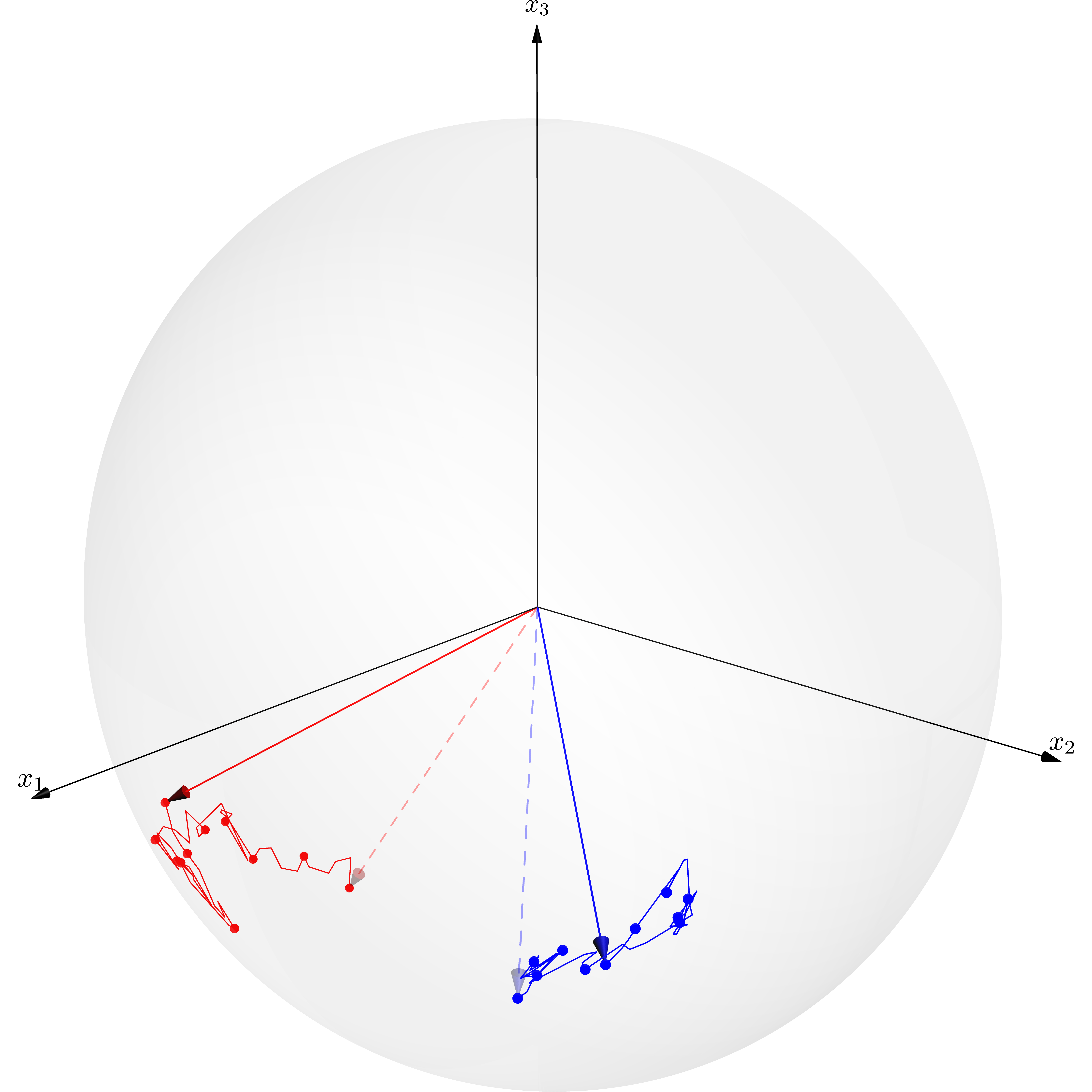}
     \caption{7th spin}
  \end{subfigure}
 \begin{subfigure}[b]{0.21\textwidth} 
  \includegraphics[width=0.99\textwidth,height=3.5cm]{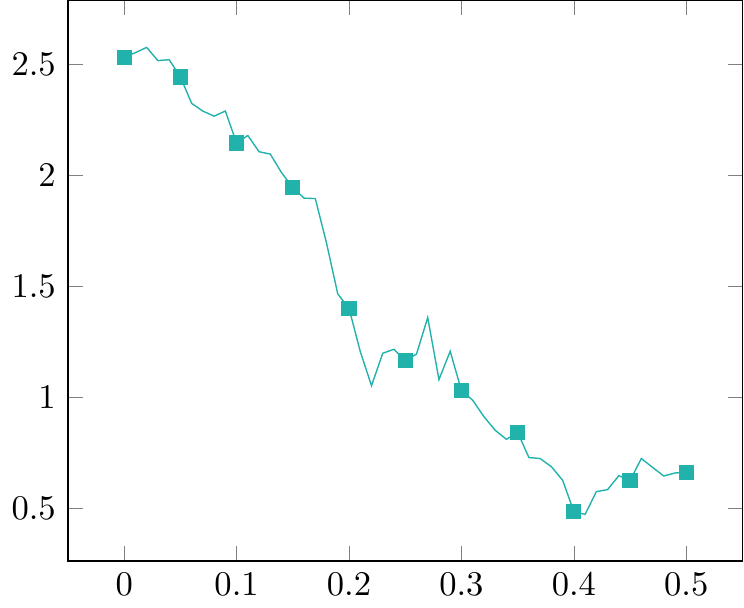}
          \caption{$t\mapsto \|u_7^*(t)\|_{\R^3}^2$}
           \end{subfigure} 
    \caption{ A single realization of the optimal state $m^*$~(red) at final time $T$, the time evolution of the optimal state $t\mapsto m^*_i(t)$~(red), the direction of the optimal control $t\mapsto u_i^*(t)\|u_i^*(t)\|_{\R^3}^{-1}$~(blue), and the magnitude of the optimal control $t\mapsto \|u_i^*(t)\|_{\R^3}^{2}$ for $i=3,7$.}
    \label{fig:state-control-10-spin}
    \end{figure}

\end{document}